\pdfoutput=1 
\DeclareSymbolFont{AMSb}{U}{msb}{m}{n}
\documentclass[11pt,noamsfonts,a4paper]{amsart}
\usepackage[left=1in, right=1in, top=1in, bottom=1in]{geometry}
\usepackage{mathtools}
\usepackage{braket}
\usepackage[charter]{mathdesign}
\usepackage[tracking]{microtype}
\usepackage{tabularx}
\usepackage{etoolbox}
\usepackage{graphicx}
\usepackage{stmaryrd} 

\usepackage{comment}

\usepackage{amsmath,amsthm}
\newtheoremstyle{pineapple}%
  {1em}{1em}%
  {\itshape}{}%
  {\bfseries}{. ---}
  {0.5em}{}

\newtheoremstyle{durian}%
  {1em}{1em}%
  {}{}%
  {\bfseries}{. ---}
  {0.5em}{}

\makeatletter
\def\swappedhead#1#2#3{%
  \thmnumber{\@upn{\the\thm@headfont#2\@ifnotempty{#1}{.~}}}%
  \thmname{#1}%
  \thmnote{ {\the\thm@notefont(#3)}}}
\makeatother

\makeatletter
\newcommand*\rel@kern[1]{\kern#1\dimexpr\macc@kerna}
\newcommand*\widebar[1]{%
  \begingroup
  \def\mathaccent##1##2{%
    \rel@kern{0.8}%
    \overline{\rel@kern{-0.8}\macc@nucleus\rel@kern{0.2}}%
    \rel@kern{-0.2}%
  }%
  \macc@depth\@ne
  \let\math@bgroup\@empty \let\math@egroup\macc@set@skewchar
  \mathsurround\z@ \frozen@everymath{\mathgroup\macc@group\relax}%
  \macc@set@skewchar\relax
  \let\mathaccentV\macc@nested@a
  \macc@nested@a\relax111{#1}%
  \endgroup
}
\makeatother

\makeatletter
\def\@sect#1#2#3#4#5#6[#7]#8{%
  \edef\@toclevel{\ifnum#2=\@m 0\else\number#2\fi}%
  \ifnum #2>\c@secnumdepth \let\@secnumber\@empty
  \else \@xp\let\@xp\@secnumber\csname the#1\endcsname\fi
  \@tempskipa #5\relax
  \ifnum #2>\c@secnumdepth
    \let\@svsec\@empty
  \else
    \refstepcounter{#1}%
    \edef\@secnumpunct{%
      \ifdim\@tempskipa>\z@ 
        \@ifnotempty{#8}{.~}%
      \else
        \@ifempty{#8}{.}{.~}%
      \fi
    }%
    \@ifempty{#8}{%
      \ifnum #2=\tw@ \def\@secnumfont{\bfseries}\fi}{}%
    \protected@edef\@svsec{%
      \ifnum#2<\@m
        \@ifundefined{#1name}{}{%
          \ignorespaces\csname #1name\endcsname\space
        }%
      \fi
      \@seccntformat{#1}%
    }%
  \fi
  \ifdim \@tempskipa>\z@ 
    \begingroup #6\relax
    \@hangfrom{\hskip #3\relax\@svsec}{\interlinepenalty\@M #8\par}%
    \endgroup
    \ifnum#2>\@m \else \@tocwrite{#1}{#8}\fi
  \else
  \def\@svsechd{#6\hskip #3\@svsec
    \@ifnotempty{#8}{\ignorespaces#8\unskip
       \@addpunct.}%
    \ifnum#2>\@m \else \@tocwrite{#1}{#8}\fi
  }%
  \fi
  \global\@nobreaktrue
  \@xsect{#5}}
\makeatother

\makeatletter
\def\@seccntformat#1{%
  \protect\textup{\protect\@secnumfont
    \ifnum\pdfstrcmp{subsection}{#1}=0 \bfseries\fi
    \csname the#1\endcsname
    \protect\@secnumpunct
  }%
}
\makeatother

\theoremstyle{pineapple}
\newtheorem{IntroTheorem}{Theorem}
\newtheorem*{IntroTheorem*}{Theorem}

\swapnumbers
\newtheorem{Theorem}[subsection]{Theorem}
\newtheorem{Lemma}[subsection]{Lemma}
\newtheorem{Proposition}[subsection]{Proposition}
\newtheorem{Corollary}[subsection]{Corollary}

\theoremstyle{durian}

\usepackage[dvipsnames]{xcolor}
\usepackage{hyperref} 
\hypersetup{
  colorlinks = true,
  linkcolor = Fuchsia,
  urlcolor = Fuchsia,
  citecolor = ForestGreen,
  linkbordercolor = {white}}
\usepackage{tikz-cd}
\usetikzlibrary{arrows}

\tikzset{
  symbol/.style={
    draw=none,
    every to/.append style={
      edge node={node [sloped, allow upside down, auto=false]{$#1$}}}
  }
}

\usepackage{enumitem}
\setlist[1]{labelindent=\parindent}
\setlist[1]{labelsep=0.5em}
\setlist[enumerate,1]{label={\upshape (\roman*)}, ref={\upshape (\roman*)}}

\makeatletter
\newcommand{\leqnomode}{\tagsleft@true\let\veqno\@@leqno}
\newcommand{\reqnomode}{\tagsleft@false\let\veqno\@@eqno}
\makeatother

\usepackage[utf8]{inputenc}

\tikzset{>={Straight Barb[length=2pt,width=4pt]}, commutative diagrams/arrow style=tikz}

\makeatletter
\let\c@equation\c@subsection

\makeatother

\DeclareMathOperator{\ch}{ch}

\DeclareMathOperator{\gr}{gr}

\DeclareMathOperator{\Fr}{Fr}
\DeclareMathOperator{\id}{id}
\DeclareMathOperator{\Spec}{Spec}
\DeclareMathOperator{\Sym}{Sym}

\DeclareMathOperator{\GL}{\mathbf{GL}}
\DeclareMathOperator{\AutSch}{\mathbf{Aut}}
\DeclareMathOperator{\HomSch}{\mathbf{Hom}}
\DeclareMathOperator{\image}{im}
\DeclareMathOperator{\HomSheaf}{\mathcal{H}\!\mathit{om}}
\DeclareMathOperator{\proj}{proj}
\DeclareMathOperator{\Div}{Div}
\DeclareMathOperator{\Fil}{Fil}
\DeclareMathOperator{\colim}{colim}
\DeclareMathOperator{\Ext}{Ext}
\DeclareMathOperator{\pr}{pr}
\DeclareMathOperator{\Pic}{Pic}

\makeatletter
\newcommand*{\coloneqq}{\mathrel{\rlap{%
           \raisebox{0.3ex}{$\m@th\cdot$}}%
           \raisebox{-0.3ex}{$\m@th\cdot$}}%
           =}
\newcommand{\eqqcolon}{=%
           \mathrel{\rlap{%
           \raisebox{0.3ex}{$\m@th\cdot$}}%
           \raisebox{-0.3ex}{$\m@th\cdot$}}}
\makeatother

\newcommand{\punct}[1]{\makebox[0pt][l]{\,#1}} 
\newcommand{\parref}[1]{{\bf\ref{#1}}}

\DeclareMathOperator{\Gram}{Gram}
\DeclareMathOperator{\rank}{rank}

\DeclareMathOperator{\Hom}{Hom}

\DeclareMathOperator{\coker}{coker}

\DeclareMathOperator{\Sing}{Sing}

\DeclareMathOperator{\Flag}{\mathbf{Flag}}

\newcommand{\kk}{\mathbf{k}}
\newcommand{\FF}{\mathbf{F}}

\newcommand{\sO}{\mathcal{O}}

\newcommand{\hrefSP}[1]{\href{https://stacks.math.columbia.edu/tag/#1}{#1}}
\newcommand{\citeSP}[1]{\cite[\textbf{\hrefSP{#1}}]{stacks-project}}
\newcommand{\citeForms}[1]{\cite[\href{https://arxiv.org/pdf/2301.09929.pdf\#subsection.#1}{\textbf{#1}}]{qbic-forms}}
\newcommand{\citeFano}[1]{\cite[\href{https://arxiv.org/pdf/2307.06160.pdf\#subsection.#1}{\textbf{#1}}]{fano-schemes}}

\newcommand{\citeThesis}[1]{\cite[\href{https://arxiv.org/pdf/2205.05273.pdf\#subsection.#1}{\textbf{#1}}]{thesis}}

\makeatletter
\newcommand{\smallbullet}{} 
\DeclareRobustCommand\smallbullet{%
  \mathord{\mathpalette\smallbullet@{0.75}}%
}
\newcommand{\smallbullet@}[2]{%
  \vcenter{\hbox{\scalebox{#2}{$\m@th#1\bullet$}}}%
}
\makeatother

\newcommand{\subsectiondash}[1]{\subsection{#1}\textbf{---}\;}
\newcommand{\PP}{\mathbf{P}}

\title{\(q\)-bic threefolds and their surface of lines}
\author{Raymond Cheng}
\address{Institute of Algebraic Geometry \\
  Leibniz University Hannover \\
  Welfengarten 1 \\
  30167 Hannover \\
  Germany
}
\email{cheng@math.uni-hannover.de}

\begin{document}
\begin{abstract}
For any power \(q\) of the positive ground field characteristic, a smooth
\(q\)-bic threefold---the Fermat threefold of degree \(q+1\) for example---has
a smooth surface \(S\) of lines which behaves like the Fano surface of a smooth
cubic threefold. I develop projective, moduli-theoretic, and degeneration
techniques to study the geometry of \(S\). Using, in addition, the modular
representation theory of the finite unitary group and the geometric theory of
filtrations, I compute cohomology of the structure sheaf of \(S\) when \(q\) is
prime.
\end{abstract}
\maketitle
\setcounter{tocdepth}{1}

\section*{Introduction}
The Fermat threefold of degree \(q+1\) is the perhaps the most familiar example
of a \emph{\(q\)-bic threefold}, that is, a hypersurface in projective
\(4\)-space of the form
\[
X \coloneqq
\Set{(x_0:x_1:x_2:x_3:x_4) \in \PP^4 :
\sum\nolimits_{i, j = 0}^4 a_{ij} x_i^q x_j = 0}
\]
where \(q\) is a power of the ground field characteristic \(p > 0\). Such
hypersurfaces have long been of recurring interest due to, for example:
idiosyncrasies in their differential projective geometry \cite{Wallace:Duality,
Hefez:Thesis, Beauville:Moduli, KP:Gauss, Noma}; their abundance of algebraic
cycles and relation to supersingularity \cite{Weil:Fermat, Tate:Conjecture,
HM:TT-Lemma, SK:Fermat, Shimada:Lattices}; and their unirationality and
geometry of rational curves \cite{Shioda:Fermat, Conduche, BDENY:Fermat,
Shen:Fermat}. They classically arise in relation to finite Hermitian geometries
\cite{BC:Hermitian, Segre:Hermitian, Hirschfeld:Geometries, Hirschfeld:Three};
the study of unitary Shimura varieties \cite{Vollaard, LTXZZ, LZ:Kudla}; and
Deligne--Lusztig theory for finite unitary groups \cite{Lusztig:Green, DL,
Hansen:DL, Li:DL}. Such hypersurfaces are distinguished in \cite{KKPSSVW} from
the point of view of singularity theory as those with the lowest possible
\(F\)-pure threshold. See \cite[pp. 7--11]{thesis} for a more extensive
survey.

The purpose of this work is to further develop a striking analogy between the
geometry of lines in \(q\)-bic and cubic hypersurfaces, and also to suggest
flexible framework with which to understand some of the resulting phenomena. To
explain, let \(\kk\) be, here and throughout, an algebraically closed field of
characteristic \(p > 0\), and \(q\) a fixed power of \(p\). Let \(X\) be a
smooth \(q\)-bic threefold over \(\kk\), and write \(S\) for its Fano scheme of
lines. The main results of \cite{fano-schemes}, specialized to dimension \(3\),
are:

\begin{IntroTheorem*}
The scheme \(S\) of lines of a smooth \(q\)-bic threefold \(X\) is an
irreducible, smooth, projective surface of general type. The Fano
correspondence \(S \leftarrow \mathbf{L} \rightarrow X\) induces purely
inseparable isogenies
\[
\mathbf{Alb}_S \xrightarrow{\mathbf{L}_*}
\mathbf{Ab}_X^2 \xrightarrow{\mathbf{L}^*}
\mathbf{Pic}_{S,\mathrm{red}}^0
\]
amongst supersingular abelian varieties of dimension \(\frac{1}{2}q(q-1)(q^2+1)\).
\end{IntroTheorem*}

Here, \(\mathbf{Ab}_X^2\) is the \emph{intermediate Jacobian} of \(X\), taken
to be the algebraic representative for algebraically trivial
\(1\)-cycles in \(X\), in the sense of Samuel and Murre, see \cite{Samuel,
Beauville:Prym, Murre:Jacobian}; see \citeFano{6.9} for a resum\'e. The
statement is analogous to the classical result \cite[Theorem 11.19]{CG} of
Clemens and Griffiths, which states that the Hodge-theoretic intermediate
Jacobian of a complex cubic threefold is isomorphic to the Albanese variety of
its Fano surface of lines. The geometry underlying the proofs also shares many
analogies; compare especially with \cite[Chapter 5]{Huybrechts:Cubics}.
Note that when \(q = 2\), \(X\) is a cubic threefold in characteristic \(2\),
and this gives a version of the theorem of Clemens and Griffiths which is, in
fact, new.

The Fano scheme \(S\) is the primary object of study in this work, and is
henceforth referred to as the \emph{Fano surface} of lines of \(X\). To begin,
refining the general techniques developed in \cite{fano-schemes} gives explicit
computations of some basic invariants of \(S\). In the following statement,
\(\mathcal{S}\) denotes the tautological rank \(2\) subbundle, and \(\sO_S(1)\)
is the Pl\"ucker line bundle on \(S\).

\begin{IntroTheorem}\label{intro-properties}
The tangent bundle \(\mathcal{T}_S\) of the Fano surface \(S\) of a smooth
\(q\)-bic threefold \(X\) is isomorphic to
\(\mathcal{S} \otimes \sO_S(2-q)\). Basic numerical invariants of \(S\)
are
\[
\begin{gathered}
c_1(S)^2 = (q+1)^2(q^2+1)(2q-3)^2,\;\;\;\;
c_2(S) = (q+1)^2(q^4 - 3q^3 + 4q^2 - 4q + 3),\;\text{and} \\
\chi(S,\sO_S) = \frac{1}{12}(q+1)^2(5q^4 - 15q^3 + 17q^2 - 16q + 12).
\end{gathered}
\]
 If \(q > 2\), then \(S\) lifts neither to the second Witt vectors nor to
 characteristic \(0\).
\end{IntroTheorem}

This is an amalgamation of \parref{generalities-chern-numbers},
\parref{generalities-nonliftable}, and \parref{generalities-chi-OS}. The
identification of \(\mathcal{T}_S\) is an analogue of the tangent bundle
theorem \cite[Proposition 12.31]{CG} of Clemens and Griffiths, see also
\cite[Theorem 1.10]{AK:Fano}. That \(S\) does not lift to the Witt vectors is
because it violates Kodaira--Akizuki--Nakano vanishing, and not to
characteristic \(0\) because it violates the Bogomolov--Miyaoka--Yau inequality
\(c_1(S)^2 \leq 3c_2(S)\). As such, \(S\) is a purely positive characteristic
phenomenon, and the analogy with cubics serves as a valuable guide in a setting
where classical geometric intuitions are less potent.

Non-liftability also means that few vanishing theorems are available,
presenting a major difficulty in coherent cohomology computations. Nonetheless,
when \(q\) is the prime \(p\) itself, cohomology of the structure sheaf
\(\sO_S\) behaves as well as possible:

\begin{IntroTheorem}\label{intro-smooth-cohomology}
Let \(X\) be a smooth \(q\)-bic threefold and \(S\) its Fano surface.
If \(q = p\), then
\begin{align*}
\dim_\kk\mathrm{H}^1(S,\sO_S) & =
\frac{1}{2}p(p-1)(p^2+1) =
\frac{1}{2}\dim_{\mathbf{Q}_\ell}\mathrm{H}^1_{\mathrm{\acute{e}t}}(S,\mathbf{Q}_\ell), \;\text{and} \\
\dim_\kk\mathrm{H}^2(S,\sO_S) & =
\frac{1}{12}p(p-1)(5p^4 - 2p^2 - 5p - 2).
\end{align*}
In particular, the Picard scheme of \(S\) is smooth.
\end{IntroTheorem}

First cohomology is computed in \parref{smooth-proof}, at which point second
cohomology is determined by the Euler characteristic computation above. This
result, together with the methods developed, provides a first step towards
understanding the precise relationship amongst the abelian varieties
\(\mathbf{Alb}_S\), \(\mathbf{Ab}_X^2\), and
\(\mathbf{Pic}^0_{S,\mathrm{red}} = \mathbf{Pic}^0_S\) appearing in the
analogue of the Clemens--Griffiths Theorem quoted above.

The method involves a delicate degeneration argument. General considerations show
that the dimension of \(\mathrm{H}^1(S,\sO_S)\) is at least half the first
Betti number. A corresponding upper bound might be obtained via upper
semicontinuity as follows: Specialize \(S\) to a singular surface
\(S_0\) by specializing \(X\) to a \(q\)-bic threefold \(X_0\) with the mildest
possible singularities; precisely, in terms of the classification of
\cite[\href{https://arxiv.org/pdf/2301.09929.pdf\#IntroTheorem.1}{\textbf{Theorem A}}]{qbic-forms},
\(X_0\) is of type \(\mathbf{1}^{\oplus 3} \oplus \mathbf{N}_2\), and there
is a choice of coordinates such that
\[
X_0 =
\Set{(x_0:x_1:x_2:x_3:x_4) \in \PP^4 : x_0^{q+1} + x_1^{q+1} + x_2^{q+1} + x_3^qx_4 = 0}.
\]
Although \(S_0\) is quite singular, its normalization is quite manageable. The
result is as follows:

\begin{IntroTheorem}\label{intro-nodal}
The Fano surface \(S_0\) of a \(q\)-bic threefold \(X_0\) of type
\(\mathbf{1}^{\oplus 3} \oplus \mathbf{N}_2\) is
nonnormal, and its normalization \(S_0^\nu\) is a projective bundle over
a smooth \(q\)-bic curve \(C\). There is a commutative diagram
\[
\begin{tikzcd}
S_0^\nu \rar["\nu"'] \dar["\tilde{\varphi}_+"'] & S_0 \dar["\varphi_-"] \\
C \rar["\phi_C"] & C
\end{tikzcd}
\]
where \(\phi_C \colon C \to C\) is conjugate to the \(q^2\)-power
geometric Frobenius morphism of \(C\). If \(q = p\), then
\begin{align*}
\dim_\kk \mathrm{H}^1(S_0,\sO_{S_0}) & =
\frac{1}{2}p(p-1)(p^2+1) + \frac{1}{6}p(p-1)(p-2),\;\;\text{and}\\
\dim_\kk \mathrm{H}^2(S_0,\sO_{S_0}) & =
\frac{1}{12}p(p-1)(5p^4 - 2p^2 - 5p - 2) + \frac{1}{6}p(p-1)(p-2).
\end{align*}
\end{IntroTheorem}

This collects \parref{nodal-nu-and-F}, \parref{normalize-cohomology}, and
\parref{cohomology-dimension-result}. See also
\parref{cohomology-F-restriction-remarks} for remarks about extending the
calculation to general \(q\). The normalization \(S_0^\nu\) is
constructed, roughly, as the space of lines in \(X_0\) that project to a
tangent of \(C\). Cohomology of \(\sO_{S_0}\) is determined by using the
diagram to relate it to a vector bundle computation \(C\); using group schemes
\(\mathbf{G}_m\) and \(\boldsymbol{\alpha}_p\) in the automorphism group of
\(S_0\) to reduce to a manageable set of computations; and using modular
representation theory of the finite unitary group \(\mathrm{U}_3(p)\),
essentially the automorphism group of \(C\), to make explicit identifications.

Alarmingly, upon comparing Theorems \parref{intro-smooth-cohomology} and
\parref{intro-nodal}, one finds that the cohomology groups of the special fibre
\(S_0\) are much larger than one would hope. Upper semicontinuity therefore
only shows
\[
\frac{1}{2}p(p-1)(p^2+1) \leq
\dim_\kk\mathrm{H}^1(S,\sO_S) \leq
\frac{1}{2}p(p-1)(p^2+1) + \frac{1}{6}p(p-1)(p-2).
\]
The upper bound can be refined by carefully analysing a specially chosen
degeneration \(S \rightsquigarrow S_0\). Specifically, there exists a flat
family \(\mathfrak{S} \to \mathbf{A}^1\) in which all fibres away from the
central fibre \(S_0\) are isomorphic to \(S\), and which carries a
\(\mathbf{G}_m\) action compatible with the weight \(q^2-1\) homothety on the
base \(\mathbf{A}^1\). The geometric theory of filtrations---the
correspondence between filtrations and \(\mathbf{G}_m\)-equivariant objects
over \(\mathbf{A}^1\) as pioneered in Simpson's work \cite{Simpson} on
nonabelian Hodge theory---endows \(\mathrm{H}^1(S,\sO_S)\) with a filtration
whose graded pieces are various graded pieces of
\(\mathrm{H}^1(S_0,\sO_{S_0})\). A careful analysis identifies
\(\frac{1}{6}p(p-1)(p-2)\) cohomology classes on \(S_0\) that do \emph{not}
lift to classes on \(S\). With this, the upper semicontinuity bound becomes
tight, completing the calculation. This use of the geometric theory of
filtrations in determining cohomology appears to be new, and I expect that the
method can be generalized and refined to be applied in other settings.

Crucial to the degeneration argument is the construction of a fibration
\(\varphi \colon S \to C\) on the smooth Fano surface that specializes to
\(\varphi_- \colon S_0 \to C\) as in Theorem \parref{intro-nodal}. This is done
using a general projective geometry method, which may be adapted to other
settings, exploiting the presence of \emph{cone points}: points at which
the tangent space intersects the hypersurface \(X\) at a cone over a curve
\(C\); these generalize \emph{Eckardt points} of cubic hypersurfaces, and are
sometimes also called \emph{star points}. A combination of projection and
intersection with a hyperplane induces a rational map \(S \dashrightarrow C\),
which can be resolved in this setting using global methods made possible by the
intrinsic theory of \(q\)-bic forms developed in \cite{qbic-forms}. A summary
of the geometric situation is as follows:

\begin{IntroTheorem}\label{intro-cone-situation}
Let \(X\) be a \(q\)-bic threefold either smooth or of type
\(\mathbf{1}^{\oplus 3} \oplus \mathbf{N}_2\), and \(S\) its Fano surface. A
choice of cone point in \(X\) over a smooth \(q\)-bic curve \(C\) induces a
canonical diagram
\[
\begin{tikzcd}[column sep=1em, row sep=1em]
& \tilde{S} \ar[dr,"\rho"] \ar[dl,"b"'] \\
S \ar[dr,"\varphi"'] && T \ar[dl,"\pi"] \\
& C
\end{tikzcd}
\]
where \(\varphi_*\sO_S \cong \pi_*\sO_T \cong \sO_C\),
\(b \colon \tilde{S} \to S\) is a blowup along \(q^3+1\) smooth points, and
\(\rho \colon \tilde{S} \to T\) is a quotient by either \(\mathbf{F}_q\) or
\(\boldsymbol{\alpha}_q\) depending on whether \(X\) is smooth or singular.
Furthermore, \(T\) is an explicit rank \(1\) degeneracy locus in a product of
projective bundles over \(C\).
\end{IntroTheorem}

This is a summary of \parref{cone-situation-rational-map-S},
\parref{cone-situation-equations-of-T}, \parref{smooth-cone-situation-G},
\parref{smooth-cone-situation-quotient}, \parref{smooth-cone-situation-blowup},
\parref{smooth-cone-situation-rational-maps}, and
\parref{smooth-cone-situation-pushforward}. The techniques developed here,
and more generally in
\cite[\href{https://arxiv.org/pdf/2205.05273.pdf\#appendix.A}{\textbf{Appendix A}}]{thesis},
may be adapted to study linear spaces in other projective varieties. For
instance, projection from a cone point maps the scheme of \(r\)-planes in an \(n\)-fold
to the scheme of \((r-1)\)-planes in an \((n-2)\)-fold, providing an inductive
method to study Fano schemes. This method gives a geometric construction of
some classically known structures, such as the fibration constructed in
\cite[\S3.2]{Roulleau:Elliptic}.

To close the Introduction, I would like to sketch a perspective with which to
contextualize the results at hand. As the prime power \(q\) varies, \(q\)-bic
hypersurfaces are of different degrees, and are even defined over different
characteristics. Yet in many respects, they ought to be viewed as constituting
a single family. One way to make sense of this might be to say that they
are constructed ``uniformly with respect to \(q\).'' This may be formulated in
terms of the shape of the defining equations, but perhaps a more flexible and
geometric way is to realize a \(q\)-bic hypersurface \(X\) as the intersection
\[
X = \Gamma_q \cap Z \subset \PP^n \times_\kk \PP^n
\]
between the graph \(\Gamma_q\) of the \(q\)-power Frobenius morphism, and a
\((1,1)\)-divisor \(Z\). The Fano surface \(S\) also fits into this point of
view: take the ambient variety to be the Grassmannian \(\mathbf{G}\), and take
\(Z\) to be the zero locus in \(\mathbf{G} \times_\kk \mathbf{G}\) of a general
section of the vector bundle \(\mathcal{S}^\vee \boxtimes \mathcal{S}^\vee\).

This point of view highlights the prime power \(q\) and the scheme \(Z\) as the
parameters of construction. What should it mean that varying \(q\) yields
schemes of the same family? Consider the more familiar matter of fixing \(q\)
and varying \(Z\): Traditionally, those deformations of \(X\) or \(S\) that are
of the same family are those obtained by \emph{flatly} varying \(Z\), and this
is justified in part by constancy of geometric invariants. When varying \(q\),
it is generally too much to ask for invariants to remain constant, but one
could hope that invariants vary in a simple way with respect to \(q\). One
particularly pleasant, albeit optimistic, condition would be to ask for Euler
characteristics or, stronger, for dimensions of cohomology groups to vary as a
polynomial in \(q\). This is true of \(q\)-bic hypersurfaces, and Theorems
\parref{intro-smooth-cohomology} and \parref{intro-nodal} may be understood as
saying the stronger condition of constancy of cohomological invariants holds
also for the Fano surfaces of \(q\)-bic threefolds, at least when \(q\) varies
only over primes \(p\). I believe it is an interesting problem to find more
examples of this type of phenomenon.

Finally, observe that if \(\PP^n\) and \(\mathbf{G}\) were to be replaced by
more general ambient varieties \(\mathrm{P}\), and \(Z\) allowed to be
subschemes not necessarily cut out by sections of vector bundles, even Euler
characteristics cannot vary as simply as a polynomial in \(q\): indeed, by
taking \(Z\) to be the diagonal, this would count points of \(\mathrm{P}\) over
finite fields. However, the Weil conjectures show that, even then, there is
structure amongst the numbers and, taken together, the give insight into the
geometry of \(\mathrm{P}\). My hope is that, by circumscribing the class of
ambient schemes \(\mathrm{P}\) and subschemes \(Z\), this line of study may
lead to new and interesting structures in positive characteristic geometry.

\medskip
\noindent\textbf{Outline. ---}
\S\parref{section-generalities} begins with a summary of the theory of
\(q\)-bic forms and hypersurfaces, as developed in \cite{qbic-forms,
fano-schemes}; the latter half refines these methods in order to establish
Theorem \parref{intro-properties}.
\S\S\parref{section-cone-situation}--\parref{section-smooth-cone-situation}
constructs and studies a rational map \(S \dashrightarrow C\) from the Fano
surface to a \(q\)-bic curve, establishing Theorem
\parref{intro-cone-situation}. These methods are applied to study \(q\)-bic
threefolds of type \(\mathbf{1}^{\oplus 3} \oplus \mathbf{N}_2\) in
\S\parref{section-nodal}.
\S\S\parref{section-D}--\parref{section-cohomology-F} are where most of
the cohomology computation for singular \(q\)-bics takes place, and completes
the proof of Theorem \parref{intro-nodal}. Finally, \S\parref{section-smooth}
puts everything together to prove Theorem \parref{intro-smooth-cohomology}.

\medskip
\noindent\textbf{Acknowledgements. ---}
This article is based on parts of Chapter 4 of my thesis \cite{thesis}.
Sincerest gratitude goes to Aise Johan de Jong for many conversations around
half-ideas and failed attempts that went into this work. Thanks to Matthias
Sch\"utt for comments on early drafts. During the completion of this work, I
was supported by a Humboldt Research Fellowship.

\section{\texorpdfstring{\(q\)}{q}-bic hypersurfaces}\label{section-generalities}
The first half of this section summarizes some of the theory and results
developed in \cite{qbic-forms, fano-schemes} regarding \(q\)-bic forms and
\(q\)-bic hypersurfaces that will be freely used in this article. The second
half of this section, starting in \parref{generalities-threefolds}, develops
a few basic results pertaining to the Fano surface of a \(q\)-bic threefold.
Throughout this article, denote by \(\kk\) an algebraically closed field of
characteristic \(p > 0\), \(q \coloneqq p^e\) for some \(e \in \mathbf{Z}_{>
0}\), and \(V\) a \(\kk\)-vector space of dimension \(n+1\).

\subsectiondash{\texorpdfstring{\(q\)}{q}-bic forms and hypersurfaces}
\label{generalities-qbic}
A \emph{\(q\)-bic form} on  \(V\) is a nonzero linear map
\[
\beta \colon V^{[1]} \otimes_{\kk} V \to \kk,
\]
where here and elsewhere, \(V^{[1]} \coloneqq \kk \otimes_{\Fr,\kk} V\) is the
\(q\)-power Frobenius twist of \(V\). Let \(\PP V\) be the projective space of
lines associated with \(V\). The \emph{\(q\)-bic equation} associated with
\(\beta\) is the map
\[
f_\beta \coloneqq \beta(\mathrm{eu}^{[1]}, \mathrm{eu}) \colon
\sO_{\PP V}(-q-1) \to
V^{[1]} \otimes_\kk V \otimes_\kk \sO_{\PP V} \to
\sO_{\PP V}
\]
obtained by pairing the canonical section
\(\mathrm{eu} \colon \sO_{\PP V}(-1) \to V \otimes_\kk \sO_{\PP V}\)
with its \(q\)-power via \(\beta\). The \emph{\(q\)-bic hypersurface}
associated with the \(q\)-bic form \(\beta\) is the hypersurface defined by
\(f_\beta\), and may be identified as the space of lines in \(V\) isotropic for
\(\beta\):
\[
X \coloneqq
X_\beta \coloneqq
\mathrm{V}(f_\beta) =
\set{[v] \in \PP V : \beta(v^{[1]}, v) = 0}.
\]
A choice of basis \(V = \langle e_0, \ldots, e_n \rangle\) makes this more
explicit: Writing \(\mathbf{x}^\vee \coloneqq (x_0: \cdots: x_n)\) for the
corresponding projective coordinates on \(\PP V \cong \PP^n\), and
\(a_{ij} \coloneqq \beta(e_i^{[1]}, e_j)\) for the \((i,j)\)-th entry of
the Gram matrix of \(\beta\) with respect to the given basis, as in
\citeForms{1.2}, the \(q\)-bic equation \(f_\beta\) is the homogeneous
polynomial of degree \(q+1\) given by
\[
f_\beta(x_0,\ldots,x_n) \coloneqq
\mathbf{x}^{[1],\vee} \cdot
\operatorname{Gram}(\beta; e_0,\ldots,e_n) \cdot
\mathbf{x} =
\sum\nolimits_{i, j = 0}^n a_{ij} x_i^q x_j.
\]
From any of the descriptions, it is easy to check that a hyperplane section of
a \(q\)-bic hypersurface is a \(q\)-bic hypersurface: see \citeFano{1.9}.

To illustrate the utility of distinguishing the shape of the \(q\)-bic equation
\(f \coloneqq f_\beta\), and to record the computation for later use, the
following shows that the action of the \(q\)-power absolute Frobenius morphism
\(\Fr \colon X \to X\) is zero on higher cohomology:

\begin{Lemma}\label{generalities-frobenius-vanishes}
If \(\dim X \geq 1\), then the map
\(
\Fr \colon
\mathrm{H}^{n-1}(X,\sO_X) \to
\mathrm{H}^{n-1}(X,\sO_X)
\)
is zero.
\end{Lemma}

\begin{proof}
Consider the commutative diagram of abelian sheaves on \(\PP V\) given by
\[
\begin{tikzcd}
0 \rar
& \sO_{\PP V}(-q-1) \rar["f"'] \dar["f^{q-1} \Fr"]
& \sO_{\PP V} \rar \dar["\Fr"]
& \sO_X \rar \dar["\Fr"]
& 0 \\
0 \rar
& \Fr_*(\sO_{\PP V}(-q-1)) \rar["f^q"]
& \Fr_*\sO_{\PP V} \rar
& \Fr_*\sO_X \rar
& 0
\end{tikzcd}
\]
where \(\Fr\) acts on local sections by \(s \mapsto s^q\). Taking cohomology
yields a commutative diagram
\[
\begin{tikzcd}
\mathrm{H}^{n-1}(X,\sO_X) \rar["\cong"'] \dar["\Fr"']
& \mathrm{H}^n(\PP V,\sO_{\PP V}(-q-1)) \dar["f^{q-1} \Fr"] \\
\mathrm{H}^{n-1}(X,\sO_X) \rar["\cong"]
& \mathrm{H}^n(\PP V,\sO_{\PP V}(-q-1))\punct{.}
\end{tikzcd}
\]
Compute the map \(f^{q-1} \Fr\) on the right as follows: Upon choosing
coordinates \((x_0:\cdots:x_n)\) for \(\PP V = \PP^n\), a basis for
\(\mathrm{H}^n(\PP V, \sO_{\PP V}(-q-1))\) is given by elements
\[
\xi \coloneqq (x_0^{i_0} \cdots x_n^{i_n})^{-1}
\;\;\text{where}\;
i_0 + \cdots + i_n = q+1\;\text{and each}\;
i_j \geq 1.
\]
Monomials appearing in \(f^{q-1}\) are of the form \(a^q b\), where \(a\) and
\(b\) are themselves monomials of degree \(q-1\). Write
\(a = x_0^{a_0} \cdots x_n^{a_n}\) with \(a_0 + \cdots + a_n = q-1\). Then
\(f^{q-1}\Fr(\xi)\) is a sum of terms
\[
a^q b \cdot (x_0^{i_0} \cdots x_n^{i_n})^{-q} =
b \cdot (x_0^{i_0 - a_0} \cdots x_n^{i_n - a_n})^{-q}.
\]
Since \((i_0 - a_0) + \cdots + (i_n - a_n) = 2 \leq n\), there is some index
\(j\) such that \(i_j - a_j \leq 0\), and so this represents \(0\).
Thus each potential contribution to \(f^{q-1}\Fr(\xi)\) vanishes, and so
\(f^{q-1}\Fr(\xi) = 0\).
\end{proof}

\subsectiondash{Classification}\label{generalities-classification}
Isomorphic \(q\)-bic forms yield projectively equivalent \(q\)-bic
hypersurfaces. Over an algebraically closed field \(\kk\),
\cite[\href{https://arxiv.org/pdf/2301.09929.pdf\#IntroTheorem.1}{\textbf{Theorem A}}]{qbic-forms}
shows that there are finitely many isomorphism classes of \(q\)-bic forms
on \(V\), and that they may be classified via their Gram matrices: there
exists a basis \(V = \langle e_0,\ldots,e_n \rangle\) and nonnegative integers
\(a, b_m \in \mathbf{Z}_{\geq 0}\) such that
\[
\Gram(\beta;e_0,\ldots,e_n) =
\mathbf{1}^{\oplus a} \oplus
\Big(\bigoplus\nolimits_{m \geq 1} \mathbf{N}_m^{\oplus b_m}\Big)
\]
where \(\mathbf{1}\) is the \(1\)-by-\(1\) matrix with unique entry \(1\),
\(\mathbf{N}_m\) is the \(m\)-by-\(m\) Jordan block with \(0\)'s on the
diagonal, and \(\oplus\) denotes block diagonal sums of matrices. The tuple
\((a; b_m)_{m \geq 1}\) is called the \emph{type} of \(\beta\), and is the
fundamental invariant of; the sum \(\sum_{m \geq 1} b_m\) is its \emph{corank}.
When \(q\)-bic forms are varied in a family, the corank can jump and the type
can change; the Hasse diagram of specialization relations is not linear in
general, and some partial relations are given in
\cite[\href{https://arxiv.org/pdf/2301.09929.pdf\#IntroTheorem.3}{\textbf{Theorem C}}]{qbic-forms}.

Of note is the \(\mathbf{N}_1^{\oplus b_1}\) piece, which is a
\(b_1\)-by-\(b_1\) zero matrix, and will sometimes denoted by
\(\mathbf{0}^{\oplus b_1}\). Its underlying subspace is the \emph{radical} of
\(\beta\):
\[
\operatorname{rad}(\beta) =
\set{v \in V : \beta(v^{[1]}, u) = \beta(u^{[1]}, v) = 0\;\text{for all}\; u \in V}.
\]
It is straightforward to check that \(X\) is a cone if and only if
\(b_1 \neq 0\), and that the vertex is the projectivization of the radical; see
also \citeThesis{2.4} for an invariant treatment.

\subsectiondash{\texorpdfstring{\(q\)}{q}-bic points and curves}
\label{generalities-classification-lowdim}
In low-dimensions, the classification and specialization relations are quite
simple and, especially when combined with the fact that hyperplane sections of
\(q\)-bics are \(q\)-bics, are geometrically very useful: see
\cite[\href{https://arxiv.org/pdf/2205.05273.pdf\#chapter.3}{\textbf{Chapter 3}}]{thesis}
for more. When \(V\) has dimension \(2\), the possible types of \(q\)-bic forms
and their specialization relations are
\[
\mathbf{1}^{\oplus 2} \rightsquigarrow
\mathbf{N}_2 \rightsquigarrow
\mathbf{0} \oplus \mathbf{1},
\]
corresponding to the subschemes in \(\PP^1\) defined by
\(x_0^{q+1} + x_1^{q+1}\), \(x_0^q x_1\), and \(x_1^{q+1}\). Thus \(q\)-bic
points are either \(q+1\) reduced points, a reduced point with a \(q\)-fold
point, or one point of multiplicity \(q+1\).

When \(V\) is \(3\)-dimensional, the types and specializations excluding cones
over \(q\)-bic points are
\[
\mathbf{1}^{\oplus 3} \rightsquigarrow
\mathbf{1} \oplus \mathbf{N}_2 \rightsquigarrow
\mathbf{N}_3.
\]
The corresponding subschemes of \(\PP^2\) may be described as:
a smooth curve of degree \(q+1\); an irreducible, geometrically rational
curve with a single unibranch singularity; and a reducible curve with a linear
component meeting an irreducible, geometrically rational component of degree
\(q\) at its unique unibranch singularity.

\subsectiondash{Smoothness}\label{generalities-smoothness}
A \(q\)-bic hypersurface \(X\) is smooth if and only if its underlying
\(q\)-bic form \(\beta\) is \emph{nonsingular} in the sense that
one---equivalently, both---of the adjoint maps
\(\beta \colon V \to V^{[1],\vee}\) or \(\beta^\vee \colon V^{[1]} \to V^\vee\)
is an isomorphism; with a choice of basis, this is equivalent to invertibility
of a Gram matrix. The singular locus of \(X\) is supported on the
projectivization of the subspace
\[
V^{\perp, [-1]} \coloneqq
\set{v \in V : \beta(v^{[1]}, u) = 0 \;\text{for every}\; u \in V},
\]
the Frobenius descent of the left kernel
\(V^\perp \coloneqq \ker(\beta^\vee \colon V^{[1]} \to V^\vee)\): see
\citeFano{2.6}. The embedded tangent space \(\mathbf{T}_{X,x} \subset \PP V\)
to \(X\) at a point \(x = \PP L\) is the projectivization of
\[
L^{[1],\perp} \coloneqq
\set{v \in V : \beta(u^{[1]}, v) = 0\;\text{for all}\; u \in L} \coloneqq
\ker(\beta \colon
  V \to
  V^{[1],\vee} \twoheadrightarrow
  L^{[1],\vee}),
\]
the right orthogonal of \(L\): see \citeFano{2.2}.

\subsectiondash{Hermitian structures}\label{generalities-hermitian}
A vector \(v \in V\) is said to be \emph{Hermitian} if
\[
\beta(u^{[1]}, v) = \beta(v^{[1]}, u)^q
\;\;\text{for all}\; u \in V.
\]
A subspace \(U \subset V\) is \emph{Hermitian} if it is spanned by Hermitian
vectors. The Hermitian equation implies that the left and right orthogonals of
\(U\),
\begin{align*}
U^{\perp,[-1]}
& \coloneqq \set{v \in V : \beta(v^{[1]}, u) = 0\;\text{for all}\; u \in U} \;\text{and} \\
U^{[1],\perp}
& \coloneqq \set{v \in V : \beta(u^{[1]}, v) = 0\;\text{for all}\; u \in U},
\end{align*}
coincide. When \(\beta\) is nonsingular, there are only finitely many Hermitian
vectors, whence finitely many Hermitian subspaces, and they span \(V\): see
\cite[%
\href{https://arxiv.org/pdf/2301.09929.pdf\#subsection.2.1}{\textbf{2.1}}--%
\href{https://arxiv.org/pdf/2301.09929.pdf\#subsection.2.6}{\textbf{2.6}}]{qbic-forms}

Continuing with \(\beta\) nonsingular, there is a canonical \(\kk\)-vector
space isomorphism
\[
\sigma_\beta \coloneqq \beta^{-1} \circ \beta^{[1],\vee} \colon
V^{[2]} \to V^{[1],\vee} \to V
\]
between the \(q^2\)-Frobenius twist of \(V\) and \(V\) itself; this is an
\(\mathbf{F}_{q^2}\) descent datum on \(V\). Pre-composing with the
universal \(q^2\)-linear map \(V \to V^{[2]}\) yields a \(q^2\)-linear
endomorphism \(\phi \colon V \to V\). It preserves isotropicity so it induces
an endomorphism \(\phi_X \colon X \to X\). In a basis
\(V = \langle e_0, \ldots, e_n \rangle\) of Hermitian vectors, the Gram
matrix of \(\beta\) is Hermitian in the sense that
\[
\Gram(\beta; e_0,\ldots,e_n)^\vee =
\Gram(\beta; e_0,\ldots,e_n)^{[1]},
\]
from which it follows that \(\sigma_\beta\) is represented by the identity
matrix, and that \(\phi_X\) is the \(q^2\)-power geometric Frobenius
morphism in the corresponding coordinates: see \citeFano{4.3}.

By \citeFano{1.3}, the fixed set of \(\phi_X\) is the subset
\(X_{\mathrm{Herm}}\) of \emph{Hermitian points} of \(X\), consisting of those
points \(x = \PP L\) where \(L\) is spanned by a Hermitian vector. Comparing
with the construction of \(\phi_X\), this implies
that the set of Hermitian points of \(X\) is the zero locus of the map
\[
\tilde\theta \colon
\sO_X(-q^2) \stackrel{\mathrm{eu}^{[2]}}{\longrightarrow}
V^{[2]} \otimes_{\kk} \sO_X \stackrel{\sigma_\beta}{\longrightarrow}
V \otimes_{\kk} \sO_X \longrightarrow
\mathcal{T}_{\PP V}(-1)\rvert_X
\]
where the final arrow arises from the dual Euler sequence. The following shows
that this map factors through the tangent bundle of \(X\). A geometric
description of \(\phi_X(x)\) may be extracted from the proof below, as can be
found in \citeThesis{2.9.9}, and this is most clear when \(X\) is a curve:
the tangent line \(\mathbf{T}_{X,x}\) intersects \(X\) at \(x\) with
multiplicity \(q\), and the residual point of intersection is \(\phi_X(x)\).

\begin{Lemma}\label{generalities-theta}
Hermitian points of \(X\) are cut out by a canonical map
\(\theta \colon \sO_X(-q^2) \to \mathcal{T}_X(-1)\).
\end{Lemma}

\begin{proof}
The tangent bundle of \(X\) is the kernel of the normal map
\(\delta \colon \mathcal{T}_{\PP V}(-1)\rvert_X \to \mathcal{N}_{X/\PP V}(-1)\);
this is induced via the dual Euler sequence by the map
\(\tilde\delta \colon V \otimes_{\kk} \sO_X \to \mathcal{N}_{X/\PP V}(-1)\) which
takes directional derivatives of a fixed equation of \(X\). A simple
computation shows that there is a commutative diagram
\[
\begin{tikzcd}
V \otimes \sO_X \rar["\tilde\delta"'] \dar["\beta"'] &
\mathcal{N}_{X/\PP V}(-1) \\
V^{[1],\vee} \otimes \sO_X \rar["\mathrm{eu}^{[1],\vee}"] &
\sO_X(q) \uar["\cong"']
\end{tikzcd}
\]
and so \(\delta \circ \tilde\theta\) factors through the map
\[
\tilde\delta \circ \sigma_\beta \circ \mathrm{eu}^{[2]} =
\mathrm{eu}^{[1],\vee} \circ \beta^{[1],\vee} \circ \mathrm{eu}^{[2]} =
\beta(\mathrm{eu}^{[1]}, \mathrm{eu})^q = 0,
\]
see \citeForms{1.7} for the pentultimate relation. This implies that
\(\tilde{\theta}\) factors through \(\mathcal{T}_X(-1)\).
\end{proof}

\subsectiondash{Scheme of lines}\label{generalities-lines}
The Fano scheme \(\mathbf{F} \coloneqq \mathbf{F}_1(X)\) of lines of a \(q\)-bic hypersurface \(X\)
may be viewed as the space of \(2\)-dimensional subspaces in \(V\) which are
totally isotropic for the \(q\)-bic form \(\beta\):
\[
\mathbf{F} =
\set{[U] \in \mathbf{G} : \beta(u^{[1]}, u') = 0\;\text{for all}\; u, u' \in U}.
\]
This description exhibits \(\mathbf{F}\) as the zero locus in the Grassmannaian
\(\mathbf{G} \coloneqq \mathbf{G}(2,V)\) of the map
\[
\beta_{\mathcal{S}} \colon
\mathcal{S}^{[1]} \otimes_{\sO_{\mathbf{G}}}
\mathcal{S} \to
\sO_{\mathbf{G}}
\]
induced by restricting \(\beta\) to the universal rank \(2\) subbundle
\(\mathcal{S} \subset V \otimes_\kk \sO_{\mathbf{G}}\). A direct
construction shows that \(\mathbf{F}\) is nonempty whenever \(n \geq 3\),
so this implies that \(\dim\mathbf{F} \geq 2n-6\): see
\cite[\href{https://arxiv.org/pdf/2307.06160.pdf\#subsection.1.10}{\textbf{1.10}}--\href{https://arxiv.org/pdf/2307.06160.pdf\#subsection.1.12}{\textbf{1.12}}]{fano-schemes}.
Furthermore, \(\mathbf{F}\) is connected whenever \(n \geq 4\):
see \citeFano{2.9}.

When \(\mathbf{F}\) is generically smooth and of expected dimension
\(2n-6\), it is a local complete intersection scheme whose structure sheaf
admits a Koszul resolution by
\(\wedge^* \mathcal{S}^{[1]} \otimes_{\sO_{\mathbf{G}}} \mathcal{S}\), and its
dualizing sheaf may be computed to be
\[
\omega_{\mathbf{F}} \cong
\sO_{\mathbf{F}}(2q+1-n) \otimes_\kk \det(V)^{\vee, \otimes 2}
\]
where the twist by \(\det(V)\) is useful for tracking weights of algebraic
group actions: see \citeFano{2.4}. Here, \(\sO_{\mathbf{F}}(1)\) is the
Pl\"ucker line bundle; its degree is computed in \citeFano{1.15} as
\[
\deg\sO_{\mathbf{F}}(1) =
\frac{(2n-6)!}{(n-1)!(n-3)!} (q+1)^2\big((n-1)q^2 + (2n-8)q + (n-1)\big).
\]

The Fano scheme \(\mathbf{F}\) is singular along the subscheme of lines
through the singular locus of \(X\):
\[
\Sing\mathbf{F} =
\Set{[\ell] \in \mathbf{F} : \ell \cap \Sing X \neq \varnothing}.
\]
In particular, \(\mathbf{F}\) is smooth if and only if \(X\) itself
is smooth, or equivalently, when \(\beta\) is nonsingular: see \citeFano{2.7}.
This description moreover implies that \(\mathbf{F}\) is of expected
dimension if and only if the locus of lines through the singular locus
of \(X\) has dimension at most \(2n-6\).

\subsectiondash{\texorpdfstring{\(q\)}{q}-bic surfaces}
\label{generalities-surfaces}
To complete the tour in low dimensions, consider a \(q\)-bic hypersurface
\(X\) of dimension \(2\), or briefly, a \emph{\(q\)-bic surface}.
Classification and specialization of types looks like
\[
\begin{tikzcd}[row sep=0.1em, column sep=0.8em]
\mathbf{1}^{4} \rar[symbol={\rightsquigarrow}]
& \mathbf{N}_2 \mathbf{1}^{2} \rar[symbol={\rightsquigarrow}]
& \mathbf{N}_4 \ar[dr,symbol={\rightsquigarrow}]
  \ar[rrr,phantom,"\mathbf{N}_2^2","\rightsquigarrow"{xshift=-2.25em},"\rightsquigarrow"{xshift=2.25em}]
&
&
& \mathbf{0}\mathbf{N}_2\mathbf{1} \rar[symbol={\rightsquigarrow}]
& \mathbf{0}\mathbf{N}_3 \rar[symbol={\rightsquigarrow}]
& \mathbf{0}^2\mathbf{1}^2 \rar[symbol={\rightsquigarrow}]
& \mathbf{0}^2\mathbf{N}_2   \rar[symbol={\rightsquigarrow}]
& \mathbf{0}^3\mathbf{1} \\
&
&
& \mathbf{N}_3\mathbf{1} \rar[symbol={\rightsquigarrow}]
& \mathbf{0}\mathbf{1}^3 \ar[ur,symbol={\rightsquigarrow}]
\end{tikzcd}
\]
where, notably, the Hasse diagram of specialization relations is not linear.
See \citeThesis{3.8} for more.

The general statements of \parref{generalities-lines} shows that \(X\) always
contains lines, and that its Fano scheme \(\mathbf{F}\) has degree
\((q+1)(q^3+1)\) when it is of expected dimension \(0\). This means, in particular,
that a smooth \(q\)-bic surface has exact \((q+1)(q^3+1)\) distinct lines.
The configuration of these lines is fascinating and is well-studied: see, for
example, \cite{Hirschfeld:Three, SSvL:Lines, BPRS}. Other aspects of the
geometry of smooth \(q\)-bic surfaces may be found, for example, in
\cite{Shioda:Fermat, Ojiro:Hermitian}.

\subsectiondash{\texorpdfstring{\(q\)}{q}-bic threefolds}\label{generalities-threefolds}
In the remainder of the article, \(X\) denotes a \(3\)-dimensional \(q\)-bic
hypersurface, or a \emph{\(q\)-bic threefold} for short. The classification
and specialization relations of \(q\)-bic threefolds, at least up to the first
few cones, looks as follows:
\[
\begin{tikzcd}[row sep=0.1em, column sep=0.8em]
  \mathbf{1}^5 \rar[symbol={\rightsquigarrow}]
& \mathbf{1}^3 \mathbf{N}_2 \rar[symbol={\rightsquigarrow}]
& \mathbf{1}  \mathbf{N}_4 \rar[symbol={\rightsquigarrow}] \ar[dr,symbol={\rightsquigarrow}]
& \mathbf{N}_5   \rar[symbol={\rightsquigarrow}]
& \mathbf{1}^2 \mathbf{N}_3 \rar[symbol={\rightsquigarrow}] \ar[dr,symbol={\rightsquigarrow}]
& \mathbf{0} \mathbf{1}^4   \ar[dr,symbol={\rightsquigarrow}]
\\
&
&
& \mathbf{1}\mathbf{N}_2^2 \ar[rr,symbol={\rightsquigarrow}]
&
& \mathbf{N}_2 \mathbf{N}_3 \rar[symbol={\rightsquigarrow}]
& \mathbf{0}\mathbf{1}^2\mathbf{N}_2\punct{.}
\end{tikzcd}
\]
Notice that all non-cones have corank at most \(2\), which by the discussion
of \parref{generalities-smoothness}, implies that the singular locus of a
\(q\)-bic threefold which is not a cone has dimension at most \(1\).

The Fano scheme \(S \coloneqq \mathbf{F}\) of lines of a \(q\)-bic threefold
has expected dimension \(2\) by \parref{generalities-lines}. It is easy to check
that when \(X\) has corank at most \(1\) and is not a cone, the singular locus
of \(S\) has dimension at most \(1\). Thus in all such cases, and also in the
case \(X\) is smooth, \(S\) is generically smooth of expected dimension \(2\):
see also \citeFano{2.8}. In any of these cases, but especially when \(X\) is
either smooth or of type \(\mathbf{1}^{\oplus 3} \oplus \mathbf{N}_2\), \(S\)
will be referred to as the \emph{Fano surface} of \(X\).

Begin by considering the Fano surface \(S\) of a smooth \(q\)-bic threefold
\(X\), describing the tangent bundle \(\mathcal{T}_S\), and computing its
Chern numbers:

\begin{Lemma}\label{generalities-chern-numbers}
If \(X\) is smooth, then
\(\mathcal{T}_S \cong \mathcal{S} \otimes_{\sO_S} \sO_S(2-q)\) and the Chern
numbers of \(S\) are
\[
c_1(S)^2 = (q+1)^2(q^2+1)(2q-3)^2
\;\;\text{and}\;\;
c_2(S) = (q+1)^2(q^4 - 3q^3 + 4q^2 - 4q + 3).
\]
\end{Lemma}

\begin{proof}
Identify the tangent bundle of \(S\) as in \citeFano{2.2} as
\(\mathcal{T}_S \cong \HomSheaf_{\sO_S}(\mathcal{S},\mathcal{S}^{[1],\perp}/\mathcal{S})\),
where the sheaf in the target of the \(\HomSheaf\) is characterized by the
short exact sequence
\[
0 \to
\mathcal{S}^{[1],\perp}/\mathcal{S} \to
\mathcal{Q} \xrightarrow{\beta}
\mathcal{S}^{[1],\vee} \to
0.
\]
Taking determinants shows
\(\mathcal{S}^{[1],\perp}/\mathcal{S} \cong \sO_S(1-q)\), and combined with the
wedge product isomorphism
\(\mathcal{S}^\vee \cong \mathcal{S} \otimes_{\sO_S} \sO_S(1)\) gives
the identification of \(\mathcal{T}_S\). To compute the first Chern number,
note that the Pl\"ucker line bundle \(\sO_S(1)\) has degree \((q+1)^2(q^2+1)\)
by taking \(n = 4\) in the formula given in \parref{generalities-lines}.
Combined with the identification of \(\mathcal{T}_S\), this gives
\[
c_1(S)^2 \coloneqq
\int_S c_1(\mathcal{T}_S)^2 =
\int_S c_1(\sO_S(3-2q))^2 =
(2q-3)^2\deg\sO_S(1) =
(q+1)^2(q^2+1)(2q-3)^2.
\]
For the second Chern number, observe that
\begin{align*}
c_2(\mathcal{T}_S)
& = c_2(\mathcal{S}^\vee) + c_1(\mathcal{S}^\vee) c_1(\sO_S(1-q)) + c_1(\sO_S(1-q))^2 \\
& = c_2(\mathcal{S}^\vee) + (q-2)(q-1) c_1(\sO_S(1))^2.
\end{align*}
Since a general section of \(\mathcal{S}^\vee\) cuts out the scheme of
lines contained in a general hyperplane section of \(X\), which is but a smooth
\(q\)-bic surface by \parref{generalities-qbic} and
\parref{generalities-classification}, the degree of \(c_2(\mathcal{S}^\vee)\)
is \((q+1)(q^3+1)\) by \parref{generalities-surfaces}. Combining with the
degree computation for \(\sO_S(1)\) gives \(c_2(S)\).
\end{proof}

This computation implies that the Fano surface \(S\) cannot lift to either
the second Witt vectors of \(\kk\) or any characteristic \(0\) base as
soon as \(q > 2\). By contrast, observe that \(S\) \emph{does} lift when
\(q = 2\), as is seen by taking any lift of \(X\) and by applying the Fano
scheme construction in families.

\begin{Proposition}\label{generalities-nonliftable}
If \(X\) is smooth and \(q > 2\), then \(S\) lifts neither to
\(\mathrm{W}_2(\kk)\) nor to characteristic \(0\).
\end{Proposition}

\begin{proof}
Dualizing the tangent bundle computation of \parref{generalities-chern-numbers}
shows \(\Omega^1_S \cong \mathcal{S}^\vee \otimes_{\sO_S} \sO_S(q-2)\). Since
\(\mathcal{S}^\vee\) has sections, this implies that Kodaira--Akizuki--Nakano
vanishing fails on \(S\), so
\cite[Corollaire 2.8]{DI} shows \(S\) cannot lift to \(\mathrm{W}_2(\kk)\).
The Chern number computation of \parref{generalities-chern-numbers} gives
\[
c_1(S)^2 - 3c_2(S) = q^2(q+1)^2(q^2-3q+1).
\]
The quadratic formula shows that this is positive whenever \(q > 2\), and so
\(S\) the Bogomolov--Miyaoka--Yau inequality of \cite[Theorem 4]{Miyaoka:Inequality}
is not satisfied on \(S\). Since Chern numbers are constant in flat families,
this implies that \(S\) cannot lift to characteristic \(0\).
\end{proof}

The Chern number computation also gives a simple way to compute the Euler
characteristic of the structure sheaf \(\sO_S\) whenever \(S\) is smooth; this
holds more generally, whenever \(S\) has expected dimension, by comparing with
the Koszul resolution of \(\sO_S\) on \(\mathbf{G}\):

\begin{Proposition}\label{generalities-chi-OS}
If \(S\) is of its expected dimension \(2\), then
\[
\chi(S,\sO_S) = \frac{1}{12}(q+1)^2(5q^4 - 15q^3 + 17q^2 - 16q + 12).
\]
\end{Proposition}

\begin{proof}
When \(X\) is a smooth, so is \(S\), and so Noether's formula, as in
\cite[Example 15.2.2]{Fulton}, applies to give the first equality in
\[
\chi(S,\sO_S)
= \frac{1}{12} (c_1(S)^2 + c_2(S))
= \sum\nolimits_{i = 0}^4 (-1)^i
\chi(\mathbf{G}, \wedge^i \mathcal{S}^{[1]} \otimes_{\sO_{\mathbf{G}}} \mathcal{S}).
\]
Substituting the Chern number computations from
\parref{generalities-chern-numbers} gives the formula in the
statement. The second equality arises from taking the Koszul resolution of
\(\sO_S\) on \(\mathbf{G}\). Since the Koszul resolution persists whenever
\(S\) is of dimension \(2\), this gives the general case.
\end{proof}

\section{Cone situation}\label{section-cone-situation}
Study the Fano scheme \(S\) by fibering it over a \(q\)-bic curve \(C\).
This is easy to arrange rationally, as is promptly explained in
\parref{cone-situation-rational-map-S}, and the main content of this section is
to construct a canonical resolution of the rational map
\(S \dashrightarrow C\) using a mixture of projective geometry techniques
and the intrinsic theory of \(q\)-bic forms from \cite{qbic-forms}:
see \parref{cone-situation-blowup} for a summary. The setting in this section
is quite general; the most important situations for the main computations of
the article---where additional smoothness hypotheses are satisfied---are
classified in
\parref{cone-situation-properties}\ref{cone-situation-properties.smooth}, and
are further studied in \S\parref{section-smooth-cone-situation}.

\subsectiondash{Cone situation}\label{cone-situation}
A triple \((X,\infty,\PP\breve{W})\) consisting of a \(q\)-bic threefold \(X\),
a point \(\infty = \PP L\) of \(X\), and a hyperplane \(\PP\breve{W}\) such
that \(X \cap \PP\breve{W}\) is a cone with vertex
\(\infty\) over a reduced \(q\)-bic curve \(C \subset \PP\widebar{W}\) is
called a \emph{cone situation}; here and later,
\(\widebar{W} \coloneqq \breve{W}/L\) and \(\widebar{V} \coloneqq V/L\). This
is sometimes considered with some of the following additional geometric
assumptions:
\begin{enumerate}
\item\label{cone-situation.plane}
there does not exist a \(2\)-plane contained in \(X\) which passes through
\(\infty\); or
\item\label{cone-situation.vertex}
\(\infty\) is a smooth point of \(X\); or
\item\label{cone-situation.curve}
\(C\) is a smooth \(q\)-bic curve.
\end{enumerate}

Conditions \ref{cone-situation.vertex} and \ref{cone-situation.curve} together
imply \ref{cone-situation.plane}. That \(\infty\) is a vertex for \(X \cap
\PP\breve{W}\) means that \(L\) lies in the radical of \(\beta\rvert_W\), so
that \(W\) is contained in both orthogonals of \(L\), see
\parref{generalities-classification}. In particular, this means that
\(\PP\breve{W}\) is contained in the embedded tangent space
\(\mathbf{T}_{X,\infty} = \PP L^{[1],\perp}\) of \(X\) at \(\infty\), see
\parref{generalities-smoothness}; and if \ref{cone-situation.vertex} is
satisfied, then equality \(\PP\breve{W} = \mathbf{T}_{X,\infty}\) necessarily
holds.

\subsectiondash{Examples}\label{cone-situation-examples}
The following are some examples of cone situations:
\begin{enumerate}
\item\label{cone-situation-examples.smooth}
Let \(X\) be a smooth \(q\)-bic threefold and let \(\infty \in X\) be a
Hermitian point. Then \((X,\infty,\mathbf{T}_{X,\infty})\) is a cone situation,
and it follows from \parref{generalities-classification} and
\parref{generalities-hermitian} that all cone situations for smooth \(X\) are
obtained this way. All the conditions
\ref{cone-situation.plane}, \ref{cone-situation.vertex},
and \ref{cone-situation.curve} are satisfied.
\end{enumerate}
The next three examples pertain to \(q\)-bic threefolds of type
\(\mathbf{1}^{\oplus 3} \oplus \mathbf{N}_2\). For concreteness, let
\[
X = \mathrm{V}(x_0^q x_1 + x_0 x_1^q +  x_2^{q+1} + x_3^q x_4) \subset \PP^4,
\]
and set \(C \coloneqq \mathrm{V}(x_0^q x_1 + x_0 x_1^q + x_2^{q+1})\) in the
\(\PP^2\) in which \(x_3\) and \(x_4\) are projected out. Write
\(x_- \coloneqq (0:0:0:1:0)\), \(x_+ \coloneqq (0:0:0:0:1)\), and
\(\infty \coloneqq (1:0:0:0:0)\).
\begin{enumerate}
\setcounter{enumi}{1}
\item\label{cone-situation-examples.N2-}
\((X,x_-,\mathbf{T}_{X,x_-})\) is a cone situation over \(C\)
satisfying each of the conditions \ref{cone-situation.plane},
\ref{cone-situation.vertex}, and \ref{cone-situation.curve}.
\item\label{cone-situation-examples.N2+}
\((X,x_+,\mathrm{V}(x_3))\) is a
cone situation over the \(C\) satisfying
\ref{cone-situation.plane} and \ref{cone-situation.curve},
but not \ref{cone-situation.vertex}.
\item\label{cone-situation-examples.N2infty}
\((X,\infty,\mathbf{T}_{X,\infty})\)
is a cone situation over \(\mathrm{V}(x_2^{q+1} + x_3^q x_4)\) satisfying
\ref{cone-situation.plane} and \ref{cone-situation.vertex},
but not \ref{cone-situation.curve}.
\end{enumerate}
The remaining examples illustrate the range of possibilities for the cone
situation.
\begin{enumerate}
\setcounter{enumi}{4}
\item\label{cone-situation-examples.N4}
Let \(X \coloneqq \mathrm{V}(x_0^q x_1 + x_1^q x_2 + x_3^q x_4)\) and
\(\infty \coloneqq (0:0:0:1:0)\). Then \((X,\infty,\mathbf{T}_{X,\infty})\)
is a cone situation over \(\mathrm{V}(x_0^q x_1 + x_1^q x_2)\) satisfying
\ref{cone-situation.vertex}, but not \ref{cone-situation.plane}
nor \ref{cone-situation.curve}.
\item\label{cone-situation-examples.1+1+1+1+0}
Let \(X\) be of type \(\mathbf{0} \oplus \mathbf{1}^{\oplus 4}\),
\(\infty\) the vertex of \(X\), and \(\PP\breve{W}\) a general
hyperplane through \(\infty\). Then \((X,\infty,\PP\breve{W})\) is a cone
situation satisfying \ref{cone-situation.curve}, but not
\ref{cone-situation.plane} nor \ref{cone-situation.vertex}.
\item\label{cone-situation-examples.1+N2+0+0}
Let \(X\) be of type
\(\mathbf{0}^{\oplus 2} \oplus \mathbf{1} \oplus \mathbf{N}_2\), \(\infty\)
any point of the vertex of \(X\), and \(\PP\breve{W}\) be any hyperplane intersecting
the vertex of \(X\) exactly at \(\infty\). Then \((X,\infty,\PP\breve{W})\) is a cone
situation satisfying none of \ref{cone-situation.plane},
\ref{cone-situation.vertex}, nor
\ref{cone-situation.curve}.
\end{enumerate}

The following illustrates how the hypotheses of \parref{cone-situation}
translate to geometric consequences on \(X\) and its Fano scheme \(S\) of
lines; the first statement gives sufficient conditions---although not necessary,
as shown by
\parref{cone-situation-examples}\ref{cone-situation-examples.1+1+1+1+0}---for
when \(S\) is of expected dimension \(2\), whereas the second classifies cone
situations satisfying the two smoothness hypotheses on \((X,\infty,\PP\breve{W})\).

\begin{Lemma}\label{cone-situation-properties}
If the cone situation \((X,\infty,\PP\breve{W})\) furthermore satisfies
\begin{enumerate}
\item\label{cone-situation-properties.S-exp-dim}
conditions
\parref{cone-situation}\ref{cone-situation.plane} and
\parref{cone-situation}\ref{cone-situation.vertex}, then
\(\dim S = 2\);
\item\label{cone-situation-properties.smooth}
conditions
\parref{cone-situation}\ref{cone-situation.vertex} and
\parref{cone-situation}\ref{cone-situation.curve},
then it is projectively equivalent to either
\parref{cone-situation-examples}\ref{cone-situation-examples.smooth}
or \parref{cone-situation-examples}\ref{cone-situation-examples.N2-}.
\end{enumerate}
\end{Lemma}

\begin{proof}
Suppose that \((X, \infty, \PP\breve{W})\) satisfies
\parref{cone-situation}\ref{cone-situation.plane} and
\parref{cone-situation}\ref{cone-situation.vertex}. Then \(X\) is not a cone:
otherwise, its vertex \(x\), which is different from the smooth
point \(\infty\), would span a plane \(\langle x, \ell \rangle\) through
\(\infty\) with any line \(\ell \subset X \cap \PP\breve{W}\) through \(\infty\) and
not passing through \(x\). The comments in \parref{generalities-threefolds}
show that \(\dim \Sing X \leq 1\). Suppose \(x \in X\) is a singular point with
underlying linear space \(K \subset V\). Lines in \(X\) through \(x\) are
contained in
\[
X \cap \PP K^{[1],\perp} \cap \PP K^{\perp,[-1]} =
X \cap \mathbf{T}_{X,x} \cap \PP K^{\perp,[-1]} =
X \cap \PP K^{\perp, [-1]}
\]
since the underlying linear space pairs to \(0\) with \(K\) on either side
of \(\beta\): see also \citeFano{3.1}. Since \(x\) is not a vertex, this is a
surface which is a cone with vertex \(x\) over a \(q\)-bic curve. Thus the
locus of lines through \(x\) is of dimension \(1\). Therefore discussion of
\parref{generalities-lines} implies that \(\dim S = 2\).

Suppose now that \((X,\infty, \PP\breve{W})\) furthermore satisfies
\parref{cone-situation}\ref{cone-situation.curve}. If \(X\) itself is smooth,
then it must be
\parref{cone-situation-examples}\ref{cone-situation-examples.smooth}, so
suppose \(X\) is singular. Since \(\infty\) is a smooth point, and since
\(\PP\breve{W} = \mathbf{T}_{X,\infty}\) intersects \(X\) at a cone with a smooth
base, the singular locus of \(X\) must be a single point \(x = \PP K\) disjoint
from \(\PP\breve{W}\). But \(W = L^{[1],\perp}\), so this means that the natural map
\(K \to L^{[1],\vee}\) is nonzero, whence an isomorphism. Splitting
\(V = \breve{W} \oplus K\), it now follows that
\[
V^{[1],\perp} =
\ker(\breve{W} \oplus K \to \breve{W}^{[1],\vee}) =
\ker(\breve{W} \to \widebar{W}^{[1],\vee}) =
L.
\]
Thus the restriction of \(\beta\) to \(U \coloneqq L \oplus K\) is of type
\(\mathbf{N}_2\). Since \(U^{[1],\perp} = L^{[1],\perp}\) and
\(U^\perp = K^\perp\) are distinct hyperplanes, their intersection in \(V\)
gives the orthogonal complement to \(U\). Whence \(X\) is of type
\(\mathbf{1}^{\oplus 3} \oplus \mathbf{N}_2\) and \(\infty\) is the smooth
point of the subform \(\mathbf{N}_2\).
\end{proof}

Let \((X,\infty,\PP\breve{W})\) be a cone situation over the \(q\)-bic curve \(C\) and
let \(S\) be the Fano scheme of lines of \(X\). Then there is a canonical
closed immersion \(C \hookrightarrow S\) identifying it with the closed
subscheme
\[
C_{\infty, \PP\breve{W}} \coloneqq \set{[\ell] \in S : \infty \in \ell \subset \PP\breve{W}}
\]
of lines containing \(\infty\) and contained in \(X \cap \PP\breve{W}\), so that the
Pl\"ucker polarization \(\sO_S(1)\) pulls back to the planar polarization
\(\sO_C(1)\). The hypotheses of \parref{cone-situation} affect
\(C_{\infty,\PP\breve{W}}\) as follows:

\begin{Lemma}\label{cone-situation-C}
Suppose the cone situation \((X,\infty,\PP\breve{W})\) furthermore satisfies
\begin{enumerate}
\item\label{cone-situation-C.plane}
condition
\parref{cone-situation}\ref{cone-situation.plane}, then
the closed subsets \(\set{[\ell] \in S : \infty \in \ell}\) and
\(\set{[\ell] \in S : \ell \subset \PP\breve{W}}\) coincide;
\item\label{cone-situation-C.vertex}
condition
\parref{cone-situation}\ref{cone-situation.vertex}, then
\(C_{\infty,\PP\breve{W}}\) coincides with the subscheme \(C_\infty\) of lines in
\(X\) through \(\infty\); and
\item\label{cone-situation-C.normal}
conditions
\parref{cone-situation}\ref{cone-situation.vertex} and
\parref{cone-situation}\ref{cone-situation.curve}, then
\(\mathcal{N}_{C_{\infty,\PP\breve{W}}/S} \cong \sO_C(-q+1)\).
\end{enumerate}
\end{Lemma}

\begin{proof}
Establish \ref{cone-situation-C.plane} through its contrapositive:
To see ``\(\supseteq\)'', if there were a line \(\ell \subset X \cap \PP\breve{W}\)
not containing \(\infty\), then the cone \(X \cap \PP\breve{W}\) would contain the
plane \(\langle \infty, \ell \rangle\) spanned by \(\infty\) and \(\ell\).
To see ``\(\subseteq\)'', if there were a line \(\ell \ni \infty\) not
contained in \(\PP\breve{W}\), then the orthogonals of \(L\) contain not only
\(\breve{W}\), by the comments of \parref{cone-situation}, but also the
subspace underlying \(\ell\). Therefore \(L\) lies in the radical of \(\beta\),
so \(X\) is a cone over a \(q\)-bic surface with vertex containing \(\infty\).
The result now follows since every \(q\)-bic surface contains a line by
\parref{generalities-surfaces}.

For \ref{cone-situation-C.vertex}, the comments following
\parref{cone-situation} imply that \(\PP\breve{W}\) is the embedded tangent
space of \(X\) at \(\infty\). Since any line in \(X\) through \(\infty\) is
necessarily contained in \(\mathbf{T}_{X,\infty}\), the moduli problem
underlying \(C_\infty\) is a closed subfunctor of that of \(C_{\infty,
\PP\breve{W}}\). Since the reverse inclusion is clear, equality holds.

For \ref{cone-situation-C.normal}, observe that there is a short
exact sequence of locally free sheaves
\[ 0 \to \mathcal{T}_C \to \mathcal{T}_S\rvert_C \to \mathcal{N}_{C/S} \to 0 \]
where \(C_{\infty,\PP\breve{W}}\) is identified with \(C\): indeed, this is because
the assumptions together with the discussion in \parref{generalities-lines}
imply that \(S\) is smooth along \(C\). Taking determinants then gives the
normal bundle.
\end{proof}

The cone situation gives a way of transforming lines in \(X\) to points of
\(C\), inducing a rational map \(S \dashrightarrow C\) in good cases. Some
hypothesis on \((X,\infty,\PP\breve{W})\) is necessary to exclude, for
instance, example
\parref{cone-situation-examples}\ref{cone-situation-examples.1+1+1+1+0},
wherein the locus of lines through \(\infty\) forms an irreducible component of
\(S\). In the following, write
\(\proj_\infty \colon \PP V \dashrightarrow \PP\widebar{V}\) for the linear
projection away from \(\infty\), and let
\(S^\circ \coloneqq S \setminus C_{\infty, \PP\breve{W}}\).

\begin{Lemma}\label{cone-situation-rational-map-S}
If the cone situation \((X,\infty,\PP\breve{W})\) satisfies
\parref{cone-situation}\ref{cone-situation.plane}, then
there is a rational map \(\varphi \colon S \dashrightarrow C\) given on
points \([\ell] \in S^\circ\) by
\[
\varphi([\ell]) =
\proj_\infty(\ell \cap \PP\breve{W}) =
\proj_\infty(\ell) \cap \PP\widebar{W}.
\]
\end{Lemma}

\begin{proof}
The point is that
\parref{cone-situation-C}\ref{cone-situation-C.plane}
implies that the indeterminacy locus of the formula is a proper closed subset
not containing any irreducible components of \(S\).
\end{proof}

\subsectiondash{}\label{cone-situation-P}
Towards a resolution of \(\varphi\), consider the scheme parameterizing
triples
\[
\PP \coloneqq
\Set{(y \mapsto y_0 \in \ell_0) : y \in \PP\breve{W}, [\ell_0] \in \mathbf{G}(2,\widebar{V}), y_0 \in C}
\]
consisting of a point \(y \in \PP\breve{W}\), a line \(\ell_0 \subset \PP\widebar{V}\), and
a point \(y_0 \in \ell_0 \cap C\), such that \(y\) is contained in the
line in \(\PP\breve{W}\) determined by \(y_0\). Equivalently, \(\PP\) is the product
\(\PP \mathcal{V}_1 \times_C \PP \mathcal{V}_2\) of projective bundles over
\(C\), where \(\mathcal{V}_1\) and \(\mathcal{V}_2\) are defined via pullback
of the left square and pushout of the right square, respectively, of the
following commutative diagram:
\[
\begin{tikzcd}[column sep=1em, row sep=.8em]
0 \ar[rr]
&& \mathcal{V}_1 \ar[rr] \ar[dr,hook] \ar[dd, two heads]
&& V_C \ar[dr, two heads] \ar[rr]
&& \mathcal{V}_2 \ar[rr]
&& 0 \\
&&& \breve{W}_C \ar[ur,hook] \ar[dr, two heads]
&& \widebar{V}_C \ar[ur,two heads] \\
0 \ar[rr]
&& \sO_C(-1) \ar[rr]
&& \widebar{W}_C \ar[rr] \ar[ur,hook]
&& \mathcal{T} \ar[rr] \ar[uu,hook]
&& 0
\end{tikzcd}
\]
where \(\mathcal{T} \coloneqq \mathcal{T}_{\PP\widebar{W}}(-1)\rvert_C\),
\(V_C \coloneqq V \otimes \sO_C\), and similarly for the others.
In particular, there are split short exact sequences
\[
0 \to L_C \to \mathcal{V}_1 \to \sO_C(-1) \to 0
\quad\text{and}\quad
0 \to \mathcal{T} \to \mathcal{V}_2 \to (V/\breve{W})_C \to 0.
\]
Let \(\PP^\circ\) be the open subscheme of \(\PP\) where \(y \neq \infty\)
and \(\ell_0 \not\subset \PP\widebar{W}\). Its closed complement is the
union of two effective Cartier divisor, with points
\[
\PP \setminus \PP^\circ =
\Set{(y \mapsto y_0 \in \ell_0) : y = \infty\;\text{or}\;\ell_0 \subset \PP\widebar{W}} =
\PP \mathcal{V}_1 \times_C \PP\mathcal{T} \cup
\PP L_C \times_C \PP\mathcal{V}_2.
\]
An equation of \(\PP \setminus \PP^\circ\) is as follows: Let
\(\pi_i \colon \PP\mathcal{V}_i \to C\) and \(\pi \colon \PP \to C\) be the
structure morphisms, let \(\pr_i \colon \PP \to \PP \mathcal{V}_i\) be the
projections, and for any \(\sO_\PP\)-module \(\mathcal{F}\) and \(a,b \in
\mathbf{Z}\), write
\[
\mathcal{F}(a,b) \coloneqq
\mathcal{F} \otimes
\pr_1^*\sO_{\pi_1}(a) \otimes
\pr_2^*\sO_{\pi_2}(b).
\]
Then \(\PP \setminus \PP^\circ = \mathrm{V}(u_1u_2)\), where \(u_i\) is
the map obtained by composing the Euler section \(\mathrm{eu}_{\pi_i}\)
with the quotient map for \(\mathcal{V}_i\) in the two sequences above:
\[
u_1 \colon
\sO_\PP(-1,0) \hookrightarrow
\pi^*\mathcal{V}_1 \twoheadrightarrow
\pi^*\sO_C(-1)
\quad\text{and}\quad
u_2 \colon
\sO_\PP(0,-1) \hookrightarrow
\pi^*\mathcal{V}_2 \twoheadrightarrow
(V/\breve{W})_\PP.
\]

Returning to \(\varphi\), observe that the morphism of
\parref{cone-situation-rational-map-S} factors through a morphism
\(S^\circ \to \PP^\circ\) given by
\[
[\ell] \mapsto
\big(\ell \cap \PP\breve{W}  \mapsto \varphi([\ell]) \in \proj_\infty(\ell)\big).
\]
The image of this morphism can be described in terms of the geometry of
\(X_{\ell_0} \coloneqq X \cap P_{\ell_0}\), where
\(P_{\ell_0} \coloneqq \proj_\infty^{-1}(\ell_0)\) is the plane in \(\PP V\)
spanned by \(\infty\) and a line \(\ell_0 \subset \PP\widebar{V}\):

\begin{Lemma}\label{points-of-T-geometrically}
The image of \(S^\circ \to \PP^\circ\) is the scheme
\[
T^\circ =
\Set{(y \mapsto y_0 \in \ell_0) \in \PP^\circ : X_{\ell_0}\;\text{is a cone with vertex}\;y}.
\]
If \((X,\infty,\PP\breve{W})\) satisfies \parref{cone-situation}\ref{cone-situation.plane},
then the morphism \(S^\circ \to T^\circ\) is quasi-finite of degree \(q\).
\end{Lemma}

\begin{proof}
A point \((y \mapsto y_0 \in \ell_0) \in \PP^\circ\) lies in the
image \(T^\circ\) if and only if there is a line \(\ell\) contained in
\(X_{\ell_0} \setminus \infty\) through \(y\). If \(P_{\ell_0} \subset X\),
then any \(y\) is a vertex of the plane \(X_{\ell_0}\), and any line in
\(X_{\ell_0}\) through \(y\) and not passing through \(\infty\) witnesses
membership in \(T^\circ\). Otherwise, \(X_{\ell_0}\) is a \(q\)-bic curve which
contains the line
\[
\langle y,\infty \rangle =
\proj_\infty^{-1}(y_0) =
\proj_\infty^{-1}(\ell_0 \cap \PP\widebar{W}) =
P_{\ell_0} \cap \PP\breve{W}
\]
spanned by \(y\) and \(\infty\) as an irreducible component. Thus
\((y \mapsto y_0 \in \ell_0)\) is contained in \(T^\circ\) if and only if the
residual curve \(X_{\ell_0} - \ell_{y,\infty}\) contains a line passing through
\(y\). Classification of \(q\)-bic curves, as in
\parref{generalities-classification-lowdim}, shows that  this happens if and
only if \(X_{\ell_0}\) is a cone with vertex \(y\). Finally, if the cone
situation satisfies \parref{cone-situation}\ref{cone-situation.plane}, then
only this second case occurs, and this analysis shows that the fibres of
\(S^\circ \to T^\circ\) are the length \(q\) schemes parameterizing the lines
in \(X_{\ell_0} - \ell_{y,\infty}\).
\end{proof}

\subsectiondash{Closure of \texorpdfstring{\(T^\circ\)}{T^o}}\label{cone-situation-closure}
To describe the closure of \(T^\circ\) in \(\PP\), first extend
the moduli description in \parref{points-of-T-geometrically} to all
of \(\PP\) by considering the locally free \(\sO_\PP\)-module \(\mathcal{P}\)
of rank \(3\) defined by the diagram
\[
\begin{tikzcd}
0 \rar
& \pi^*\mathcal{V}_1 \rar
& V_\PP \rar
& \pi^*\mathcal{V}_2 \rar
& 0 \\
0 \rar
& \pi^*\mathcal{V}_1 \rar \uar[equal]
& \mathcal{P} \rar \uar[hook]
& \sO_\PP(0,-1) \uar[hook,"\mathrm{eu}_{\pi_2}"] \rar
& 0
\end{tikzcd}
\]
where the upper sequence is as in \parref{cone-situation-P}. The fibre
at a point \((y \mapsto y_0 \in \ell_0) \in \PP\) of
\begin{itemize}
\item \(\mathcal{P}\) is the subspace of \(V\) underlying the plane \(P_{\ell_0}\);
\item \(\pi^*\mathcal{V}_1\) is the line \(\ell_{0,\infty} \coloneqq \langle y_0,\infty \rangle\); and
\item the tautological subbundle \(\sO_{\PP}(-1,0) \hookrightarrow \pi^*\mathcal{V}_1\)
is the point \(y\).
\end{itemize}
Let
\(\beta_{\mathcal{P}} \colon \mathcal{P}^{[1]} \otimes \mathcal{P} \to \sO_{\PP}\)
be the restriction of the \(q\)-bic form \(\beta\). Since \(\mathcal{V}_1\)
is isotropic---it parameterizes lines in \(X\)!---the adjoints of
\(\beta_{\mathcal{P}}\) induce maps
\[
\beta^\vee_{\mathcal{P}} \colon
\pi^*\mathcal{V}_1^{[1]} \xrightarrow{\beta^\vee\rvert_{\mathcal{V}_1}}
\pi^*\mathcal{V}_2^\vee \xrightarrow{\mathrm{eu}_{\pi_2}^\vee}
\sO_\PP(0,1)
\quad\text{and}\quad
\beta_{\mathcal{P}} \colon
\pi^*\mathcal{V}_1 \xrightarrow{\beta\rvert_{\mathcal{V}_1}}
\pi^*\mathcal{V}_2^{[1],\vee} \xrightarrow{\mathrm{eu}_{\pi_2}^{[1],\vee}}
\sO_\PP(0,q).
\]
Set
\(v_1 \coloneqq \beta^\vee_{\mathcal{P}} \circ \mathrm{eu}_{\pi_1}^{[1]}\) and
\(v_2 \coloneqq \beta_{\mathcal{P}} \circ \mathrm{eu}_{\pi_1}\).
Then
\(v \coloneqq (v_1,v_2) \colon \sO_\PP \to \sO_\PP(q,1) \oplus \sO_\PP(1,q)\)
vanishes at points where the line in \(\mathcal{P}\) given by \(\mathrm{eu}_{\pi_1}\)
lies in the radical of \(\beta_{\mathcal{P}}\), or, geometrically:
\[
T'
\coloneqq \mathrm{V}(v)
= \Set{(y \mapsto y_0 \in \ell_0) \in \PP : X_{\ell_0}\; \text{is a cone with vertex}\; y}.
\]
In particular, \(T^\circ = T' \cap \PP^\circ\).

This contains too much: for instance, \(T'\) contains the intersection
of the irreducible components of \(\PP \setminus \PP^\circ\), consisting of
\((\infty \mapsto y_0 \in \ell_0)\) where \(\ell_0 \subset \PP\widebar{W}\).
In fact, the map \(v = (v_1,v_2)^\vee\) often factors through
\(u \coloneqq (u_1,u_2)^\vee\). To explain, given splittings
\(\mathcal{V}_1 \cong \sO_C(-1) \oplus L_C\)
and \(\mathcal{V}_2 \cong \mathcal{T} \oplus (V/\breve{W})_C\), write
\[
u_1' \colon
  \sO_\PP(-1,0) \stackrel{\mathrm{eu}_{\pi_1}}{\longrightarrow}
  \pi^*\mathcal{V}_1 \longrightarrow
  L_\PP
\quad\text{and}\quad
u_2' \colon
  \sO_\PP(0,-1) \stackrel{\mathrm{eu}_{\pi_2}}{\longrightarrow}
  \pi^*\mathcal{V}_2 \longrightarrow
  \pi^*\mathcal{T}
\]
for the projection of the Euler sections to the subbundle.

\begin{Lemma}\label{cone-situation-v'}
Assume there exists a \(2\)-dimensional subspace \(U \subset V\) such that
\begin{enumerate}
\item \(U \cap \breve{W} = L\), and
\item \(U\) admits an orthogonal complement \(W\) in \(V\).
\end{enumerate}
Then the induced splittings \(\mathcal{V}_1 \cong \sO_C(-1) \oplus L_C\)
and \(\mathcal{V}_2 \cong \mathcal{T} \oplus (V/\breve{W})_C\) are such that the
adjoint maps
\begin{align*}
\beta^\vee\rvert_{\mathcal{V}_1}
& \eqqcolon \beta^\vee_1 \oplus \beta^\vee_2 \colon
\sO_C(-q) \oplus L_C^{[1]}
\to \mathcal{T}^\vee \oplus (V/\breve{W})_C^\vee, \\
\beta\rvert_{\mathcal{V}_1}
& \eqqcolon \beta_1^{\phantom{\vee}} \oplus \beta_2^{\phantom{\vee}} \colon
\sO_C(-1) \oplus L_C^{\phantom{[1]}}
\to \mathcal{T}^{[1],\vee} \oplus (V/\breve{W})_C^{[1],\vee}
\end{align*}
are diagonal, and there exists a factorization \(v = v' \circ (u_1,u_2)^\vee\)
where \(v'\) is the map
\[
\begin{pmatrix}
u_2' \cdot \beta_1^\vee \cdot u_1^{q-1} &
\beta_2^\vee \cdot u_1'^q \\
u_2'^q \cdot \beta_1 &
u_2^{q-1} \cdot \beta_2 \cdot u_1'
\end{pmatrix}
\colon
\sO_\PP(1,0) \otimes \pi^*\sO_C(-1) \oplus \sO_\PP(0,1) \otimes (V/\breve{W})_\PP
\to \sO_\PP(q,1) \oplus \sO_\PP(1,q).
\]
\end{Lemma}

\begin{proof}
The decomposition \(V = W \oplus U\) induces splittings
\(\mathcal{V}_1 \cong \sO_C(-1) \oplus L_C\) and
\(\mathcal{V}_2 \cong \mathcal{T} \oplus (V/\breve{W})_C\) by intersecting
\(\mathcal{V}_1\) with \(W_C\) as subbundles of \(V_C\), and
by projecting \(U_C\) to \(\mathcal{V}_2\), respectively. The adjoint maps of
\(\beta\) are now diagonal since \(W\) and \(U\) are orthogonal complements.
The Euler sections split as \(\mathrm{eu}_{\pi_1} = (u_1,u_1')^\vee\) and
\(\mathrm{eu}_{\pi_2} = (u_2',u_2)^\vee\), and so
\begin{align*}
v_1
& = \mathrm{eu}_{\pi_2}^\vee \cdot \beta^\vee\rvert_{\mathcal{V}_1} \cdot \mathrm{eu}_{\pi_1}^{[1]}
= u_2' \cdot \beta_1^\vee \cdot u_1^q + u_2 \cdot \beta_2^\vee \cdot u_1'^q,\;\text{and} \\
v_2
& = \mathrm{eu}_{\pi_2}^{[1],\vee} \cdot \beta\rvert_{\mathcal{V}_1} \cdot \mathrm{eu}_{\pi_1}
= u_2'^q \cdot \beta_1 \cdot u_1 + u_2^q \cdot \beta_2 \cdot u_1'.
\end{align*}
Since \(u_1\) and \(u_2\) are locally multiplication by a scalar, their action
commutes with the other maps present, and a direct computation shows that there
is a factorization \(v = v \circ (u_1,u_2)^\vee\), as claimed.
\end{proof}

The hypotheses of \parref{cone-situation-v'} are satisfied, for example, in
all of the cone situations
\parref{cone-situation-examples}\ref{cone-situation-examples.smooth}--\ref{cone-situation-examples.N2+},
so, in particular, in all smooth cone situations as in
\parref{cone-situation-properties}\ref{cone-situation-properties.smooth}: in
\parref{cone-situation-examples}\ref{cone-situation-examples.smooth}, \(U\) may
be taken to be the span of \(L\) with any other isotropic Hermitian vector not
lying in \(\breve{W}\); in
\parref{cone-situation-examples}\ref{cone-situation-examples.N2-} and
\parref{cone-situation-examples}\ref{cone-situation-examples.N2+}, \(U\) is
necessarily the subspace supporting the subform of type \(\mathbf{N}_2\). In
contrast, no \(U\) exists in \parref{cone-situation-examples}\ref{cone-situation-examples.N4}.

This setting provides a candidate for the Zariski closure of \(T^\circ\) in
\(\PP\). To state the result, write
\begin{align*}
\mathcal{E}_2
& \coloneqq
\sO_\PP(1,0) \otimes \pi^*\sO_C(-1) \oplus
\sO_\PP(0,1) \otimes (V/\breve{W})_\PP, \\
\mathcal{E}_1
& \coloneqq
\sO_\PP(q,1) \oplus \sO_\PP(1,q) \oplus \det(\mathcal{E}_2),
\end{align*}
and \(\wedge u \colon \mathcal{E}_2 \to \det(E_2)\) for the natural surjection
to the torsion-free quotient of \(\coker(u \colon \sO_\PP \to \mathcal{E}_2)\).

\begin{Proposition}\label{cone-situation-equations-of-T}
In the setting of \parref{cone-situation-v'}, the scheme
\(T \coloneqq T' \cap \mathrm{V}(\det(v'))\) contains the Zariski
closure of \(T^\circ\) in \(\PP\), and is the rank \(1\) degeneracy locus of
\[
\phi \coloneqq
\left(\begin{smallmatrix} v' \\ \wedge u \end{smallmatrix}\right) \colon
\mathcal{E}_2 \to
\mathcal{E}_1
\]
If furthermore \(\dim S^\circ = 2\), then \(\dim T = 2\), \(T\) is connected,
Cohen--Macaulay, and there is an exact complex of sheaves on \(\PP\) given by
\[
0 \longrightarrow
\mathcal{E}_2(-q-1,-q-1) \stackrel{\phi}{\longrightarrow}
\mathcal{E}_1(-q-1,-q-1)  \stackrel{\wedge^2\phi^\vee}{\longrightarrow}
\sO_\PP \longrightarrow
\sO_T \longrightarrow
0.
\]
\end{Proposition}

\begin{proof}
By \parref{cone-situation-closure}, \(T^\circ\) is the vanishing locus of
\(v = v' \circ u\) on the open subscheme \(\PP^\circ\) where neither \(u_1\)
nor \(u_2\) vanish. Therefore \(v'\) has rank at most \(1\) on \(T^\circ\),
meaning \(\det(v')\rvert_{T^\circ} = 0\) and \(T^\circ \subseteq T\).
To express \(T\) as a degeneracy locus, restrict first to the vanishing locus
of \(\det(v')\), which is but the rank \(1\) locus of
\(v' \colon \mathcal{E}_2 \to \sO_\PP(q,1) \oplus \sO_\PP(1,q)\). Now the
factorization \(v = v' \circ u\) means that \(T\) is the locus where
\(\image(u) \subseteq \ker(v')\). On the one hand, \(\ker(\wedge u)\) is the
saturation in \(\mathcal{E}_2\) of \(\image(u)\). On the other hand, \(v'\) is
a map between locally free sheaves, and so \(\ker(v')\) is saturated.
Therefore \(T\) is equivalently the locus in \(\mathrm{V}(\det(v'))\) where
\(\ker(\wedge u) \subseteq \ker(v')\). Altogether, this means that \(T\) is the
locus in \(\PP\) where
\(\phi \coloneqq \left(\begin{smallmatrix} v' \\ \wedge u \end{smallmatrix}\right)\)
has rank at most \(1\), as required.

For the remainder, note that its description as a degeneracy locus implies
\(\dim T \leq 2\). Suppose that \(\dim S^\circ = 2\). Then its image
\(T^\circ\) has dimension at most \(2\). Since
\(T \setminus T^\circ \subseteq \PP \setminus \PP^\circ\), it follows
that \(\dim T = 2\). Since the base curve \(C\) is reduced, it is
Cohen--Macaulay, and thus so is \(\PP\). With this, the remaining properties
follow from the fact that \(T\) is a degeneracy locus of expected dimension:
see \cite[Theorem 1]{HE:Determinantal}, or also \cite[Theorem 14.3(c)]{Fulton},
for Cohen--Macaulayness; see \cite[Theorem B.2.2(ii)]{Lazarsfeld:PositivityI}
for exactness of the Eagon--Northcott complex \((\mathrm{EN}_0)\) associated
with \(\phi\). Finally, connectedness of \(T\) follows from that of \(S\), see
\parref{generalities-lines}.
\end{proof}

In many cases, \(T\) is the closure of \(T^\circ\). This is verified
for smooth cone situations in \parref{smooth-cone-situation-closure}.

\subsectiondash{}\label{cone-situation-blowup}
Write \(\mathcal{T}_{\pi_i}\) for the pullback to \(\PP\) of the relative
tangent bundle of \(\pi_i \colon \PP\mathcal{V}_i \to C\), and set
\[
\mathcal{V} \coloneqq
\mathcal{H}\big(
\sO_\PP(-1,0) \hookrightarrow
V_\PP \twoheadrightarrow
\mathcal{T}_{\pi_2}(0,-1)\big)
\cong
\coker\big(
\sO_\PP(-1,0) \hookrightarrow
\pi^*\mathcal{V}_1 \hookrightarrow
\mathcal{P}\big)
\]
where \(\mathcal{H}\) extracts the homology sheaf, and
\(\mathcal{P}\) is as in \parref{cone-situation-closure}. Its
associated \(\PP^1\)-bundle is
\[
\mathbf{P}\mathcal{V} =
\Set{((y \in \ell) \mapsto (y_0 \in \ell_0)) :
(y \mapsto y_0 \in \ell_0) \in \PP\;\text{and}\;
\ell\;\text{a line in}\; P_{\ell_0}}.
\]
Extracting the line \([\ell]\) yields a birational morphism
\(\PP\mathcal{V} \to \mathbf{G}\) which is an isomorphism away
from the locus where either \(\infty \in \ell\) or \(\ell \subset \PP\breve{W}\).
The discussion preceding \parref{points-of-T-geometrically}
implies that \(S\) is contained in the image of \(\PP\mathcal{V}\) and is not
completely contained in the non-isomorphism locus, so it has a well-defined
strict transform \(\tilde S\) along \(\PP\mathcal{V} \to \mathbf{G}\)
that has a morphism \(\tilde S \to T\);
when \((X,\infty,\PP\breve{W})\) satisfies
\parref{cone-situation}\ref{cone-situation.plane}, this resolves the rational
map \(\varphi \colon S \dashrightarrow T\) from
\parref{cone-situation-rational-map-S}. In summary, there is a commutative
diagram:
\[
\begin{tikzcd}
S \ar[dr,dashed]
& \lar \tilde S \rar[hook] \dar
& \PP\mathcal{V} \dar \\
& T \rar[hook] \ar[dr]
& \PP \dar["\pi"] \\
&& C\punct{.}
\end{tikzcd}
\]

From now on, assume that \((X,\infty,\PP\breve{W})\) satisfies
\parref{cone-situation}\ref{cone-situation.plane}. The remainder of this
section is dedicated to constructing equations for \(\tilde{S}\) by describing
it as a bundle of \(q\)-bic points over \(T\). Write \(\rho \colon
\PP\mathcal{V}_T \to T\) for the projection of \(\PP\mathcal{V} \to \PP\)
restricted to \(T\). Then \parref{points-of-T-geometrically} implies that
\(\tilde{S}\) is a hypersurface in \(\PP\mathcal{V}_T\).
Let \((\mathcal{P}_T,\beta_{\mathcal{P}_T})\) be the restriction to \(T\) of
the \(q\)-bic form \((\mathcal{P},\beta_{\mathcal{P}})\) from
\parref{cone-situation-closure}. As a first step:

\begin{Lemma}\label{cone-situation-section-over-T}
The \(q\)-bic form \(\beta_{\mathcal{P}_T}\) induces a \(q\)-bic form
\(\beta_{\mathcal{V}_T} \colon \mathcal{V}_T^{[1]} \otimes \mathcal{V}_T \to \sO_T\)
whose \(q\)-bic equation
\[
\beta_{\mathcal{V}_T}(\mathrm{eu}_\rho^{[1]}, \mathrm{eu}_\rho) \colon
\sO_\rho(-q-1)
\hookrightarrow \rho^*\mathcal{V}_T^{[1]} \otimes \rho^*\mathcal{V}_T
\rightarrow \sO_{\PP\mathcal{V}_T},
\]
vanishes at \(((y \in \ell) \mapsto (y_0 \in \ell_0)) \in \PP\mathcal{V}_T\)
if and only if \(\ell\) is an isotropic line for \(\beta\). In particular, this
section vanishes on the strict transform \(\tilde{S}\) of \(S\).
\end{Lemma}

\begin{proof}
By \parref{cone-situation-closure}, the intersection
\(X_{\ell_0} = X \cap P_{\ell_0}\) is a cone with vertex \(y\) for every
\((y \mapsto y_0 \in \ell_0) \in T\). The description of the fibres of
\(\mathcal{P}_T\) from \parref{cone-situation-closure} implies that
the tautological subbundle \(\sO_T(-1,0) \hookrightarrow \mathcal{P}_T\) lies
in the radical of the form \(\beta_{\mathcal{P}_T}\), and so it induces a
\(q\)-bic form \(\beta_{\mathcal{V}_T}\) on the quotient
\(\mathcal{V}_T\). Comparing again with the description of the fibres of \(\mathcal{P}_T\)
shows that the fibre of the tautological subbundle \(\sO_\rho(-1)\) at
\(((y \in \ell) \mapsto (y_0 \in \ell_0))\) extracts the subspace of \(V\)
underlying \(\ell\), whence the latter statement.
\end{proof}

\subsectiondash{}\label{cone-situation-V-sequence}
Comparing the homology sheaf construction of \(\mathcal{V}\) from
\parref{cone-situation-blowup} with the sequences for \(\mathcal{V}_1\) and
\(\mathcal{V}_2\) from \parref{cone-situation-P} gives a canonical short exact
sequence
\[
0 \to
\mathcal{T}_{\pi_1}(-1,0) \to
\mathcal{V} \to
\sO_\PP(0,-1) \to
0.
\]
The subbundle may be identified via the Euler sequence for \(\PP\mathcal{V}_1 \to C\)
as:
\[
\mathcal{T}_{\pi_1}(-1,0)
= \coker(\mathrm{eu}_{\pi_1} \colon \sO_\PP(-1,0) \to \pi^*\mathcal{V}_1)
\cong \det(\pi^*\mathcal{V}_1)(1,0)
\cong \sO_\PP(1,0) \otimes \pi^*\sO_C(-1) \otimes L.
\]
Its inverse image under the quotient map \(\mathcal{P} \to \mathcal{V}\)
is the subbundle \(\pi^*\mathcal{V}_1 \subset \mathcal{P}\), so its points are
\[
\PP(\mathcal{T}_{\pi_1}(-1,0)) =
\Set{((y \in \ell) \mapsto (y_0 \in \ell_0)) \in \PP\mathcal{V} :
\ell = \langle y_0, \infty\rangle}.
\]
Since \(\ell = \langle y_0,\infty \rangle \subset \PP\breve{W}\) whenever
\((y \mapsto y_0 \in \ell_0) \in T\), this
subbundle is isotropic for \(\beta_{\mathcal{V}_T}\) by
\parref{cone-situation-section-over-T}, yielding the following
observation:

\begin{Lemma}\label{cone-situation-section-over-T-subbundle}
The \(q\)-bic equation
\(\beta_{\mathcal{V}_T}(\mathrm{eu}_\rho^{[1]}, \mathrm{eu}_\rho)\) from
\parref{cone-situation-section-over-T} vanishes on
\(\PP(\mathcal{T}_{\pi_1}(-1,0)\rvert_T)\),
and so it factors through the section
\[
v_3 \coloneqq
u_3^{-1}\beta_{\mathcal{V}_T}(\mathrm{eu}^{[1]}_\rho, \mathrm{eu}_\rho) \colon
\sO_{\PP\mathcal{V}_T} \to \sO_\rho(q) \otimes \rho^*\sO_T(0,1)
\]
where
\(u_3 \colon \sO_\rho(-1) \to \rho^*\mathcal{V}_T \to \rho^*\sO_T(0,-1)\)
is the equation of the subbundle. \qed
\end{Lemma}

\begin{Lemma}\label{cone-situation-vanishing-on-tilde-S}
The section \(v_3\) vanishes on \(\tilde{S}\), and
\(\tilde{S} \to T\) is finite flat of degree \(q\) onto its image.
\end{Lemma}

\begin{proof}
The discussion of \parref{cone-situation-blowup} and
\parref{cone-situation-section-over-T-subbundle} together with
\parref{cone-situation-C}\ref{cone-situation-C.plane} implies that the
intersection of \(\tilde S\) with the exceptional locus of
\(\PP\mathcal{V} \to \mathbf{G}\) is contained in
\(\PP(\mathcal{T}_{\pi_1}(-1,0)\rvert_T)\). Therefore
\(v_3\) vanishes on \(\tilde{S}\) if and only if it vanishes on
\(S^\circ \cong \tilde S \setminus \PP(\mathcal{T}_{\pi_1}(-1,0)\rvert_T)\).
Since \(u_3\) is invertible on the latter open subscheme, \(v_3\)
vanishes on \(S^\circ\) by \parref{cone-situation-section-over-T}.

Since \(S^\circ \to T\) is quasi-finite of degree \(q\) by
\parref{points-of-T-geometrically}, the final statement
follows upon showing that \(\mathrm{V}(v_3) \to T\) is finite flat of degree
\(q\). Since \(v_3\) is degree \(q\) on each fibre of \(\PP\mathcal{V}_T \to T\)
by \parref{cone-situation-section-over-T-subbundle},
it suffices to see that \(v_3\) does not vanish on an entire fibre.
But if \(v_3\) did vanish on the fibre over
\((y \mapsto y_0 \in \ell_0) \in T\),
\parref{cone-situation-section-over-T} would imply that all lines
\(\ell \subset P_{\ell_0}\) passing through \(y\) are isotropic for \(\beta\),
and hence \(P_{\ell_0}\) would be contained in \(X\). This is impossible with
condition \parref{cone-situation}\ref{cone-situation.plane}.
\end{proof}

\section{Smooth cone situation}\label{section-smooth-cone-situation}
The most useful cone situations that arise when studying \(q\)-bic threefolds
which are either smooth or of type \(\mathbf{1}^{\oplus 3} \oplus \mathbf{N}_2\)
are the \emph{smooth cone situations}: those \((X,\infty,\PP\breve{W})\) which
satisfy \parref{cone-situation}\ref{cone-situation.vertex} and
\parref{cone-situation}\ref{cone-situation.curve}, and which were classified in
\parref{cone-situation-properties}\ref{cone-situation-properties.smooth} to
be as in either example
\parref{cone-situation-examples}\ref{cone-situation-examples.smooth} or
\parref{cone-situation-examples}\ref{cone-situation-examples.N2-}. In this
setting, the scheme \(T\) constructed in \parref{cone-situation-equations-of-T}
is the Zariski closure of \(T^\circ\) and also a quotient of \(\tilde{S}\) by a
finite group scheme of order \(q\): see \parref{smooth-cone-situation-closure}
and \parref{smooth-cone-situation-quotient}. Furthermore, \(\tilde{S}\) is a
blowup of \(S\) along smooth points, see \parref{smooth-cone-situation-blowup};
this implies that \(\varphi\) from \parref{cone-situation-rational-map-S}
extends to a morphism \(S \to C\) and that the sheaf
\(\mathbf{R}^1\varphi_*\sO_S\) carries a canonical filtration, crucial for
the computations that follow: see \parref{smooth-cone-situation-rational-maps}
and \parref{smooth-cone-situation-pushforward}.

Throughout, let \((X,\infty,\PP\breve{W})\) be a smooth cone
situation, fix an orthogonal decomposition \(V \cong W \oplus U\)
as in \parref{cone-situation-v'}, and write
\(\mathcal{V}_1 \cong \sO_C(-1) \oplus L_C\) and
\(\mathcal{V}_2 \cong \mathcal{T} \oplus (V/\breve{W})_C\) for the induced
splittings. The Euler sections decompose as
\(\mathrm{eu}_{\pi_1} = (u_1,u_1')^\vee\)
and \(\mathrm{eu}_{\pi_2} = (u_2',u_2)^\vee\), and their vanishing loci are:
\begin{align*}
\mathrm{V}(u_1) & = \set{(y \mapsto y_0 \in \ell_0) \in \PP : y = \infty},
&
\mathrm{V}(u_2') & = \set{(y \mapsto y_0 \in \ell_0) \in \PP : \PP\widebar{U} \in \ell_0}, \\
\mathrm{V}(u_1') & = \set{(y \mapsto y_0 \in \ell_0) \in \PP : y \in \PP W},
&
\mathrm{V}(u_2) & = \set{(y \mapsto y_0 \in \ell_0) \in \PP : \ell_0 \subset \PP\widebar{W}}.
\end{align*}
Begin by describing the boundary of \(T^\circ\) in \(T\):

\begin{Proposition}\label{cone-situation-boundary}
The boundary \(T \cap (\PP \setminus \PP^\circ)\) is the union of the effective
Cartier divisors
\begin{align*}
C'
& \coloneqq
\big\{
(\infty \mapsto y_0 \in \ell_0) \in \PP :
y_0 = \PP L_0 \in C
\;\text{and}\;\ell_0^{[1]} = \PP L_0^\perp
\big\} \;\;\text{and} \\
D & \coloneqq
\big\{
(y \mapsto y_0 \in \mathbf{T}_{C,y_0}) \in \PP :
y_0 \in \mathrm{C}_{\mathrm{Herm}}
\;\text{and}\;
y \in \langle y_0, \infty \rangle
\big\}.
\end{align*}
Furthermore, \(C'\) is a purely inseparable multisection of degree \(q\),
and \(D\) is a disjoint union of \(q^3 + 1\) smooth rational curves.
\end{Proposition}

\begin{proof}
Consider the intersection between \(T' = \mathrm{V}(v)\) with the irreducible
components \(Z_i \coloneqq \mathrm{V}(u_i)\) of
\(Z \coloneqq \PP \setminus \PP^\circ = \mathrm{V}(u_1u_2)\). The computation
of \parref{cone-situation-v'} gives the first equalities in
\begin{align*}
T' \cap Z_1
& = (u_1 = u_2 \cdot \beta_2^\vee \cdot u_1'^q = u_2^q \cdot \beta_2 \cdot u_1' = 0)
  = Z_1 \cap Z_2, \;\text{and} \\
T' \cap Z_2
& = (u_2 = u_2' \cdot \beta_1^\vee \cdot u_1^q = u_2'^q \cdot \beta_1 \cdot u_1 = 0)
  = (Z_1 \cap Z_2) \cup (u_2 = u_2' \cdot \beta_1^\vee = u_2'^q \cdot \beta_1 = 0).
\end{align*}
The second equality is then clear for \(T' \cap Z_2\). As for \(T' \cap Z_1\),
it is because \(u_1'\) and \(u_1\) do not vanish simultaneously, and
\(\beta_2^\vee \colon L_C^{[1]} \to (V/\breve{W})_C^\vee\) is an isomorphism by
\parref{cone-situation}\ref{cone-situation.vertex}. Next,
\parref{cone-situation-v'} implies that
\[
-\det(v')\rvert_{T' \cap Z} = u_2'^q \cdot \beta_1 \cdot \beta_2^\vee \cdot u_1'^q.
\]
This cuts out the locus
\((u_1 = u_2 = u_2'^q \cdot \beta_1 = 0)\) on \(Z_1 \cap Z_2\) and
vanishes on the second component of \(T' \cap Z_2\). It remains to identify
these with \(C'\) and \(D\), respectively. The value of
\[
u_2'^q \cdot \beta_1 \colon \sO_C(-1) \to \mathcal{T}^{[1],\vee} \to \sO_\PP(0,q)
\]
at a point \((y \mapsto y_0 \in \ell_0)\) is determined as follows: a local
generator of \(\sO_C(-1)\) corresponds to a basis vector \(v\) of the subspace
\(L_0 \subset \widebar{W}\) underlying \(y_0\); \(\beta_1\) maps this to the
linear functional \(\beta_{\widebar{W}}(-,v)\) on \(\widebar{W}^{[1]}\); and the
quotient map corresponds to restricting this to the linear space underlying
\(\ell_0\). Thus this vanishes if and only if \(\ell_0 = \PP L_0^{\perp,[-1]}\),
and so
\[
(u_1 = u_2 = u_2'^q \cdot \beta_1 = 0) =
\big\{
(\infty \mapsto y_0 \in \ell_0) \in \PP :
y_0 = \PP L_0 \in C
\;\text{and}\; \ell_0 = \PP L_0^{\perp,[-1]}
\big\},
\]
which is precisely \(C'\). Similarly, \(u_2' \cdot \beta_1^\vee\) vanishes at
points \((y \mapsto y_0 \in \ell_0)\) when \(\ell_0 = \PP L_0^{[1],\perp}\).
Therefore
\[
(u_2 = u_2' \cdot \beta_1^\vee = u_2'^q \cdot \beta_1 = 0) =
\big\{
(y \mapsto y_0 \in \ell_0) \in \PP :
y_0 = \PP L_0 \in C\;\text{and}\;
\ell_0 = \PP L_0^{[1],\perp} = \PP L_0^{\perp,[-1]}
\big\}.
\]
This is \(D\) since \(L_0^{[1],\perp} = L_0^{\perp,[-1]}\) if and only if
\(y_0\) is a Hermitian point of \(C\), see \parref{generalities-hermitian}.

Both \(C'\) and \(D\) are effective Cartier divisors on \(T\): The analysis
shows that \(C' = T \cap Z_1\) and \(C' + D = T \cap Z_2\); since \(D\) is the
difference of effective Cartier divisors, it is Cartier itself. The moduli
description easily shows that \(D\) is a union of fibres of
\(\PP\mathcal{V}_1\) over the \(q^3 + 1\) Hermitian points of \(C\).
To see that \(C'\) is a multisection, let \(\phi_C \colon C \to C\) be the map
that sends a point \(y_1\) to the residual intersection point with its tangent
line \(\mathbf{T}_{C,y_1}\), as in \parref{generalities-hermitian}. I claim
that there exists a morphism
\[
\phi_C' \colon C \to C' \colon y_1 \mapsto (\infty \mapsto \phi_C(y_1) \in \mathbf{T}_{C,y_1}).
\]
If \(y_1 = \PP L_1\), then \(L_0 \coloneqq \sigma_\beta(L_1^{[2]})\) underlies
\(\phi_C(y_1)\). Since \(\sigma_\beta = \beta^{-1} \circ \beta^{[1],\vee}\),
the diagram
\[
\begin{tikzcd}
V^{[1]} \rar["\beta^\vee"'] \dar[equal]
& V^\vee \rar \dar["\sigma_\beta^\vee"]
& L_0^\vee \dar["\sigma_\beta^\vee"] \\
V^{[1]} \rar["\beta^{[1]}"]
& V^{[2],\vee} \rar
& L_1^{[2],\vee}
\end{tikzcd}
\]
commutes, and it follows that \(\mathbf{T}_{C,y_1}^{[1]}\) has underlying
linear space \(L_0^\perp\) and that \(\phi_C'\) exists. Observe that
the projection \(C' \to C\) and \(\phi_C'\) are both of degree \(q\), the
latter because \(\mathbf{T}_{C,y_1}\) intersects \(C\) at \(y_1\) generically
with multiplicity \(q\). Since \(\phi_C\) is of purely inseparable of degree
\(q^2\), it follows that \(\phi_C'\) is surjective, and so \(C'\) is a
purely inseparable multisection.
\end{proof}

\begin{Corollary}\label{smooth-cone-situation-closure}
\(T\) is the Zariski closure of \(T^\circ\) in \(\PP\).
\end{Corollary}

\begin{proof}
If not, then some irreducible component of \(T \setminus T^\circ\) would
be an irreducible component of \(T\). But this is impossible:
on the one hand, \(T \setminus T^\circ\) is of pure dimension \(1\) by
\parref{cone-situation-boundary}; on the other hand, \(T\) is
connected and Cohen--Macaulay by
\parref{cone-situation-equations-of-T} and is therefore
equidimension \(2\) by \citeSP{00OV}.
\end{proof}

\begin{Corollary}\label{smooth-cone-situation-S-T}
The map \(\rho \colon \tilde{S} \to T\) is surjective and finite flat of degree
\(q\), \(\tilde{S}\) is the vanishing locus of \(v_3\) in \(\PP\mathcal{V}_T\),
and there is a short exact sequence of bundles on \(T\) given by
\[
0 \to
\sO_T \to
\rho_*\sO_{\tilde{S}} \to
\Div^{q-2}(\mathcal{V}_T)(1,-2) \otimes \pi^*\sO_C(-1) \otimes L \to
0.
\]
In particular, \(\rho_*\sO_{\tilde{S}}\) has an increasing filtration whose
graded pieces are
\[
\gr_i(\rho_*\sO_{\tilde{S}}) =
\begin{dcases*}
\sO_T & if \(i = 0\), \\
(\pi^*\sO_C(-q+i) \otimes L^{\otimes q-i})(q-i,-i-1) & if \(1 \leq i \leq q - 1\).
\end{dcases*}
\]
\end{Corollary}

\begin{proof}
Since \(T\) is the closure of \(T^\circ\) by
\parref{smooth-cone-situation-closure}, the map \(\rho \colon \tilde{S} \to T\)
is proper and dominant, whence surjective; \parref{cone-situation-vanishing-on-tilde-S}
now implies \(\rho\) is flat of degee \(q\), and that \(\tilde{S}\)
is the vanishing locus of \(v_3\). Then
\parref{cone-situation-section-over-T-subbundle} gives a short exact sequence
of sheaves on \(\PP\mathcal{V}_T\):
\[
0 \to
\sO_\rho(-q) \otimes \rho^*\sO_T(0,-1) \xrightarrow{v_3}
\sO_{\PP\mathcal{V}_T} \to
\sO_{\tilde{S}} \to
0.
\]
Pushing this along \(\rho\) gives a short exact sequence of \(\sO_T\)-modules
\[
0 \to
\sO_T \to
\rho_*\sO_{\tilde{S}} \to
\mathbf{R}^1\rho_*\sO_\rho(-q) \otimes \sO_T(0,-1) \to
0.
\]
The Euler sequence gives \(\omega_\rho \cong \rho^*\det(\mathcal{V}_T^\vee) \otimes \sO_\rho(-2)\)
and so by Grothendieck duality
\begin{align*}
\mathbf{R}^1\rho_*\sO_\rho(-q)
 \cong \mathbf{R}\rho_*
    \mathbf{R}\mathcal{H}\!\mathit{om}_{\sO_{\PP\mathcal{V}_T}}(
    \sO_\rho(q) \otimes \omega_\rho, \omega_\rho)
& \cong \mathbf{R}\mathcal{H}\!\mathit{om}_{\sO_T}(
    \mathbf{R}\rho_*\sO_\rho(q-2) \otimes \det(\mathcal{V}_T^\vee), \sO_T) \\
& \cong \Div^{q-2}(\mathcal{V}_T) \otimes \det(\mathcal{V}_T).
\end{align*}
Since \(\det(\mathcal{V}_T) \cong (\pi^*\sO_C(-1) \otimes L)(1,-1)\), this gives
the exact sequence in the statement; the filtration comes from applying divided
powers to the short exact sequence for \(\mathcal{V}_T\) preceding
\parref{cone-situation-section-over-T-subbundle}.
\end{proof}

The map \(\rho \colon \tilde{S} \to T\) is a quotient map. To explain,
let \(G\) be the subgroup scheme of \(\mathbf{GL}_V\) which preserves both
the flag \(L \subset \breve{W} \subset V\) and the \(q\)-bic form
\(\beta\), and induces the identity on \(\breve{W}\) and \(V/\breve{W}\).

\begin{Lemma}\label{smooth-cone-situation-G}
If \((X,\infty,\PP\breve{W})\) is a smooth cone situation, then
\[
G \cong
\begin{dcases*}
  \mathbf{F}_q &
    if \((X,\infty,\PP\breve{W})\) is as in
    \parref{cone-situation-examples}\ref{cone-situation-examples.smooth}, and \\
  \boldsymbol{\alpha}_q &
    if \((X,\infty,\PP\breve{W})\) is as in
    \parref{cone-situation-examples}\ref{cone-situation-examples.N2-}.
\end{dcases*}
\]
\end{Lemma}

\begin{proof}
A point \(g\) of \(G\) induces the identity on \(\breve{W}\), so
\(\delta_g \coloneqq g - \id_V\) descends to a map \(V/\breve{W} \to V\);
since \(g\) also induces the identity on \(V/\breve{W}\), \(\delta_g\)
factors as a map \(V/\breve{W} \to \breve{W}\). Thus the assignment
\(g \mapsto \delta_g\) yields a closed immersion
\(\delta \colon G \to \HomSch(V/\breve{W}, \breve{W})\), the latter
viewed as a vector group. In fact, \(\delta\) factors through the algebraic
subgroup \(\HomSch(V/\breve{W},L)\): that \(G\) preserves \(\beta\) and
acts as the identity on \(\breve{W}\) together means that
\[
\beta(\delta_g(v)^{[1]}, w) = \beta(w', \delta_g(v)) = 0
\]
for every \(\kk\)-algebra \(A\), \(g \in G(A)\), \(v \in V \otimes_\kk A\),
\(w \in \breve{W} \otimes_\kk A\), and
\(w' \in (\breve{W} \otimes_\kk A)^{[1]}\). Splitting
\(\breve{W} \cong W \oplus L\) as in \parref{cone-situation-v'} then shows
that \(\delta_g(v) \in L \otimes_\kk A\).

Construct an equation of \(G\) in \(\HomSch(V/\breve{W},L)\) as follows:
Fix a nonzero \(w \in L\), and choose \(v \in V\) such that its image
\(\bar{v} \in V/\breve{W}\) is nonzero, so that
\((\bar{v} \mapsto t \cdot w) \mapsto t\) is an isomorphism
\(\HomSch(V/\breve{W},L) \cong \mathbf{G}_a\). If \(\delta_g\) corresponds
to \(t \in \mathbf{G}_a(A)\) in this way, then
\[
0 =
\beta((g \cdot v)^{[1]}, g \cdot v) - \beta(v^{[1]},v) \pm \beta(v^{[1]}, g \cdot v) =
\beta(w^{[1]},v) t^q + \beta(v^{[1]}, w) t.
\]
Since \(\infty\) is a smooth point, \(L^{[1],\perp} = \breve{W}\) as
explained in \parref{cone-situation}, and so \(\beta(w^{[1]},v)\) is a nonzero
scalar. The scalar \(\beta(v^{[1]},w)\) is nonzero in
\parref{cone-situation-examples}\ref{cone-situation-examples.smooth}, and
so \(G \cong \FF_q\); whereas it is zero in
\parref{cone-situation-examples}\ref{cone-situation-examples.N2-}, and so \(G
\cong \boldsymbol{\alpha}_q\).
\end{proof}

The linear action of \(G\) on \(V\) induces an action on the schemes under
consideration. First, it is straightforward that this action is trivial on
\(C\) and \(T\): For \(C\), this is because \(G\) acts trivially on
\(\breve{W}\) and fixes \(L\). For \(T\), its points are triples
\((y \mapsto y_0 \in \ell_0)\) where \(y_0 \in C\), \(y \in \PP\breve{W}\),
and \(\ell_0 \subset \PP\widebar{V}\) intersects \(C\) at \(y_0\); since
\(G\) moves neither \(y_0\) nor \(y\), and since \(G\) maps \(V/\breve{W}\)
to \(L\) as in the proof of \parref{smooth-cone-situation-G}, it does not move
\(\ell_0\). Next, comparing with the description in
\parref{cone-situation-section-over-T-subbundle}, this implies that \(G\) fixes
the subbundle \(\PP(\mathcal{T}_{\pi_1}(-1,0))\) in \(\PP\mathcal{V}_T\).
Finally, the action of \(G\) on \(\tilde{S}\) is as follows:

\begin{Lemma}\label{smooth-cone-situation-quotient}
The morphism \(\rho \colon \tilde{S} \to T\) is the quotient map for \(G\).
\end{Lemma}

\begin{proof}
The unipotent algebraic group \(\HomSch(V/\breve{W},L)\) acts freely on the open
subscheme of \(\mathbf{G}\) parameterizing lines \(\ell\) satisfying
\(\infty \notin \ell\) and \(\ell \not\subset \PP\breve{W}\). The proof of
\parref{smooth-cone-situation-G} shows that \(G\) is a closed subgroup scheme
of this unipotent group, and so \(G\) acts freely on the open subscheme
\(S^\circ\) of \(S\). The result follows upon identifying \(S^\circ\) with the
open subscheme \(\tilde{S}\) consisting of points \(((y \in \ell) \mapsto (y_0
\in \ell_0))\) where \(y \neq \infty\) and \(\ell \not\subset
\PP\breve{W}\): Indeed the canonical morphism \(\tilde{S}/G \to T\) is an
isomorphism over \(T^\circ\). Since \(\tilde{S} \to T\) is surjective and
finite of degree \(q\) by \parref{smooth-cone-situation-S-T}, and lengths of
fibres of finite morphisms are upper semicontinuous by Nakayama,
\(\tilde{S}/G \to T\) has degree \(1\), and so is an isomorphism.
\end{proof}

Putting \parref{points-of-T-geometrically}, \parref{cone-situation-boundary},
and \parref{smooth-cone-situation-S-T} together identifies the points of
the complement \(S \setminus S^\circ\) as follows:

\begin{Lemma}\label{smooth-cone-situation-S-tilde-divisors}
The complement of \(S^\circ\) in \(\tilde{S}\) the union of effective
Cartier divisors
\begin{align*}
\tilde{C}_\infty & \coloneqq
\big\{
((\infty \in \ell) \mapsto (y_0 \in \ell_0)) \in \tilde{S} :
y_0 = \PP L_0 \in C,\;
\ell_0^{[1]} = \PP L_0^\perp,\;
\ell = \langle y_0,\infty \rangle
\big\}\;\text{and} \\
E & \coloneqq
\big\{
((y \in \ell) \mapsto (y_0 \in \mathbf{T}_{C,y_0})) \in \tilde{S} :
y_0  \in C_{\mathrm{Herm}}
\,\;\text{and}\;\,
\ell = \langle y_0, \infty \rangle
\big\}.
\end{align*}
\end{Lemma}

\begin{proof}
It remains to prove that \(\tilde{C}_\infty\) and \(E\) are Cartier in
\(\tilde{S}\). Let \(C_\infty\) be the closed subscheme of \(S\)
parameterizing lines through \(\infty\); it follows from
\parref{cone-situation-C}\ref{cone-situation-C.vertex} that it is isomorphic to
the smooth \(q\)-bic curve \(C\), and that, by \parref{generalities-lines},
\(S\) is smooth along \(C_\infty\). Therefore the description of its points
shows that \(\tilde{C}_\infty\) is the strict transform of the effective
Cartier divisor \(C_\infty\), and so it, too, is effective Cartier. Since
\(\tilde{C}_\infty + E\) is the vanishing locus of the section \(u_3\) from
\parref{cone-situation-section-over-T-subbundle}, it follows that \(E\) is also
Cartier.
\end{proof}

Observing that the Hermitian points of \(C_\infty\) parameterize lines \(\ell\)
in \(X\) which project to Hermitian points of \(C\) and comparing with
\parref{smooth-cone-situation-S-tilde-divisors} essentially gives:

\begin{Corollary}\label{smooth-cone-situation-blowup}
The morphism \(\tilde{S} \to S\) is a blowup along the Hermitian points of
\(C_\infty\).
\end{Corollary}

\begin{proof}
The Hermitian points of \(C_\infty\) pullback to the effective Cartier divisor
\(E\) in \(\tilde{S}\), so the blowup \(S' \to S\) along these points admits
a canonical morphism \(S' \to \tilde{S}\). Since \(S\) is smooth along
\(C_\infty\), Zariski's Main Theorem as in \cite[Corollary
III.11.4]{Hartshorne:AG} applies to show that \(\tilde{S} \to S'\) is an
isomorphism.
\end{proof}

In particular, this implies that \(\tilde{S}\) and \(T\) are smooth along their
boundary:

\begin{Corollary}\label{smooth-cone-situation-smooth-boundary}
\(\tilde{S}\) is smooth along \(\tilde{C}_\infty \cup E\), and
\(T\) is smooth along \(C' \cup D\).
\end{Corollary}

\begin{proof}
Smoothness of \(\tilde{S}\) along \(\tilde{C}_\infty \cup E\) follows from
that of \(S\) along \(C_\infty\); that of \(T\) along \(C' \cup D\)
is because \(\rho \colon \tilde{S} \to T\) is flat by
\parref{smooth-cone-situation-S-T}, and smoothness descends along flat
morphisms, see \citeSP{05AW}.
\end{proof}

\begin{Corollary}\label{smooth-cone-situation-rational-maps}
The rational map \(S \dashrightarrow T\) is defined away from the
Hermitian points of \(C_\infty\), and the rational map
\(\varphi \colon S \dashrightarrow C\) extends to a morphism.
\end{Corollary}

\begin{proof}
The rational maps from \(S\) to both \(T\) and \(C\) are resolved up on
\(\tilde{S}\), so the statement about \(S \dashrightarrow T\)  follows directly
from \parref{smooth-cone-situation-blowup}, and that about
\(\varphi \colon S \dashrightarrow C\) follows from
\parref{smooth-cone-situation-S-tilde-divisors} which implies that each
component of the exceptional divisor \(E\) is mapped to a single point along
\(\tilde{S} \to C\).
\end{proof}

Since the boundary divisor \(Z\) in \(\PP\) is relatively ample and since
\(T \cap Z \to C\) has connected fibres by \parref{cone-situation-boundary},
the fibres of \(T \to C\) are connected. Using
\parref{smooth-cone-situation-quotient} and
\parref{smooth-cone-situation-S-tilde-divisors} then implies the same about the
fibres of \(\tilde{S} \to C\), and then \parref{smooth-cone-situation-blowup}
implies the same for \(S \to C\). This almost implies that, for instance,
\(\varphi_*\sO_S \cong \sO_C\), but there is the matter of reduced fibres and
the possibility of a factoring through a purely inseparable cover of \(C\). At
any rate, the following clarifies the structure of
\(\mathbf{R}\varphi_*\sO_S\):

\begin{Lemma}\label{smooth-cone-situation-pushforward}
The natural maps give isomorphisms
\[
\varphi_*\sO_S \cong (\pi \circ \rho)_*\sO_{\tilde{S}} \cong \pi_*\sO_T \cong \sO_C,
\]
and
\(\mathbf{R}^1\varphi_*\sO_S\) is locally free and carries a filtration with
graded pieces
\[
\mathrm{gr}_i(\mathbf{R}^1\varphi_*\sO_S) \cong
\begin{dcases*}
\mathbf{R}^1\pi_*\sO_T & if \(i = 0\), and \\
\sO_C(-q+i) \otimes L^{\otimes q-i} \otimes \mathbf{R}^1\pi_*\sO_T(q-i,-i-1) &
if \(1 \leq i \leq q-1\).
\end{dcases*}
\]
\end{Lemma}

\begin{proof}
By \parref{cone-situation-equations-of-T}, the structure sheaf of
\(T\) admits a resolution on \(\PP\) of the form
\[
0 \to
\mathcal{E}_2(-q-1,-q-1) \to
\mathcal{E}_1(-q-1,-q-1) \to
\sO_\PP \to
\sO_T \to
0.
\]
The twists of \(\mathcal{E}_1\) and \(\mathcal{E}_2\) are sums of negative
line bundles on a \(\PP^1 \times \PP^2\)-bundle over \(C\), and so cohomology
is in degree \(3\). Therefore the spectral sequence computing cohomology of
\(\sO_T\) shows that the natural map \(\sO_C \to \pi_*\sO_T\) is an
isomorphism.

By \parref{smooth-cone-situation-S-T},
\(\rho_*\sO_{\tilde{S}}\) has a filtration with graded pieces \(\sO_T\) and
\(\sO_T(q-i,-i-1) \otimes \pi^*\sO_C(-q+i)\) where \(1 \leq i \leq q-1\).
The resolution above implies that the latter terms have vanishing pushforward:
\(\mathcal{E}_1(-i-1,-q-i-2)\) consists of line bundles which are
negative on the \(\PP^2\)-side of \(\PP \to C\), so cohomology is supported in
degrees \(2\) and \(3\); whereas \(\mathcal{E}_2(-i-1,-q-i-2)\) is negative
on both factors and so cohomology is supported in degree \(3\). Therefore
\(\sO_C \to (\pi \circ \rho)_*\sO_{\tilde{S}}\) is an isomorphism. Finally,
since \(\tilde{S} \to S\) is a blowup at smooth point by
\parref{smooth-cone-situation-blowup}, it follows that
\(\mathbf{R}\varphi_*\sO_S \cong \mathbf{R}(\pi \circ \rho)_*\sO_{\tilde{S}}\)
and so \(\sO_C \to \varphi_*\sO_S\) is an isomorphism, \(\mathbf{R}^1\varphi_*\sO_S\)
is locally free by cohomology and base change, and carries the filtration by
\parref{smooth-cone-situation-S-T}.
\end{proof}

\section{\texorpdfstring{\(q\)}{q}-bic threefolds of type \texorpdfstring{\(\mathbf{1}^{\oplus 3} \oplus \mathbf{N}_2\)}{1³N₂}}
\label{section-nodal}
This and the following two sections are concerned with the geometry of
mildly singular \(q\)-bic threefolds, namely, those of type
\(\mathbf{1}^{\oplus 3} \oplus \mathbf{N}_2\). The purpose of this section
is to construct the normalization \(\nu \colon S^\nu \to S\) of the Fano
surface, see \parref{nodal-normalization}, and to relate the cohomology of
\(\sO_S\) with that of an \(\sO_C\)-module \(\mathcal{F}\) related to the
quotient \(\nu_*\sO_{S^\nu}/\sO_S\),
see \parref{normalize-cohomology} and \parref{conductor-dual}.

\subsectiondash{}\label{nodal-basics}
Throughout \S\S\parref{section-nodal}--\parref{section-cohomology-F}, let
\((V,\beta)\) be a \(q\)-bic form of type
\(\mathbf{1}^{\oplus 3} \oplus \mathbf{N}_2\), and let \(X\) be the associated
\(q\)-bic threefold. The two kernels \(L_+ \coloneqq V^{\perp,[-1]}\) and
\(L_- \coloneqq V^{[1],\perp}\) of \(\beta\) underlie the singular point
\(x_+\) and the special smooth point \(x_-\) of \(X\). Setting
\(U \coloneqq L_- \oplus L_+\), the classification of \(q\)-bic forms
provides a canonical orthogonal decomposition \(V = W \oplus U\)
where \(\beta_W\) is of type \(\mathbf{1}^{\oplus 3}\). The plane \(\PP W\)
intersects \(X\) at the smooth \(q\)-bic curve \(C\) defined by \(\beta_W\);
this curve is canonically identified as the base of the cones \(X_-\) and
\(X_+\) in the cone situations
\[
(X,x_-,\PP L_-^{[1],\perp})
\quad\text{and}\quad
(X,x_+,\PP L_+^{\perp,[-1]})
\]
as in
\parref{cone-situation-examples}\ref{cone-situation-examples.N2-} and
\parref{cone-situation-examples}\ref{cone-situation-examples.N2+}.

The Fano scheme \(S\) of lines in \(X\) is of expected dimension \(2\):
this follows, for example, from
\parref{cone-situation-properties}\ref{cone-situation-properties.S-exp-dim}
applied to either cone situation above. Alternatively, by
\parref{cone-situation-C}\ref{cone-situation-C.plane} and
\parref{cone-situation-C}\ref{cone-situation-C.vertex}, the subschemes
\(C_\pm \subset S\) parameterizing lines through the special points
\(x_\pm \in X\) are supported on a scheme isomorphic to \(C\), so the singular
locus of \(S\) has dimension \(1\), and the discussion of
\parref{generalities-lines} shows \(S\) is a surface.

\subsectiondash{Automorphisms}\label{nodal-automorphisms}
Consider the automorphism group scheme \(\AutSch(V,\beta)\) of the \(q\)-bic
form, as introduced in \citeForms{5.1}. Specializing the computation of
\citeThesis{1.3.7} with \(a = 1\) and \(b = 3\) gives the following explicit
description as a closed sub-group scheme of \(\mathbf{GL}_5\):
\[
\AutSch(V,\beta) \cong
\Set{
\left(
\begin{array}{c@{}|cc}
A \;\;\;& 0 & \mathbf{x} \\
\hline
\mathbf{y}^\vee\;\; & \lambda^{-1} & \epsilon \\
0\;\;\;\; & 0 & \lambda^q
\end{array}
\right) \in \mathbf{GL}_5 |
\begin{array}{ccc}
\lambda \in \mathbf{G}_m, &
\mathbf{x} \in \boldsymbol{\alpha}^3_{q\phantom{^2}}, &
A \in \mathrm{U}_3(q), \\
\epsilon \in \boldsymbol{\alpha}_q, &
\mathbf{y} \in \boldsymbol{\alpha}_{q^2}^3, &
\lambda^q A^{\vee,[1]} \mathbf{x} = \mathbf{y}^{[1]}
\end{array}
}.
\]
A particularly useful subgroup is that which preserves the orthogonal
decomposition \(V = W \oplus U\):
\[
\mathrm{G} \coloneqq \AutSch(W,\beta_W) \times \AutSch(L_- \subset U, \beta_U) \cong
\mathrm{U}_3(q) \times
\Set{
  \begin{pmatrix} \lambda^{-1} & \epsilon \\ 0 & \lambda^q \end{pmatrix} :
  \lambda \in \mathbf{G}_m,  \epsilon \in \boldsymbol{\alpha}_q}
\]
Observe that the torus \(\mathbf{G}_m\) in \(\AutSch(V,\beta)\) acts, in
particular, on \(X\) and the Fano surface \(S\); it is straightforward to
check that its fixed schemes are \(X^{\mathbf{G}_m} = \{x_-, x_+\} \cup C\) and
\(S^{\mathbf{G}_m} = C_- \cup C_+\), respectively.

\subsectiondash{Cone situations}\label{nodal-cone-situations}
Both cone situations in \parref{nodal-basics} associated with the special
points \(x_\mp \in X\) satisfy
\parref{cone-situation}\ref{cone-situation.plane} and
\parref{cone-situation}\ref{cone-situation.curve}, so
\parref{cone-situation-rational-map-S} gives rational maps
\[
\varphi_- \colon S \longrightarrow C
\quad\text{and}\quad
\varphi_+ \colon S \dashrightarrow C,
\]
where the orthogonal decomposition \(V = W \oplus U\) identifies the target as
the curve \(C = X \cap \PP W\). Since \(\varphi_-\) arises from a smooth cone
situation, it extends to a morphism by
\parref{smooth-cone-situation-rational-maps}, whereas \(\varphi_+\) is defined
only away from \(C_+\). Writing
\(\proj_{\PP U} \colon X \dashrightarrow \PP W\) for the rational map induced
by linear projection centred at \(\PP U\), the specific geometry of \(X\)
offers an alternative description of the maps \(\varphi_\mp\):

\begin{Lemma}\label{nodal-projection-from-line-tangent}
Let \(\ell \subset X\) be a line not passing through either \(x_\mp\). Then
\(\ell_0 \coloneqq \proj_{\PP U}(\ell)\) is a line in \(\PP W\) tangent to
\(C\) at the point \(\proj_{\PP U}(\ell \cap X_+)\) with
residual intersection point \(\proj_{\PP U}(\ell \cap X_-)\), so
\begin{align*}
\varphi_+([\ell]) & = \text{point of tangency between \(\ell_0\) and \(C\)}, \\
\varphi_-([\ell]) & = \text{residual point of intersection between \(\ell_0\) and \(C\)}.
\end{align*}
\end{Lemma}

\begin{proof}
Projection from \(\PP U\) contracts only lines through \(x_\mp\), so \(\ell_0\)
is a line in \(\PP W\). Since \(\PP U\) intersects \(X\) only at the vertices
\(x_\mp\) of the two cones over \(C\), and since \(x_+\) is a singular point of
multiplicity \(q\),
\[
C \cap \ell_0 =
\proj_{\PP U}\big(\ell \cap \proj_{\PP U}^{-1}(C)\big) =
\proj_{\PP U}\big(\ell \cap (X_- \cup q X_+)\big).
\]
Thus \(\ell_0\) and \(C\) are
tangent at \(\proj_{\PP U}(\ell \cap X_+)\) and have residual point of
intersection at \(\proj_{\PP U}(\ell \cap X_-)\). On the other hand, the
definition of \(\varphi_\mp\) from \parref{cone-situation-rational-map-S} gives
\[
\varphi_\mp([\ell])
= \proj_{x_\mp}(\ell \cap X_\mp)
= \proj_{\PP U}(\ell \cap X_\mp)
\]
with the second equality because \(x_\pm \notin X_\mp\).
\end{proof}

With more notation, the proof works for families of lines in
\(X \setminus \{x_-, x_+\}\). This leads to another proof of
\parref{smooth-cone-situation-rational-maps}, that \(\varphi_-\) extends over
\(C_-\): \parref{nodal-projection-from-line-tangent} means that
\(\varphi_-([\ell])\) is the residual intersection point of \(C\) with its
tangent line at \(\proj_{\PP U}(\ell \cap X_+)\), and this makes sense even
when \(x_- \in \ell\). More interestingly, this gives an alternative moduli
interpretation of \(S \setminus C_+\) as a scheme over \(C\) via \(\varphi_+\).
To describe it, let \(\mathcal{E}_C\) be the \emph{embedded tangent bundle}
\(C\) in \(\PP W\), which is defined via pullback of the right square in the
following commutative diagram of short exact sequences:
\[
\begin{tikzcd}[row sep = 1em]
0 \rar
& \sO_C(-1) \rar \dar[symbol={=}]
& \mathcal{E}_C \rar \dar[symbol={\subset}]
& \mathcal{T}_C(-1) \rar \dar[symbol={\subset}]
& 0 \\
0 \rar
& \sO_C(-1) \rar
& W_C(-1) \rar
& \mathcal{T} \rar
& 0
\end{tikzcd}
\]
where, as in \parref{cone-situation-P},
\(\mathcal{T} \coloneqq \mathcal{T}_{\PP W}(-1)\rvert_C\). The fibre of
\(\mathcal{E}_C\) over a point \(y_0 = \PP L_0\) is the subspace
\(L_0^{[1],\perp}\) in \(W\) underlying the embedded tangent space
\(\mathbf{T}_{C,y_0}\), see \parref{generalities-smoothness}. The subbundle
\(\mathcal{W} \coloneqq \mathcal{E}_C \oplus U_C\) of \(V_C\) defines the
family \(\PP\mathcal{W} \to C\) of hyperplanes in \(\PP V\) spanned by
the tangent lines to \(C\) and \(U\). Let
\(X_{\PP\mathcal{W}} \coloneqq X \times_{\PP V} \PP\mathcal{W}\)
be the corresponding family of hyperplane sections of \(X\). A reformulation
of \parref{nodal-projection-from-line-tangent} in these terms is then:

\begin{Corollary}\label{nodal-tangent-lines-moduli}
Let \(\mathbf{L}^\circ\) be the restriction of the Fano correspondence
\(S \leftarrow \mathbf{L} \rightarrow X\) to \(S \setminus C_+\). Then
the morphism \(\mathbf{L}^\circ \to X\) factors through \(X_{\PP\mathcal{W}}\)
and fits into a commutative diagram
\[
\begin{tikzcd}
\mathbf{L}^\circ \dar \rar
& X_{\PP\mathcal{W}} \dar \rar & X \\
S \setminus C_+ \rar["\varphi_+"] & C\punct{.}
\end{tikzcd}
\]
\end{Corollary}

\begin{proof}
Away from \(C_- \subset S\), this follows from \parref{nodal-projection-from-line-tangent}.
It remains to observe that the lines through \(x_-\) correspond to the
subbundle \(\sO_C(-1) \oplus L_{-,C}\) of \(\mathcal{W}\).
\end{proof}

A line \(\ell \subset X \setminus x_+\) such that \(\varphi_+([\ell]) = y_0\)
intersects \(X_+\) along the line \(\langle x_+, y_0 \rangle\). These lines
are globally parameterized by the subbundle
\(\mathcal{W}' \coloneqq \sO_C(-1) \oplus L_{+,C}\); together with
\parref{nodal-projection-from-line-tangent} and
\parref{nodal-tangent-lines-moduli}, this implies that
\(\mathbf{L}^\circ \times_{X_{\PP\mathcal{W}}} \PP\mathcal{W}'\)
projects
isomorphically to \(S \setminus C_+\). This suggests that \(\varphi_+\) may be
resolved by considering all lines in \(X_{\PP\mathcal{W}}\) that pass
through \(\PP\mathcal{W}'\). To proceed, consider the morphism
\(X_{\PP\mathcal{W}} \to C\):

\begin{Lemma}\label{nodal-tangent-pencil-types}
\(X_{\PP\mathcal{W}} \to C\) is a family of corank \(2\) \(q\)-bic
surfaces with singular locus \(\PP\mathcal{W}'\).
\end{Lemma}

\begin{proof}
This family of \(q\)-bic surfaces over \(C\) is determined by the \(q\)-bic
form
\(\beta_{\mathcal{W}} \colon \mathcal{W}^{[1]} \otimes \mathcal{W} \to \sO_C\)
obtained by restricting \(\beta\) to \(\mathcal{W}\). The orthogonal
splitting \(V = W \oplus U\) restricts to a decomposition
\[
(\mathcal{W},\beta_{\mathcal{W}}) =
(\mathcal{E}_C, \beta_{W,\mathrm{tan}}) \oplus
(U_C, \beta_U)
\]
where \(\beta_{W,\mathrm{tan}}\) is \(q\)-bic form induced by \(\beta_W\) on
the embedded tangent spaces to \(C\). Observe that \(\beta_U\) is of constant
type \(\mathbf{N}_2\), and \(\beta_{W,\mathrm{tan}}\) is of type
\(\mathbf{0} \oplus \mathbf{1}\) or \(\mathbf{N}_2\) depending on whether or
not the point of tangency is a Hermitian point of \(C\); in particular,
\(\beta_{\mathcal{W}}\) is everywhere of corank \(2\). By \citeFano{2.6},
the nonsmooth locus of \(X_{\PP\mathcal{W}} \to C\) corresponds to the
subbundle
\[
\mathcal{W}^\perp
= \mathcal{E}_C^{\perp_{\beta_{W,\mathrm{tan}}}} \oplus U_C^{\perp_{\beta_U}}
= \sO_C(-q) \oplus L_{+,C}^{[1]} \subset
\mathcal{W}^{[1]}
\]
where the two orthogonals can be computed geometrically by noting that
\(X \cap \PP U\) is singular at \(x_+\), and the intersection between \(C\)
and its tangent line at \(y_0\) is singular at \(y_0\). This bundle descends
through Frobenius to \(\mathcal{W}'\) as in the statement,
and so \(X_{\PP\mathcal{W}}\) itself is singular along thereon.
\end{proof}

Construct the family of lines in \(X_{\PP\mathcal{W}}\) over \(C\) through
\(\PP\mathcal{W}'\) as follows: Set
\(\mathcal{W}'' \coloneqq \mathcal{W}/\mathcal{W}'\),
let \(S^\nu \coloneqq \PP\mathcal{W}''\) be the associated
\(\PP^1\)-bundle, and write \(\tilde\varphi_+ \colon S^\nu \to C\) for the
structure map. Linear projection over \(C\) with centre
\(\PP\mathcal{W}'\) produces a rational map
\(X_{\PP\mathcal{W}} \dashrightarrow S^\nu\) which is resolved on the
blowup \(\tilde{X}_{\PP\mathcal{W}}\) along
\(\PP\mathcal{W}'\). Since \(X_{\PP\mathcal{W}}\) is a family
of \(q\)-bic surfaces over \(C\), its singular locus has multiplicity \(q\),
the fibres of \(\tilde{X}_{\PP\mathcal{W}} \to S^\nu\) are curves of
degree \(1\) in \(\PP V\). In other words, this is a family of lines in \(X\),
and so defines a morphism \(\nu \colon S^\nu \to S\).

\begin{Proposition}\label{nodal-normalization}
The morphism \(\nu \colon S^\nu \to S\) is the normalization,
fits into a commutative diagram
\[
\begin{tikzcd}
S^\nu \ar[rr,"\nu"] \ar[dr,"\tilde\varphi_+"'] && S \ar[dl,dashed,"\varphi_+"] \\
& C
\end{tikzcd}
\]
and satisfies
\(\nu^*\sO_S(1) = \tilde\varphi_+^*\sO_C(1) \otimes \sO_{\tilde\varphi_+}(q+1) \otimes L_+^\vee\).
\end{Proposition}

\begin{proof}
Note \(\mathcal{W}'' \cong \mathcal{T}_C(-1) \oplus L_{-,C}\) and the
fibres of \(\tilde{X}_{\PP\mathcal{W}} \to S^\nu\) over
\(C_+^\nu \coloneqq \PP(\mathcal{T}_C(-1))\) are those lines through
\(x_+ \in X\). The statement of and the comments following
\parref{nodal-tangent-lines-moduli} imply that \(S \setminus C_+\) and
\(S^\nu \setminus C_+^\nu\) represent the same moduli problem over \(C\), and
this implies the first two statements.

To compute the pullback of the Pl\"ucker line bundle, describe the
\(\PP^1\)-bundle \(\tilde{X}_{\PP\mathcal{W}} \to S^\nu\) more
precisely: The blowup of \(\PP\mathcal{W}\) along
\(\PP\mathcal{W}'\) is canonically the \(\PP^2\)-bundle over \(S^\nu\)
associated with \(\tilde{\mathcal{W}}\) formed in the pullback diagram
\[
\begin{tikzcd}
0 \rar
& \tilde\varphi_+^*\mathcal{W}' \rar \dar[equal]
& \tilde{\mathcal{W}} \rar \dar[hook]
& \sO_{\tilde\varphi_+}(-1) \rar \dar[hook, "\mathrm{eu}_{\tilde\varphi_+}"]
& 0 \\
0 \rar
& \tilde\varphi_+^*\mathcal{W}' \rar
& \tilde\varphi_+^*\mathcal{W} \rar
& \tilde\varphi_+^*\mathcal{W}'' \rar
& 0
\end{tikzcd}
\]
and so the exceptional divisor of \(\PP\tilde{\mathcal{W}} \to \PP\mathcal{W}\)
is the subbundle \(\PP(\tilde\varphi_+^*\mathcal{W}') \subset \PP\tilde{\mathcal{W}}\).
The inverse image
\(X_{\PP\tilde{\mathcal{W}}} \coloneqq X_{\PP\mathcal{W}} \times_{\PP\mathcal{W}} \PP\tilde{\mathcal{W}}\)
of \(X_{\PP\mathcal{W}}\) along this blowup is the bundle of
\(q\)-bic curves over \(S^\nu\) defined by the \(q\)-bic form
\(
\beta_{\tilde{\mathcal{W}}} \colon \tilde{\mathcal{W}}^{[1]} \otimes \tilde{\mathcal{W}}
\to \sO_{S^\nu}
\)
obtained by restricting \(\tilde\varphi^*_+\beta_{\mathcal{W}}\).
Since \(\mathcal{W}^\perp = \mathcal{W}'^{[1]}\) by
\parref{nodal-tangent-pencil-types}, the diagram above implies that
\(\tilde\varphi_+^*\mathcal{W}'^{[1]} = \tilde{\mathcal{W}}^\perp\). Thus by
\citeForms{1.5}, there is an exact sequence
\[
0 \to
\mathcal{K} \to
\tilde{\mathcal{W}} \xrightarrow{\beta_{\tilde{\mathcal{W}}}}
\tilde{\mathcal{W}}^{[1],\vee} \to
\tilde\varphi_+^*\mathcal{W}'^{[1],\vee} \to
0
\]
where \(\mathcal{K}\) is the rank \(2\) subbundle defining the
\(\PP^1\)-bundle \(\tilde{X}_{\PP\mathcal{W}} \to S^\nu\). So since
\(\nu^*\mathcal{S} = \mathcal{K}\),
\begin{align*}
\nu^*\sO_S(1)
\cong \nu^*\det(\mathcal{S})^\vee
\cong \det(\mathcal{K})^\vee
& \cong \det(\tilde{\mathcal{W}})^{\vee,\otimes q+1} \otimes \tilde\varphi^*_+\det(\mathcal{W}')^{\otimes q} \\
& \cong \tilde\varphi^*_+\det(\mathcal{W}')^\vee \otimes \sO_{\tilde\varphi_+}(q+1) \\
& \cong \tilde\varphi^*_+\sO_C(1) \otimes \sO_{\tilde\varphi_+}(q+1) \otimes L_+^\vee
\end{align*}
since \(\mathcal{W}' = \sO_C(-1) \oplus L_{+,C}\).
\end{proof}

The following statement summarizes \parref{nodal-projection-from-line-tangent}
and \parref{nodal-normalization} and gives a direct geometric relationship
between the maps \(\tilde\varphi_+\) and \(\varphi_-\). For this, let
\(\phi_C \colon C \to C\) be the endomorphism that sends a point \(x\) to the
residual intersection point between \(C\) and its tangent line at \(x\), as
described in \parref{generalities-hermitian}.

\begin{Corollary}\label{nodal-nu-and-F}
There is a commutative diagram of morphisms
\[
\begin{tikzcd}
S^\nu \rar["\nu"'] \dar["\tilde\varphi_+"'] & S \dar["\varphi_-"] \\
C \rar["\phi_C"] & C\punct{.}
\end{tikzcd}
\]
\end{Corollary}

\subsectiondash{Conductors}\label{nodal-conductors}
Consider the conductor ideals associated with the normalization
\(\nu \colon S^\nu \to S\):
\[
\operatorname{cond}_{\nu,S} \coloneqq
\mathcal{A}\!\mathit{nn}_{\sO_S}(\nu_*\sO_{S^\nu}/\sO_S) \subset \sO_S
\quad\text{and}\quad
\operatorname{cond}_{\nu,S^\nu} \coloneqq
\nu^{-1}\operatorname{cond}_{\nu,S} \cdot \sO_{S^\nu} \subset
\sO_{S^\nu}.
\]
The corresponding conductor subschemes \(D \subset S\) and
\(D^\nu \subset S^\nu\) are thickenings of the curves \(C_+\) and \(C_+^\nu\),
respectively, and fit into a commutative diagram
\[
\begin{tikzcd}
D^\nu \rar["\nu"'] \dar["\tilde\varphi_+"'] & D \dar["\varphi_-"] \\
C \rar["\phi_C"] & C
\end{tikzcd}
\]
where the vertical maps are finite by properness of \(S\) and \(S^\nu\) over
\(C\). Algebraically, the conductor ideal of \(S\) is characterized as the
largest ideal of \(\sO_S\) which is also an ideal of \(\nu_*\sO_{S^\nu}\), so
there is a commutative diagram of exact sequences of sheaves on \(S\):
\[
\begin{tikzcd}[row sep=1.5em]
& \operatorname{cond}_{\nu,S} \dar[hook] \rar["\cong"']
& \nu_*\operatorname{cond}_{\nu,S^\nu} \dar[hook] \\
0 \rar
& \sO_S \rar["\nu^\#"] \dar[two heads]
& \nu_*\sO_{S^\nu} \rar \dar[two heads]
& \nu_*\sO_{S^\nu}/\sO_S \rar \dar["\cong"]
& 0 \\
0 \rar
& \sO_D \rar["\nu^\#"]
& \nu_*\sO_{D^\nu} \rar
& \nu_*\sO_{D^\nu}/\sO_D \rar
& 0\punct{.}
\end{tikzcd}
\]

Duality theory for \(\nu \colon S^\nu \to S\) identifies the conductor ideal
\(\operatorname{cond}_{\nu,S^\nu}\) with the relative dualizing sheaf
\(\omega_{S^\nu/S} \cong \omega_{S^\nu} \otimes \nu^*\omega_S^\vee\), compare
\citeSP{0FKW} and \cite[Proposition 2.3]{Reid}.

\begin{Proposition}\label{nodal-compute-conductor}
The conductor ideal of \(S^\nu\) is isomorphic to
\[
\operatorname{cond}_{\nu,S^\nu} \cong
\sO_{\tilde\varphi_+}(-\delta-1) \otimes (L_+^{\otimes 2q-1} \otimes L_-)
\quad\text{where}\;\delta \coloneqq 2q^2 - q - 2.
\]
In particular, the conductor subscheme \(D^\nu\) is the \(\delta\)-order
neighbourhood of \(C_+^\nu\).
\end{Proposition}

\begin{proof}
Since \(S^\nu\) is the \(\PP^1\)-bundle on
\(\mathcal{W}'' \cong \mathcal{T}_C(-1) \oplus L_{-,C}\), as in
\parref{nodal-normalization}, the relative Euler sequence gives
\[
\omega_{S^\nu} \cong
\omega_{S^\nu/C} \otimes \tilde\varphi_+^*\omega_C \cong
\sO_{\tilde\varphi_+}(-2) \otimes
\tilde\varphi^*_+(\omega_C^{\otimes 2} \otimes \sO_C(1)) \otimes L_-^\vee.
\]
By \parref{generalities-lines},
\(\omega_S \cong \sO_S(2q - 3) \otimes (L_+\otimes L_-)^{\vee,\otimes 2}\),
so \parref{nodal-normalization} gives
\begin{align*}
\nu^*\omega_S
& \cong \nu^*\sO_S(2q-3) \otimes (L_+ \otimes L_-)^{\vee, \otimes 2} \\
& \cong
\tilde\varphi_+^*\sO_C(2q-3) \otimes
\sO_{\tilde\varphi_+}\big((2q-3)(q+1)\big) \otimes
L_+^{\vee, \otimes 2q-1} \otimes L_-^{\vee,\otimes 2}.
\end{align*}
Since \(C\) is a plane curve of degree \(q+1\),
\(\omega_C^{\otimes 2} \otimes \sO_C(1) \cong \sO_C(2q-3)\), and so
\[
\omega_{S^\nu/S}
\cong \omega_{S^\nu} \otimes \nu^*\omega_S^\vee
\cong \sO_{\tilde\varphi_+}(-\delta-1) \otimes (L_+^{\otimes 2q-1} \otimes L_-).
\qedhere
\]
\end{proof}

\subsectiondash{The sheaf \texorpdfstring{\(\mathcal{F}\)}{F}}\label{nodal-F}
Let \(\mathcal{D}\) and \(\mathcal{D}^\nu\) be the coherent \(\sO_C\)-algebras
by pushing structure sheaves along the finite morphisms
\(\varphi_- \colon D \to C\) and
\(\phi_C \circ \tilde\varphi_+ \colon D^\nu \to C\), respectively.
Restricting \(\nu\) to the conductors induces an injective map
\(\nu^\# \colon \mathcal{D} \to \mathcal{D}^\nu\) of \(\sO_C\)-algebras.
Its cokernel
\[
\mathcal{F}
\coloneqq \mathcal{D}/\mathcal{D}^\nu
\cong \varphi_{-,*}(\nu_*\sO_{D^\nu}/\sO_D)
\cong \varphi_{-,*}(\nu_*\sO_{S^\nu}/\sO_S),
\]
identifications arising from the diagram of \parref{nodal-conductors}, is an
\(\sO_C\)-module related to \(S\) as follows:

\begin{Lemma}\label{normalize-splitting}
There is an exact sequence of finite locally free \(\sO_C\)-modules
\[
0 \to
\sO_C \to
\phi_{C,*} \sO_C \to
\mathcal{F} \to
\mathbf{R}^1\varphi_{-,*}\sO_S \to
0
\]
in which the map \(\mathcal{F} \to \mathbf{R}^1\varphi_{-,*}\sO_S\) splits.
\end{Lemma}

\begin{proof}
Push the middle row in the diagram of \parref{nodal-conductors} along
\(\varphi_-\) to obtain the sequence, noting that
\(\varphi_{-,*}\sO_S \cong \tilde\varphi_{+,*}\sO_{S^\nu} \cong \sO_C\) by
\parref{smooth-cone-situation-pushforward}, and that
\(\mathbf{R}^1\tilde\varphi_{+,*}\sO_{S^\nu} = 0\) since \(\tilde\varphi_+\) is
a projective bundle over \(C\). Observe that \(\mathcal{F}\) is locally free
since it is an extension of the locally free \(\sO_C\)-modules
\(\mathbf{R}^1\varphi_{-,*}\sO_S\) is and \(\phi_{C,*}\sO_C/\sO_C\): the former
is locally free by \parref{smooth-cone-situation-pushforward}; for the latter,
this follows by Kunz's Theorem \cite[Theorem 2.1]{Kunz} or \citeSP{0EC0}, since
\(\phi_C\) is, up to an automorphism,  the \(q^2\)-Frobenius of the regular
curve \(C\) by \parref{generalities-hermitian}. The \(\mathbf{G}_m\)-action
from \parref{nodal-automorphisms} provides \(\mathcal{D}\) and
\(\mathcal{D}^\nu\) with gradings such that \(\sO_C\) and \(\phi_{C,*}\sO_C\),
being the constant functions, make up the degree \(0\) components. Therefore
the positively graded components of \(\mathcal{F}\) map isomorphically to
\(\mathbf{R}^1\varphi_{-,*}\sO_S\) and provide the desired splitting.
\end{proof}

The importance of \(\mathcal{F}\) is in the following relation with the
cohomology of the structure sheaf of \(S\):

\begin{Proposition}\label{normalize-cohomology}
The cohomology of \(\sO_S\) is given by
\[
\mathrm{H}^i(S,\sO_S) \cong
\begin{dcases}
\mathrm{H}^0(C,\sO_C) & \text{if}\; i = 0, \\
\mathrm{H}^0(C,\mathcal{F}) & \text{if}\; i = 1,\;\text{and} \\
\mathrm{H}^1(C,\mathcal{F})/\mathrm{H}^1(C,\sO_C) & \text{if}\; i = 2.
\end{dcases}
\]
\end{Proposition}

\begin{proof}
Consider the cohomology sequence associated with the exact sequence
\[
0 \to
\sO_S \to
\nu_*\sO_{S^\nu} \to
\nu_*\sO_{S^\nu}/\sO_S \to
0.
\]
Note that \(\varphi_{-,*}\sO_S \cong \sO_C\) by
\parref{smooth-cone-situation-pushforward} and, for each \(i = 0,1,2\),
\[
\mathrm{H}^i(S,\nu_*\sO_{S^\nu}/\sO_S) \cong
\mathrm{H}^i(C,\mathcal{F})
\quad\text{and}\quad
\mathrm{H}^i(S^\nu,\sO_{S^\nu}) \cong
\mathrm{H}^i(C,\sO_C)
\]
since \(\nu_*\sO_{S^\nu}/\sO_S\) is supported on \(D\) and \(D \to C\) is
affine, and and the fact that \(\tilde\varphi_+\colon S^\nu \to C\) is a
projective bundle. Thus
\(\mathrm{H}^0(S,\sO_S) \cong \mathrm{H}^0(S^\nu,\sO_{S^\nu}) \cong \mathrm{H}^0(C,\sO_C)\)
and there is an exact sequence
\[
0 \to
\mathrm{H}^0(C,\mathcal{F}) \xrightarrow{a}
\mathrm{H}^1(S,\sO_S) \xrightarrow{b}
\mathrm{H}^1(S,\nu_*\sO_{S^\nu}) \to
\mathrm{H}^1(C,\mathcal{F}) \to
\mathrm{H}^2(S,\sO_S) \to 0.
\]
The result will follow upon verifying that
\(b \colon \mathrm{H}^1(S,\sO_S) \to \mathrm{H}^1(S,\nu_*\sO_{S^\nu})\)
vanishes. Since \(\varphi_{-,*}\sO_S \cong \sO_C\), the Leray spectral sequence
gives a short exact sequence
\[
0 \to
\mathrm{H}^1(C,\sO_C) \xrightarrow{c}
\mathrm{H}^1(S,\sO_S) \xrightarrow{d}
\mathrm{H}^0(C,\mathbf{R}^1\varphi_{-,*}\sO_S) \to
0.
\]
The splitting in \parref{normalize-splitting} implies that the
composite
\(
d \circ a \colon
\mathrm{H}^0(C,\mathcal{F}) \to
\mathrm{H}^0(C,\mathbf{R}^1\varphi_{-,*}\sO_S)
\)
is a surjection. So exactness of the long sequence means it remains to show
that
\(
b \circ c \colon
\mathrm{H}^1(C,\sO_C) \to
\mathrm{H}^1(S,\nu_*\sO_{S^\nu})
\)
vanishes. Pushing down to \(C\) along \(\varphi_-\) and applying
\parref{nodal-nu-and-F} shows that this is
\[
\phi_C \colon
\mathrm{H}^1(C,\sO_C) \to
\mathrm{H}^1(C,\phi_{C,*}\sO_C)
\]
which, as in \parref{normalize-splitting}, is the map induced by the
\(q^2\)-power Frobenius up to an automorphism, and this is the zero map
by \parref{generalities-frobenius-vanishes}.
\end{proof}

Structure sheaf cohomology of \(S\) is therefore reduced to that of
\(\mathcal{F}\), which will be computed when \(q = p\) in
\parref{cohomology-dimension-result}. Combined with the Euler characteristic
computation in \parref{generalities-chi-OS}, this gives Theorem
\parref{intro-smooth-cohomology}. In turn, cohomology of \(\mathcal{F}\) is
determined by describing its structure, and this is achieved in the next
section via the following duality relationship with the algebra
\(\mathcal{D}\):

\begin{Proposition}\label{conductor-dual}
There is a canonical isomorphism of graded \(\sO_C\)-modules
\[
\mathcal{F} \cong
\mathcal{D}^\vee \otimes \sO_C(-q+1) \otimes L_+^{\otimes 2q-1} \otimes L_-^{\otimes 2}.
\]
\end{Proposition}

\begin{proof}
Applying \(\mathbf{R}\mathcal{H}\!\mathit{om}_{\sO_S}(-,\sO_S)\) to the
ideal sheaf sequence of the conductor \(D \hookrightarrow S\) yields a triangle
in the derived category of \(S\)
\[
\mathbf{R}\mathcal{H}\!\mathit{om}_{\sO_S}(\sO_D,\sO_S) \to
\sO_S \to
\nu_*\sO_{S^\nu} \xrightarrow{+1}
\]
since \(\mathbf{R}\mathcal{H}\!\mathit{om}_{\sO_S}(\sO_S,\sO_S) = \sO_S\)
and
\(\mathbf{R}\mathcal{H}\!\mathit{om}_{\sO_S}(\operatorname{cond}_{\nu,S},\sO_S) = \nu_*\sO_{S^\nu}\),
upon identifying \(\operatorname{cond}_{\nu,S}\) with \(\nu_*\omega_{S^\nu/S}\)
as in \parref{nodal-conductors} and using duality for \(\nu\).
The map \(\sO_S \to \nu_*\sO_{S^\nu}\) is dual to evaluation at \(1\), and
hence is the \(\sO_S\)-module map determined by \(1 \mapsto 1\); in other
words, this is the map \(\nu^\#\), and so
\[
\nu_*\sO_{S^\nu}/\sO_S \cong
\mathbf{R}\mathcal{H}\!\mathit{om}_{\sO_S}(\sO_D,\sO_S)[1]
\]
in the derived category of \(S\). Applying \(\mathbf{R}\varphi_{-,*}\) and
applying relative duality for \(\varphi_- \colon S \to C\) yields
\[
\mathcal{F}
\cong \mathbf{R}\varphi_{-,*}\mathbf{R}\mathcal{H}\!\mathit{om}_{\sO_S}(\sO_D,\sO_S)[1]
\cong \mathbf{R}\mathcal{H}\!\mathit{om}_{\sO_C}(\mathbf{R}\varphi_{-,*}(\sO_D \otimes \omega_{\varphi_-}), \sO_C)[1]
\]
where \(\omega_{\varphi_-} = (\omega_S \otimes \varphi_-^*\omega_C^\vee)[1]\).
By \parref{generalities-lines},
\[
\sO_D \otimes \omega_S \cong
\sO_S(2q-3)\rvert_D \otimes L_+^{\vee, \otimes 2} \otimes L_-^{\vee, \otimes 2} \cong
\varphi_-^*\sO_C(2q-3)\rvert_D \otimes
L_+^{\vee,\otimes 2q-1} \otimes L_-^{\vee, \otimes 2}
\]
where \(\sO_S(1)\rvert_D = \varphi_-^*\sO_C(1)\rvert_D \otimes L_+^\vee\) since
\(\varphi_- \colon D \to C\) is induced by the line subbundle of
\(\mathcal{S}\rvert_D\) obtained by intersecting with \((L_- \oplus W)_D\) by
\parref{cone-situation-rational-map-S}, so there is a short exact sequence
\[
0 \to
\varphi_-^*\sO_C(-1)\rvert_D \to
\mathcal{S}\rvert_D \to
L_{+,D} \to
0,
\]
and taking determinants yields the desired identification. Combining with
\(\omega_C \cong \sO_C(q-2)\) gives
\[
\mathbf{R}\varphi_{-,*}(\sO_D \otimes \omega_{\varphi_-}) =
(\mathbf{R}\varphi_{-,*}\sO_D) \otimes \sO_C(q-1) \otimes L_+^{\vee,\otimes 2q-1} \otimes L_-^{\vee,\otimes 2}[1].
\]
Since \(D \to C\) is of relative dimension \(0\),
\(\mathbf{R}\varphi_{-,*}\sO_D = \varphi_{-,*}\sO_D = \mathcal{D}\), yielding
the result.
\end{proof}

\subsectiondash{The algebra \(\mathcal{D}^\nu\)}\label{cohomology-D-nu}
Before turning to \(\mathcal{D}\), consider the simpler algebra
\(\mathcal{D}^\nu\) associated with
\(\phi_C \circ \tilde\varphi_+ \colon D^\nu \to C\). Since \(D^\nu\) is
disjoint from the curve \(C_-^\nu = \PP L_{-,C}\) parameterizing lines through
\(x_- \in X\), it is contained in the \(\mathbf{A}^1\)-bundle over \(C\) given by
\[
S^{\nu,\circ} \coloneqq
S^\nu \setminus C_-^\nu
= \PP(\mathcal{T}_C(-1) \oplus L_{-,C}) \setminus \PP L_{-,C}
\cong \mathbf{A}(\Omega_C^1(1) \otimes L_-),
\]
the identification obtained by taking graphs of a linear function
\(\mathcal{T}_C(-1) \to L_{-,C}\). Then \parref{nodal-compute-conductor} means
that \(D^\nu\) is the \(\delta\)-order neighbourhood  of the zero section
\(C_+^\nu\), so this identifies the graded \(\sO_C\)-module underlying the
algebra of \(D^\nu\) as
\[
\mathcal{D}^\nu =
(\phi_C \circ \tilde\varphi_+)*\sO_{D^\nu}
\cong
\phi_{C,*}\big(\bigoplus\nolimits_{i = 0}^\delta (\mathcal{T}_C(-1) \otimes L_-^\vee)^{\otimes i}\big).
\]

\section{Structure of \texorpdfstring{\(\mathcal{D}\)}{D} and \texorpdfstring{\(\mathcal{F}\)}{F}}\label{section-D}
The purpose of this section is to describe the structure of the sheaves
\(\mathcal{D}\) and---especially!---\(\mathcal{F}\) in terms of more familiar
sheaves on the curve \(C\). To give the main statement regarding
\(\mathcal{F}\), some notation: View the restricted Euler sequence
\[
0 \to
\Omega_{\PP W}^1\rvert_C \to
W^\vee \otimes \sO_C(-1) \to
\sO_C \to
0
\]
as a \(2\)-step filtration on \(W^\vee \otimes \sO_C(-1)\), with the sheaf of
differentials the \(0\)-th filtered piece. This induces a \((d+1)\)-step
filtration on the symmetric powers \(\Sym^d(W^\vee) \otimes \sO_C(-d)\) with
graded pieces
\[
\gr_i\big(\Sym^d(W^\vee) \otimes \sO_C(-d)\big)
\cong \Sym^{d-i}(\Omega_{\PP W}^1\rvert_C)
\;\;\text{for}\; 0 \leq i \leq d.
\]
Generally, given any locally free sheaf \(\mathcal{E}\) and \(d \geq q\),
\(\Sym^d(\mathcal{E})\) contains a subsheaf
\(\mathcal{E}^{[1]} \otimes \Sym^{d-q}(\mathcal{E})\) consisting of products of
\(q\)-powers and monomials of degree \(d-q\). Write
\[
\Sym^d_{\mathrm{red}}(\mathcal{E}) \coloneqq
\Sym^d(\mathcal{E})/(\mathcal{E}^{[1]} \otimes \Sym^{d-q}(\mathcal{E}))
\]
for the quotient by this subsheaf, with the convention that
\(\Sym^d_{\mathrm{red}} = \Sym^d\) when \(d < q\). Finally, write
\(\Div^d_{\mathrm{red}}(\mathcal{E}) \coloneqq \Sym^d_{\mathrm{red}}(\mathcal{E}^\vee)^\vee\)
for the reduced \(d\)-th divided power of \(\mathcal{E}\). The result is:

\begin{Proposition}\label{D-F-duality}
The graded \(\sO_C\)-module \(\mathcal{F}\) carries a \(q\)-step filtration
such that
\[
\Fil_0\mathcal{F}
\cong
\bigoplus\nolimits_{b = 0}^{q-2}
\bigoplus\nolimits_{a = 0}^{q-1}
\Div^{q-2-b}(\mathcal{T}) \otimes \sO_C(-a) \otimes
L_+^{\otimes b} \otimes
L_-^{\vee, \otimes a}.
\]
For each \(0 \leq b \leq 2q-3\), there is a canonical short exact sequence
of filtered bundles
\[
0 \to
\Div^{2q-3-b}_{\mathrm{red}}(W) \otimes \sO_C \to
\mathcal{F}_{bq + q -1} \to
\Div^{q-3-b}(W) \otimes \sO_C(q-1) \to
0.
\]
There is a degree \(-q-1\) map \(\partial \colon \mathcal{F} \to \mathcal{F}\)
such that, for each \(0 \leq i \leq q-1\) and \(0 \leq d \leq \delta-q-1\),
\(\partial(\Fil_i\mathcal{F}) \subseteq \Fil_{i-1}\mathcal{F}\) and
\(\gr_i\partial \colon \Fil_i\mathcal{F}_{d+q+1} \to \Fil_{i-1}\mathcal{F}_d\)
is an isomorphism if \(p \nmid i\) and zero otherwise.
\end{Proposition}

This is proved in \parref{D-F-proof} at the end of the section.

\subsectiondash{Affine bundles}\label{D-affine-bundles}
The duality relation \parref{conductor-dual} relates \(\mathcal{F}\) with
\(\mathcal{D}\), and the latter is a quotient of coordinate rings of schemes
affine over \(C\): namely, \(S^\circ \coloneqq S \setminus C_-\) and
\(T^\circ\) as in \parref{points-of-T-geometrically} with respect to the smooth
cone situation \((X,x_-,\PP L_-^{[1],\perp})\). These lie in affine space
bundles
\[
\mathbf{A}_1 \coloneqq \PP\mathcal{V}_1 \setminus \PP L_{-,C}, \qquad
\mathbf{A}_2 \coloneqq \PP\mathcal{V}_2 \setminus \PP\mathcal{T},
\qquad
\mathbf{B} \coloneqq \PP\mathcal{V} \setminus \PP(\mathcal{T}_{\pi_1}(-1,0))\rvert_{\mathbf{A}}
\]
with \(\mathcal{V}_1\) and \(\mathcal{V}_2\) as in \parref{cone-situation-P},
and \(\mathcal{V}\) is as in \parref{cone-situation-blowup}; set
\(\mathbf{A} \coloneqq \mathbf{A}_1 \times_C \mathbf{A}_2\). Comparing the
diagram in \parref{cone-situation-P} with the description of the boundaries
\(T \setminus T^\circ\) from \parref{cone-situation-boundary} implies that
there is a commutative diagram of affine schemes over \(C\) given by
\[
\begin{tikzcd}
D \rar[hook] \ar[dr]
& S^\circ \rar[hook] \dar
& \mathbf{B} \dar["\rho"] \ar[dd,bend left=45, "\varphi"] \\
& T^\circ \rar[hook] \ar[dr]
& \mathbf{A}\dar["\pi"] \\
&& C\punct{.}
\end{tikzcd}
\]
Observe that the relative Euler sequences for these affine bundles give
canonical isomorphisms
\begin{align*}
\sO_{\mathbf{A}}(-1,0) & \cong \pi^*\sO_C(-1), &
\mathcal{T}_{\pi_2}(0,-1)\rvert_{\mathbf{A}} & \cong \pi^*\mathcal{T}, &
\sO_\rho(-1) & \cong \rho^*\sO_{\mathbf{A}}(0,-1) \cong L_{+,\mathbf{B}}, \\
\mathcal{T}_{\pi_1}(-1,0)\rvert_{\mathbf{A}} & \cong L_{-,\mathbf{A}}, &
\sO_{\mathbf{A}}(0,-1) & \cong L_{+,\mathbf{A}}, &
\mathcal{T}_\rho(-1)\rvert_{\mathbf{B}} & \cong \rho^*\mathcal{T}_{\pi_1}(-1,0)\rvert_{\mathbf{B}} \cong L_{-,\mathbf{B}}.
\end{align*}
This identifies the \(\sO_C\)-algebras
\(\mathcal{A} \coloneqq \pi_*\sO_{\mathbf{A}}\) and
\(\mathcal{B} \coloneqq \varphi_*\sO_{\mathbf{B}}\) as follows:

\begin{Lemma}\label{cohomology-A-B}
The \(q\)-bic form \(\beta\) induces an isomorphism
\[
\mathcal{A} \cong
\Sym^*(\sO_C(-1) \otimes L_-^\vee) \otimes
\Sym^*(\Omega_{\PP W}^1(1)\rvert_C \otimes L_+),
\]
and endows the \(\mathcal{A}\)-algebra \(\mathcal{B}\) with an increasing
filtration whose associated graded pieces are
\[
\gr_i\mathcal{B} \coloneqq
\Fil_i\mathcal{B}/\Fil_{i-1}\mathcal{B} \cong
\mathcal{A} \otimes (L_-^\vee \otimes L_+)^{\otimes i}
\quad\text{for all}\; i \in \mathbf{Z}_{\geq 0}.
\]
\end{Lemma}

\begin{proof}
The splittings
\(\mathcal{V}_1 \cong \sO_C(-1) \oplus L_{-,C}\) and
\(\mathcal{V}_2 \cong \mathcal{T} \oplus L_{+,C}\) as in
\S\parref{section-smooth-cone-situation} give canonical relative projective
coordinates on \(\PP\mathcal{V}_i\) over \(C\), and identifies the associated
affine bundles as
\[
\mathbf{A}
\cong \mathbf{A}(\sO_C(1) \otimes L_-) \times_C
\mathbf{A}(\mathcal{T} \otimes L_+^\vee).
\]
This identifies \(\mathcal{A} = \pi_*\sO_{\mathbf{A}}\) as claimed.

For \(\mathcal{B}\), begin with the affine bundle
\(\rho \colon \mathbf{B} \to \mathbf{A}\) obtained as the complement of
\(\PP(\mathcal{T}_{\pi_1}(-1,0))\) in \(\PP\mathcal{V}\) over \(\mathbf{A}\).
Dualizing the short exact sequence in \parref{cone-situation-V-sequence}
and restricting to \(\mathbf{A}\) gives a sequence
\[
0 \to
\sO_{\mathbf{A}}(0,1) \to
\mathcal{V}^\vee\rvert_{\mathbf{A}} \to
\Omega^1_{\pi_1}(1,0)\rvert_{\mathbf{A}} \to
0.
\]
View \(\mathcal{V}^\vee\rvert_{\mathbf{A}}\) as the linear functions on
\(\PP\mathcal{V}\rvert_{\mathbf{A}}\) over \(\mathcal{A}\).
A local generator for the subbundle \(\sO_{\mathbf{A}}(0,1)\) is then a linear
equation defining \(\PP(\mathcal{T}_{\pi_1}(-1,0)\rvert_{\mathbf{A}})\), and so
becomes invertible on \(\mathbf{B}\). Therefore
\[
\rho_*\sO_{\mathbf{B}} \cong
\colim_n \Sym^n(\mathcal{V}^\vee(0,-1))\rvert_{\mathbf{A}}
\]
where the transition maps are induced by multiplication by a local
generator for the subbundle
\(\sO_{\mathbf{A}} \hookrightarrow \mathcal{V}^\vee(0,-1)\rvert_{\mathbf{A}}\).
The \(2\)-step filtration on \(\mathcal{V}^\vee(0,-1)\rvert_{\mathbf{A}}\)
starting with \(\sO_{\mathbf{A}}\) and followed by the entire bundle induces
filtrations on the symmetric powers, compatible with the transition maps, whence
a filtration on \(\rho_*\sO_{\mathbf{B}}\) with graded pieces
\[
\gr_i\rho_*\sO_{\mathbf{B}}
\cong \Omega^1_{\pi_1}(1,-1)\rvert_{\mathbf{A}}^{\otimes i}
\cong \sO_{\mathbf{A}} \otimes (L_-^\vee \otimes L_+)^{\otimes i}
\;\;\text{for all \(i \in \mathbf{Z}_{\geq 0}\)}
\]
upon applying the identifications of \parref{D-affine-bundles}. Pushing along
\(\pi \colon \mathbf{A} \to C\) then gives the result.
\end{proof}

\subsectiondash{Symmetries}\label{D-symmetries}
Consider the linear algebraic subgroup of \(\GL(V)\) which preserves the
decomposition \(V = W \oplus U\), fixes the point \(x_- = \PP L_-\), and which
acts via automorphisms on \(C\):
\[
\AutSch(L_- \subset U) \times \AutSch(W,\beta_W) \cong
\Set{
\begin{pmatrix}
  \lambda^{-1}_- & \epsilon \\
  0              & \lambda_+
\end{pmatrix}
\in \mathbf{GL}(L_- \oplus L_+)}
\times \mathrm{U}_3(q).
\]
Via its natural linear action on \(V\), it acts on each of the sheaves
\(\mathcal{V}_1\), \(\mathcal{V}_2\), and \(\mathcal{V}\), and preserves the
subbundles excised in defining the tower of affine bundles
\(\mathbf{B} \to \mathbf{A} \to C\). As such, the \(\sO_C\)-algebras
\(\mathcal{A}\) and \(\mathcal{B}\) are equivariant for this group.
Two pieces of structure now leap to the forefront:

First, the action of the maximal torus \((\lambda_-^{-1}, \lambda_+)\) endows
\(\mathcal{A}\) and \(\mathcal{B}\) with a bigrading, normalized so that
\(L_-^{\vee, \otimes a} \otimes L_+^{\otimes b}\) has weight
\((a,b) \in \mathbf{Z}_{\geq 0}^2\). The bigraded pieces are as follows:

\begin{Lemma}\label{cohomology-bigrading}
The form \(\beta\) induces isomorphisms of filtered bundles
\[
\mathcal{B}_{(a,0)} \cong
\mathcal{A}_{(a,0)} \cong
\sO_C(-a),
\;\;
\mathcal{B}_{(0,b)} \cong
\mathcal{A}_{(0,b)} \cong
\Sym^b(\Omega_{\PP W}(1))\rvert_C,
\;\;
\mathcal{B}_{(1,1)} \cong W^\vee \otimes \sO_C(-1),
\]
and
\(\mathcal{B}_{(a,b)} \cong \Fil_a\big(\Sym^b(W^\vee) \otimes \sO_C(-a)\big)\)
for all \(a,b \in \mathbf{Z}_{\geq 0}\).
\end{Lemma}

\begin{proof}
Matching weights in \parref{cohomology-A-B} identifies the bigraded pieces
of weights \((a,0)\) and \((0,b)\) as claimed, and shows that the \((1,1)\)
piece of \(\mathcal{B}\) sits as an extension
\[
0 \to
\Omega^1_{\PP W}\rvert_C \to
\mathcal{B}_{(1,1)} \to
\sO_C \to
0.
\]
Begin by considering the extension class of \(\mathcal{V}\) in the sequence
\parref{cone-situation-V-sequence}: By its construction in
\parref{cone-situation-blowup}, the class \([\mathcal{V}]\) is the image under
the map, induced by pullback and pushforward along tautological maps,
\[
\Ext^1_{\PP}(\pi^*\mathcal{V}_2, \pi^*\mathcal{V}_1) \to
\Ext^1_{\PP}(\sO_{\PP}(0,-1), \pi^*\mathcal{V}_1) \to
\Ext^1_{\PP}(\sO_{\PP}(0,-1), \mathcal{T}_{\pi_1}(-1,0))
\]
of the class of \([V_\PP]\) from the top sequence in the diagram
\parref{cone-situation-P}. The kernel of the two maps in question are
\(\Ext^1_\PP(\mathcal{T}_{\pi_2}(0,-1), \pi^*\mathcal{V}_1)\) and
\(\Ext^1_\PP(\sO_{\PP}(0,-1), \sO_\PP(-1,0))\), both of which vanish. Therefore
the class \([\mathcal{V}]\) is nonzero, and is the image of the class of the
dual Euler sequence on \(\PP W\). Since \(\mathcal{B}_{(1,1)}\) is obtained by
pushing a twist of the dual of the sequence for \(\mathcal{V}\) in
\parref{cone-situation-V-sequence}, its extension class is that of the Euler
sequence, so \(\mathcal{B}_{(1,1)} \cong W^\vee \otimes \sO_C(-1)\).

The remaining bigraded pieces are obtained via multiplication: Since
\(\mathcal{B}\) is locally a polynomial algebra with two generators of
degree \((0,1)\), one of degree \((1,0)\), and one of degree \((1,1)\), it
follows that the multiplication maps, for \(d,e \in \mathbf{Z}_{\geq 0}\),
\[
\Sym^d(\mathcal{B}_{(1,1)}) \to \mathcal{B}_{(d,d)},
\quad
\mathcal{B}_{(d,d)} \otimes \mathcal{B}_{(e,0)} \to \mathcal{B}_{(d+e,d)},
\quad
\mathcal{B}_{(d,d+e)} \otimes \mathcal{B}_{(e,0)} \to \mathcal{B}_{(d+e,d+e)}
\]
are isomorphisms in the first two cases, and an isomorphism onto the \(d\)-th
filtered piece of \(\mathcal{B}_{(d+e,d+e)}\). Combined with the identification
of the low degree pieces completes the proof.
\end{proof}

Second, the additive group \(\mathbf{G}_a\) acts through \(\epsilon\), the
unipotent radical of \(\AutSch(L_- \subset U)\). This action sends \(L_+\)
to \(L_-\), so it is trivial on \(\mathcal{V}_1\) and \(\mathcal{V}_2\), whence
on \(\mathcal{A}\). Therefore \(\mathcal{B}\) is \(\mathcal{A}\)-linearly
\(\mathbf{G}_a\)-equivariant. This structure is described algebraically via an
\(\mathcal{A}\)-comodule structure
\[
\mathcal{B} \to \mathcal{B} \otimes \kk[\epsilon] \colon
z \mapsto \sum\nolimits_{j = 0}^\infty \partial_j(z) \otimes \epsilon^j
\]
where \(\mathbf{G}_a \coloneqq \Spec\kk[\epsilon]\), and the \(\partial_j \colon \mathcal{B} \to \mathcal{B}\)
are \(\mathcal{A}\)-linear maps such that a given local section \(z\) lies in the
kernel of all but finitely many \(\partial_j\), \(\partial_0 = \id\), and
\[
\partial_j \circ \partial_k = \binom{j+k}{j}\, \partial_{j+k}
\;\;\text{for all}\;j,k \in \mathbf{Z}_{\geq 0}.
\]
See \cite[I.7.3, I.7.8, and I.7.12]{Jantzen:RAGS} for details. Of particular
importance is the operator \(\partial \coloneqq \partial_1\), and its salient
features are as follows:

\begin{Lemma}\label{D-unipotent}
The map \(\partial \colon \mathcal{B} \to \mathcal{B}\) is of bidegree
\((-1,-1)\) and satisfies
\[
\partial(\Fil_i \mathcal{B}) \subseteq \Fil_{i-1}\mathcal{B}
\;\;\text{for each}\;i \in \mathbf{Z}_{\geq 0}.
\]
The associated graded map
\(\gr_i\partial \colon \gr_i\mathcal{B} \to \gr_{i-1}\mathcal{B}\)
is an isomorphism if \(p \nmid i\) and is zero otherwise.
\end{Lemma}

\begin{proof}
The action of \(\mathbf{G}_a\) on \(V\) corresponds in degree \(1\) to the
linear map \(\partial \colon V \to V\) which is an isomorphism between the
components \(L_+ \to L_-\), and zero elsewhere. Tracing through the
construction of \(\mathcal{V}\) from \parref{cone-situation-P} and
\parref{cone-situation-blowup}, it is straightforward that upon restricting
to \(\mathbf{A}\) and in using the identifications in \parref{D-affine-bundles}
to write the short exact sequence in \parref{cone-situation-V-sequence} as
\[
0 \to
\sO_{\mathbf{A}} \to
\mathcal{V}^\vee(0,-1)\rvert_{\mathbf{A}} \to
\sO_{\mathbf{A}} \otimes (L_-^\vee \otimes L_+) \to
0,
\]
\(\partial\) acts on \(\mathcal{V}^\vee(0,-1)\rvert_{\mathbf{A}}\) as \(0\)
on the subbundle \(\sO_{\mathbf{A}}\), and sends the quotient
\(\sO_{\mathbf{A}} \otimes (L_-^\vee \otimes L_+)\) isomorphically to the
subbundle via the isomorphism \(\partial \colon L_-^\vee \otimes L_+ \to \kk\).
Taking symmetric powers shows that
\[
\partial(\Fil_i \Sym^n(\mathcal{V}^\vee(0,-1))\rvert_{\mathbf{A}}) \subseteq
\Fil_{i-1} \Sym^n(\mathcal{V}^\vee(0,-1))\rvert_{\mathbf{A}}\;\;
\text{for each}\; 0 \leq i \leq n,
\]
and that the associated graded map
\(
\gr_i \partial \colon
\sO_{\mathbf{A}} \otimes (L_-^\vee \otimes L_+)^{\otimes i} \to
\sO_{\mathbf{A}} \otimes (L_-^\vee \otimes L_+)^{\otimes i-1}
\)
is multiplication by \(i\). Passing to the colimit as in
\parref{cohomology-A-B} and pushing along
\(\pi \colon \mathbf{A} \to C\) gives the lemma.
\end{proof}

\subsectiondash{Coordinate rings}\label{D-coordinate-rings}
The coordinate rings of \(T^\circ\) and \(S^\circ\) over \(C\)
as quotients of \(\mathcal{A}\) and \(\mathcal{B}\):
By \parref{cone-situation-closure}, \(T^\circ \subset \mathbf{A}\)
is the codimension \(2\) complete intersection cut out by the sections
\begin{align*}
v_1 & \coloneqq
\mathrm{eu}_{\pi_2}^\vee \circ \beta^\vee \circ \mathrm{eu}_{\pi_1}^{[1]}
\colon \sO_\PP(-q,-1) \to \sO_\PP \\
v_2 & \coloneqq
\mathrm{eu}_{\pi_2}^{[1],\vee} \circ \beta \circ \mathrm{eu}_{\pi_1}
\colon \sO_\PP(-1,-q) \to \sO_\PP
\end{align*}
restricted to \(\mathbf{A}\). By \parref{cone-situation-section-over-T-subbundle},
\(S^\circ\) is the hypersurface in \(\mathbf{B} \times_{\mathbf{A}} T^\circ\)
cut out by the restriction of the section
\[
v_3 \coloneqq
u_3^{-1}\beta_{\mathcal{V}_T}(\mathrm{eu}_\rho^{[1]},\mathrm{eu}_\rho) \colon
\sO_\rho(-q) \otimes \rho^*\sO_T(0,-1)\rvert_{\PP\mathcal{V}_T} \to
\sO_{\PP\mathcal{V}_T}.
\]
Pushing forward to \(C\) and using the identifications from
\parref{D-affine-bundles} then gives presentations
\begin{align*}
\pi_*\sO_{T^\circ} & \cong \coker\big(
  v_1 \oplus v_2 \colon
  (\mathcal{A}(-q) \otimes L_+) \oplus
  (\mathcal{A}(-1) \otimes L_+^{\otimes q}) \to
  \mathcal{A}
\big),\;\text{and} \\
\varphi_*\sO_{S^\circ} & \cong
\coker\big(
  v_3 \colon
  \mathcal{B}_T \otimes L_+^{\otimes q+1} \to
  \mathcal{B}_T
\big)\;\;
\text{where}\;
\mathcal{B}_T \coloneqq \mathcal{B} \otimes_{\mathcal{A}} \pi_*\sO_{T^\circ}.
\end{align*}
The coordinate rings carry a grading induced by the torus in the group
\(\mathrm{G}\) from \parref{nodal-automorphisms}. This torus is the subgroup
\((\lambda^{-1}, \lambda^q)\) of the \(\mathbf{G}_m^2\) from
\parref{D-symmetries}, meaning that the gradings on \(\pi_*\sO_{T^\circ}\) and
\(\varphi_*\sO_{S^\circ}\) are related to the bigradings of \(\mathcal{A}\) and
\(\mathcal{B}\) via
\[
\mathcal{A}_d = \bigoplus\nolimits_{a + bq = d} \mathcal{A}_{(a,b)}
\quad\text{and}\quad
\mathcal{B}_d = \bigoplus\nolimits_{a + bq = d} \mathcal{B}_{(a,b)}.
\]
Similarly, the action of \(\mathbf{G}_a\) on \(\mathbf{B}\) is related to that
of \(\boldsymbol{\alpha}_q\) on \(S^\circ\); in particular, the operator
\(\partial \colon \mathcal{B} \to \mathcal{B}\)
from \parref{D-unipotent} induces a \(\pi_*\sO_{T^\circ}\)-module endomorphism
\(\partial \colon \varphi_*\sO_{S^\circ} \to \varphi_*\sO_{S^\circ}\) which
is of degree \(-q-1\) and shifts the induced filtration down by \(1\).

To proceed, examine the equations \(v_1\), \(v_2\), and \(v_3\) in detail.
Begin with \(v_1\): This is of degree \(q\) and of the form \(x^q + \delta\),
where \(x\) is the degree \(1\) generator of \(\mathcal{B}\) and \(\delta\)
involves the degree \(q\) generators. Using this to eliminate the
\(q\)-th power of the degree \(1\) coordinate results in the following:

\begin{Lemma}\label{D-T}
Let
\(\mathcal{B}' \coloneqq \coker(v_1 \colon \mathcal{B}(-q) \otimes L_+ \to \mathcal{B})\).
Then for \(b \in \mathbf{Z}_{\geq 0}\) and \(0 \leq a < q\),
\[
\mathcal{B}'_{bq + a} \cong
\Fil_a\big(\Sym^b(W^\vee) \otimes \sO_C(-a)\big).
\]
\end{Lemma}

\begin{proof}
More globally, \(v_1\) is obtained by pairing \(q\)-powers of the coordinates
of \(\PP\mathcal{V}_1\) with the coordinates of \(\PP\mathcal{V}_2\) via
\(\beta\), so it acts as the isomorphism
\(\beta_U \colon L_+ \to L_-^{\vee, \otimes q}\) on the components in \(U\),
and as
\[
\delta = \beta_W^\vee \circ \mathrm{eu}_{\PP W}^{[1]} \colon
\sO_C(-q) \to
W^{[1]}\otimes \sO_C \xrightarrow{\sim}
W^\vee \otimes \sO_C
\]
on the components in \(W\); furthermore, since \(C\) is the locus over which
\(\mathrm{eu}_{\PP W}\) is isotropic for \(\beta_W\), \(\delta\) factors
through \(\Omega^1_{\PP W}(1)\rvert_C \subset W^\vee \otimes \sO_C\).
Therefore \(v_1\) is induced by the map
\[
(\beta_U, \delta) \colon
\sO_C(-q) \otimes L_+ \to
\sO_C(-q) \otimes L_-^{\vee,\otimes q} \oplus
\Omega_{\PP W}^1(1)\rvert_C \otimes L_+,
\]
and generally maps the degree \((d,e-1)\) bigraded piece of \(\mathcal{B}\)
to the degrees \((d+q,e-1)\) and \((d,e)\) pieces. But
\(\mathcal{B}_{(d,e-1)}\) and \(\mathcal{B}_{(d+q,e-1)}\) differ only by a
twist of \(\sO_C(-q) \otimes L_-^{\vee,\otimes q}\) by
\parref{cohomology-bigrading}, so \(v_1\) followed by projection yields an
isomorphism between the two sheaves. Considering now all bigraded pieces that
lie in total degree \(bq+a\), this implies that the map
\[
\mathcal{B}_{(b-1)q+a}(-q) \otimes L_+ \xrightarrow{v_1}
\mathcal{B}_{bq+a} \twoheadrightarrow
\bigoplus\nolimits_{e = 0}^{b-1} \mathcal{B}_{((b-e)q+a, e)}
\]
obtained by composing \(v_1\) with projection to its first \(d\) bigraded
components is an isomorphism, and so
\(\mathcal{B}'_{bq+a} \cong \mathcal{B}_{(a,b)}\). The result now
follows from \parref{cohomology-bigrading}.
\end{proof}

The same argument shows that
\(\coker(v_1 \colon \mathcal{A}(-q) \otimes L_+ \to \mathcal{A})_{bq+a} \cong \mathcal{A}_{(a,b)}\).
Since \(v_2\) and \(v_3\) have degrees at least \(q^2\), this identifies the
low degree components of \(\pi_*\sO_{T^\circ}\) and \(\varphi_*\sO_{S^\circ}\):

\begin{Corollary}\label{D-low-degree}
For each \(0 \leq a, b \leq q - 1\), there are canonical identifications
\[
\pushQED{\qed}
(\pi_*\sO_{T^\circ})_{bq+a} \cong \Sym^b(\Omega_{\PP W}(1))\rvert_C \otimes \sO_C(-a)
\;\;\text{and}\;\;
(\varphi_*\sO_{S^\circ})_{bq + a}  \cong \Fil_a\big(\Sym^b(W^\vee) \otimes \sO_C(-a)\big).
\qedhere
\popQED
\]
\end{Corollary}

Consider the equation \(v_3\) next: By its construction in
\parref{cone-situation-section-over-T-subbundle}, it arises from a map induced
by the \(q\)-bic form \(\beta_{\mathcal{V}_T}\). However, the identifications
from \parref{D-affine-bundles} together with the sequence from
\parref{cone-situation-V-sequence} shows that
\(\mathcal{V}\rvert_{\mathbf{A}}\) is canonically identified as an extension
\[
0 \to
L_{-,\mathbf{A}} \to
\mathcal{V}_{\mathbf{A}} \to
L_{+,\mathbf{A}} \to
0.
\]
Therefore \(\beta\) already induces a \(q\)-bic form on
\(\mathcal{V}_{\mathbf{A}}\), meaning that \(v_3\) extends to all of
\(\mathbf{A}\). From now on, view \(v_3\) as a map
\(\mathcal{B}' \otimes L_+^{\otimes q+1} \to \mathcal{B}'\). Since \(v_3\) is
locally a degree \(q\) polynomial in the fibre coordinate of
\(\rho \colon \mathbf{B} \to \mathbf{A}\), and since the filtration of
\(\mathcal{B}\) from \parref{cohomology-A-B} is by degree of the fibre
coordinate of \(\rho\), it follows that \(v_3\) shifts the filtration up by
\(q\) steps. In fact:

\begin{Lemma}\label{D-v3}
The section \(v_3\) maps \(\mathcal{B}' \otimes L_+^{\otimes q+1}\)
isomorphically onto \(\mathcal{B}'/\Fil_{q-1}\mathcal{B}'\).
\end{Lemma}

\begin{proof}
If \(x\) and \(y\) are local coordinates of the subbundles
\(\PP L_{-,\mathbf{A}}\) and \(\PP L_{+,\mathbf{A}}\) in
\(\PP\mathcal{V}_{\mathbf{A}}\), then \(x/y\) is the local
fibre coordinate of the affine bundle \(\mathbf{B}\), and \(v_3\) is
\((x/y)^q\). Globally, this means that \(v_3\) acts through the isomorphism
\(\beta_U \colon L_+ \to L_-^{\vee, \otimes q}\), and that it maps
\(\rho_*\sO_{\mathbf{B}} \otimes L_+^{\otimes q+1}\) onto the principal ideal
of \(\rho_*\sO_{\mathbf{B}}\) generated in degree \(q\) by
\((L_-^\vee \otimes L_+)^{\otimes q}\). Comparing with \parref{cohomology-A-B}
gives the result.
\end{proof}

Finally, consider the equation \(v_2\). Arguing as in \parref{D-T} shows
that it is induced by the map
\[
\beta_W \circ \mathrm{eu}_{\PP W} \colon
\sO_C(-1) \otimes L_+^{\otimes q} \to
\Omega_{\PP W}(1)\rvert_C^{[1]} \otimes L_+^{\otimes q}.
\]
This makes higher degree components of \(\varphi_*\sO_{S^\circ}\) less
straightforward to describe. Components of the form
\((\varphi_*\sO_{S^\circ})_{dq+q-1}\) for \(d < 2q\) are a notable and
useful exception, as they may be exhibited as an extension of two rather simple
bundles:

\begin{Proposition}\label{cohomology-D}
For each \(q \leq b \leq 2q-1\), there are short exact sequences of filtered
bundles
\[
0 \to
\Sym^{b-q}(W^\vee) \otimes \sO_C(-2b+q) \to
(\varphi_*\sO_{S^\circ})_{bq+q-1} \to
\Sym^b_{\mathrm{red}}(W^\vee) \otimes \sO_C(-2b+2q-1) \to 0.
\]
\end{Proposition}

\begin{proof}
Consider the degree \(b(q+1)\) component of \(\varphi_*\sO_{S^\circ}\): First,
the degree \(b(q+1)\) component of \(\mathcal{B}\) modulo \(v_1\) and \(v_3\)
is identified with \(\Fil_{q-1}\big(\Sym^b(W^\vee) \otimes \sO_C(-b)\big)\)
using \parref{cohomology-A-B}, \parref{D-T}, and \parref{D-v3}. Next, there is
a short exact sequence
\[
0 \to
\Fil_{q-1}\big(\Sym^{b-q}(W^\vee) \otimes \sO_C(-b+q-1)\big) \to
\Fil_{q-1}\big(\Sym^b(W^\vee) \otimes \sO_C(-b)\big) \to
(\varphi_*\sO_{S^\circ})_{b(q+1)} \to
0
\]
where the first map is induced by \(v_2\) and multiplication. Since Frobenius
stretches filtrations by a factor of \(q\), the inclusion
\(\Omega_{\PP W}^1(1)\rvert_C^{[1]} \hookrightarrow W^{\vee,[1]} \otimes \sO_C\)
induces an equality
\[
\Fil_{q-1}\big(\Omega_{\PP W}^1(1)\rvert_C^{[1]} \otimes \Sym^{b-q}(W^\vee) \otimes \sO_C(-b+q)\big) =
\Fil_{q-1}\big(W^{\vee,[1]} \otimes \Sym^{b-q}(W^\vee) \otimes \sO_C(-b+q)\big).
\]
Factoring the first map in the short exact sequence through \(v_2\), using
this identification of filtered pieces, and applying the snake lemma
with \(\coker(v_2) \cong \sO_C(-q+1)\)
yields a short exact sequence
\[
0 \to
\Fil_{q-1}\big(\Sym^{b-q}(W^\vee) \otimes\sO_C(-b+1)\big) \to
(\varphi_*\sO_{S^\circ})_{q(b+1)} \to
\Fil_{q-1}\big(\Sym^b_{\mathrm{red}}(W^\vee) \otimes \sO_C(-b+q)\big) \to
0.
\]
Since \(b - q < q\), \(\Fil_{q-1}\) gives the entire subbundle. Since
\(\Omega^1_{\PP W}(1)\rvert_C\) lies in the \(0\)-th step of the filtration of
\(W^\vee \otimes \sO_C(-1)\), the linear algebra fact
\parref{D-filtration-jump} shows the same for the quotient. Finally,
multiplication with \((\varphi_*\sO_{S^\circ})_1^{\otimes b - q + 1} \cong
\sO_C(b-q+1)\) maps \((\varphi_*\sO_{S^\circ})_{bq+q-1}\) injectively into
\((\varphi_*\sO_{S^\circ})_{b(q+1)}\), and this is an isomorphism because the
relations of \(\varphi_*\sO_{S^\circ}\) lie in degrees \(q\), \(q^2\), and
\(q(q+1)\), meaning that the ranks of the two bundles match. Twisting
then gives the result.
\end{proof}

The following is a simple observation about how Frobenius twists interact with
filtrations:

\begin{Lemma}\label{D-filtration-jump}
Let \(V\) be a finite dimensional vector space with a two step filtration
\(\Fil_0 V \subseteq \Fil_1 V = V\). If \(\gr_1 V\) is one-dimensional,
then, for all integers \(b \geq q\), the map
\[
V^{[1]} \otimes \Sym^{b-q}(V) \to
\Sym^b(V)/\Fil_{q-1}\Sym^b(V)
\]
induced by multiplication is surjective, and it induces a canonical
isomorphism
\[
\Fil_{q-1}\Sym^b(V)/\Fil_{q-1}(V^{[1]} \otimes \Sym^{b-q}(V)) \cong
\Sym^b_{\mathrm{red}}(V).
\]
\end{Lemma}

\begin{proof}
Choose a basis \(\Fil_0 V = \langle v_1,\ldots,v_n \rangle\) and extend it to a
basis of \(V\) with a lift \(w \in V\) of a basis vector of \(\gr_1 V\). Then
\(\Sym^b(V)\) has a basis given by the monomials of degree \(b\) in the
\(v_1,\ldots,v_n,w\), and the \((q-1)\)-st piece of the induced filtration is
\[
\Fil_{q-1}\Sym^b(V) =
\langle v_1^{i_1} \cdots v_n^{i_n} w^j \mid
i_1 + \cdots i_n + j = b\; \text{and}\; j \leq q-1 \rangle.
\]
Therefore each member of the monomial basis of
\(\Sym^b(V)/\Fil_{q-1}\Sym^b(V)\) is a product of \(w^q\) with an element of
\(\Sym^{b-q}(V)\), so the multiplication map in question is surjective. Since
\[
\Fil_{q-1}(V^{[1]} \otimes \Sym^{b-q}(V)) =
\ker\big(V^{[1]} \otimes \Sym^{b-q}(V) \to \Sym^b(V)/\Fil_{q-1}\Sym^b(V)\big)
\]
the second statement now follows from the five lemma.
\end{proof}

\subsectiondash{Proof of \parref{D-F-duality}}\label{D-F-proof}
It remains to put everything together. First, by
\parref{nodal-compute-conductor}, the coordinate ring \(\mathcal{D}\) of the
conductor subscheme \(D \hookrightarrow S\) is the truncation of
\(\varphi_*\sO_{S^\circ}\) at degrees larger than \(\delta = 2q^2-q-2\).
Second, \parref{conductor-dual} gives a canonical isomorphism
\[
\mathcal{F} \cong
\mathcal{D}^\vee \otimes
\sO_C(-q+1) \otimes
L_+^{\otimes 2q-1} \otimes L_-^{\otimes 2}.
\]
Gradings are related by
\(\mathcal{F}_d \cong \mathcal{D}_{\delta - d}^\vee \otimes \sO_C(-q+1)\), so
since \(\delta - (bq+q-1) = (2q-b-1)q-q-1\), combined with
\parref{cohomology-D}, this gives the exact sequence of \parref{D-F-duality}.
Third, since the filtration on \(\mathcal{D}\) has only \(q\) steps by
\parref{D-v3}, it induces a filtration on \(\mathcal{F}\) via
\[
\Fil_i\mathcal{F}
\coloneqq (\mathcal{D}/\Fil_{q-2-i}\mathcal{D})^\vee \otimes
\sO_C(-q+1) \otimes L_+^{\otimes 2q-1} \otimes L_-^{\otimes 2}
\;\text{for}\;0 \leq i \leq q-1.
\]
In particular, \(\Fil_0\mathcal{F}\) is related to the \((q-1)\)-st graded
piece of \(\mathcal{D}\), which by \parref{D-low-degree} is
\begin{align*}
\gr_{q-1}\mathcal{D}
& \cong (\pi_*\sO_{T^\circ} \otimes (L_-^\vee \otimes L_+)^{\otimes q-1})_{\leq 2q^2-q-2}
  \cong (\pi_*\sO_{T^\circ})_{\leq q^2-q-1} \\
& \cong
\bigoplus\nolimits_{b = 0}^{q-2}
\bigoplus\nolimits_{a = 0}^{q-1}
  \Sym^b(\Omega_{\PP W}^1(1))\rvert_C \otimes \sO_C(-a) \otimes  L_+^{\otimes b} \otimes L_-^{\vee, \otimes a}.
\end{align*}
Dualizing, twisting, inverting summation indices, and effacing a global factor
of \((L_+ \otimes L_-)^{\otimes q+1}\) identifies \(\Fil_0\mathcal{F}\) as in
\parref{D-F-duality}. Finally, \(\partial \colon \mathcal{D} \to \mathcal{D}\)
from the \(\boldsymbol{\alpha}_q\)-action fits into a commutative diagram
\[
\begin{tikzcd}
0 \rar
& \Fil_{q-1-i}\mathcal{D} \rar \dar["\partial"]
& \mathcal{D} \rar \dar["\partial"]
& \mathcal{D}/\Fil_{q-1-i}\mathcal{D} \rar \dar["\partial"]
& 0 \\
0 \rar
& \Fil_{q-2-i}\mathcal{D} \rar
& \mathcal{D} \rar
& \mathcal{D}/\Fil_{q-2-i}\mathcal{D} \rar
& 0
\end{tikzcd}
\]
and so there is an induced map \(\partial \colon \mathcal{F} \to \mathcal{F}\)
which, by \parref{D-unipotent}, is of degree \(-q-1\), satisfies
\(\partial(\Fil_i\mathcal{F}) \subseteq \Fil_{i-1}\mathcal{F}\) for each
\(0 \leq i \leq q-1\), and such that
\(\gr_i\partial \colon \gr_i\mathcal{F}_{d+q+1} \to \gr_{i-1}\mathcal{F}_d\)
is an isomorphism if \(p \nmid i\) and zero otherwise. This completes the
proof of \parref{D-F-duality}.
\qed

\section{Cohomology of \texorpdfstring{\(\mathcal{F}\)}{F}}
\label{section-cohomology-F}
The purpose of this section is to compute cohomology of \(\mathcal{F}\) when
\(q = p\) using the results of \S\parref{section-D}: see
\parref{cohomology-theorem}. The strategy is to identify each graded component
\(\mathrm{H}^0(C,\mathcal{F}_i)\) as a representation of the special unitary
group \(\mathrm{SU}_3(p) \coloneqq \mathrm{U}_3(p) \cap \mathbf{SL}(W)\), and
this is achieved in three steps: First, many global sections are constructed
for pieces of the form \(\mathcal{F}_{bp+p-1}\) using the exact sequence in
\parref{D-F-duality}. Second, the action of \(\boldsymbol{\alpha}_p\) on
\(\mathcal{F}\) gives maps
\[
\mathcal{F}_{bp+p-1} \xrightarrow{\partial}
\mathcal{F}_{bp+p-1-(p+1)} \xrightarrow{\partial}
\cdots \xrightarrow{\partial}
\begin{dcases*}
\mathcal{F}_{p-1-b} & if \(0 \leq b \leq p-2\), and \\
\mathcal{F}_{(b-p+1)p} & if \(p-1 \leq b \leq 3p-3\).
\end{dcases*}
\]
Each of the maps are injective on global sections by
\parref{cohomology-del-sections}, and so this gives a lower bound on the space
of sections of each component, see \parref{cohomology-F-sequences}. Third, a
corresponding upper bound is determined for those rightmost sheaves, see
\parref{cohomology-upper-bound} and \parref{cohomology-left}.

Begin with a simple, but crucial, computation of sections lying in the
\(0\)-th filtered piece of \(\mathcal{F}\):

\begin{Lemma}\label{cohomology-fil-F}
\(
\mathrm{H}^0(C,\Fil_0\mathcal{F}) \cong
\bigoplus\nolimits_{b = 0}^{p-2} \Div^{p-2-b}(W) \otimes L_+^{\otimes b}
\).
\end{Lemma}

\begin{proof}
The filtration statement of \parref{D-F-duality} shows that
\[
\mathrm{H}^0(C, \Fil_0\mathcal{F}) \cong
\bigoplus\nolimits_{b = 0}^{p-2}
\bigoplus\nolimits_{a = 0}^{p-1}
\mathrm{H}^0(C,
\Div^{p-2-b}(\mathcal{T}) \otimes \sO_C(-a)) \otimes
L_+^{\otimes b} \otimes
L_-^{\vee, \otimes a}.
\]
All divided powers appearing have exponent less than \(p\), so may
be replaced by symmetric powers. Then Griffith's Borel--Weil--Bott
vanishing \parref{cohomology-BWB-vanishing} applies to give the result.
\end{proof}

Consider now the operator \(\partial \colon \mathcal{F} \to \mathcal{F}\)
induced by the action of \(\boldsymbol{\alpha}_p\) on \(S\). The final
statement of \parref{D-F-duality} implies that
\[
\ker(\partial \colon \mathcal{F} \to \mathcal{F}) =
\bigoplus\nolimits_{i = 0}^p \mathcal{F}_i \oplus
\bigoplus\nolimits_{i = p+1}^\delta \Fil_0\mathcal{F}_i.
\]
Taking global sections and comparing with \parref{cohomology-fil-F} shows that
\(\partial\) acts injectively on most of the graded components of
\(\mathrm{H}^0(C,\mathcal{F})\):

\begin{Lemma}\label{cohomology-del-sections}
\(\displaystyle
\ker\big(\partial \colon \mathrm{H}^0(C,\mathcal{F}) \to \mathrm{H}^0(C,\mathcal{F})\big) =
\bigoplus\nolimits_{a = 0}^{p-1} \mathrm{H}^0(C,\mathcal{F}_a) \oplus
\bigoplus\nolimits_{b = 1}^{p-2} \mathrm{H}^0(C,\Fil_0\mathcal{F}_{bp}).\)
\qed
\end{Lemma}

Iterating \(\partial\) gives a lower bound on the sections of the
\(\mathcal{F}_i\):

\begin{Lemma}\label{cohomology-F-sequences}
For each \(0 \leq b \leq 2p-3\) and \(0 \leq a \leq \min(b,p-1)\),
\[
\Div^{2p-3-b}_{\mathrm{red}}(W) \subseteq
\mathrm{H}^0(C,\mathcal{F}_{bp+p-1-a(p+1)}).
\]
\end{Lemma}

\begin{proof}
When \(a = 0\), this concerns \(\mathcal{F}_{bp+p-1}\), and the statement
follows from the exact sequence in \parref{D-F-duality}. Iteratively
applying \(\partial\) then gives the result for \(a > 0\) in view of
injectivity from \parref{cohomology-del-sections}.
\end{proof}

It remains to give a matching upper bound. Generic injectivity of
\(\partial \colon \mathrm{H}^0(C,\mathcal{F}) \to \mathrm{H}^0(C,\mathcal{F})\)
from \parref{cohomology-del-sections} means it suffices to determine
\(\mathrm{H}^0(C,\mathcal{F}_i)\) when \(0 \leq i \leq p - 1\), and when
\(i = jp\) for \(0 \leq j \leq 2p-2\). The cases \(0 \leq i \leq p\) are dealt
with an explicit cohomology computation, see \parref{cohomology-nonzero-map}
and \parref{cohomology-upper-bound}; the remaining cases then follow from this
explicit calculation by further analysing the action of \(\partial\) on global
sections, see \parref{cohomology-left}.

\subsectiondash{}
Consider the degree \(0 \leq i \leq p\) components of the defining presentation
of \(\mathcal{F}\) from \parref{nodal-F}:
\[
0 \to
\mathcal{D}_i \xrightarrow{\nu_i^\#}
\mathcal{D}^\nu_i \to
\mathcal{F}_i \to
0.
\]
Identifying low degree pieces of \(\mathcal{D}\) and \(\mathcal{D}^\nu\) via
\parref{cohomology-D} and \parref{cohomology-D-nu}, respectively, and taking
the long exact sequence in cohomology shows that
\[
\mathrm{H}^0(C,\mathcal{F}_i) \cong
\begin{dcases*}
\ker\big(\nu_i^\# \colon
  \mathrm{H}^1(C,\sO_C(-i)) \to
  \mathrm{H}^1(C,\phi_{C,*}(\mathcal{T}_C(-1)^{\otimes i}))
\big) & if \(0 \leq i \leq p-1\), and \\
\ker\big(\nu_p^\# \colon
  \mathrm{H}^1(C,\Omega_{\PP W}^1(1)\rvert_C) \to
  \mathrm{H}^1(C,\phi_{C,*}(\mathcal{T}_C(-1)^{\otimes p}))
\big) & if \(i = p\).
\end{dcases*}
\]
An explicit description of \(\nu_1^\#\) can be given in terms of the canonical
section \(\theta \colon \sO_C(-p^2) \to \mathcal{T}_C(-1)\) from
\parref{generalities-theta} determining the Hermitian points of \(C\):

\begin{Lemma}\label{cohomology-nu}
The sheaf map \(\nu_1^\# \colon \mathcal{D}_1 \to \mathcal{D}_1^\nu\) is
identified with
\[
\phi_{C,*}(\theta) \circ \phi_C^\# \colon
  \sO_C(-1) \to
  \phi_{C,*}\phi_C^*(\sO_C(-1)) \to
  \phi_{C,*}(\mathcal{T}_C(-1))
\]
\end{Lemma}

\begin{proof}
The degree \(1\) generators of \(\mathcal{D}\) and \(\mathcal{D}^\nu\)
correspond to the fibre coordinates over \(C\) of the affine bundles
\(\mathbf{A}_1\) from \parref{D-affine-bundles} and \(S^{\nu,\circ}\)
from \parref{cohomology-D-nu}, respectively. So consider the commutative diagram
\[
\begin{tikzcd}
S^\nu \rar["\nu"'] \dar["\tilde\varphi_+"]
& S \rar[dashed,"\psi"'] \dar["\varphi_-"]
& \PP\mathcal{V}_1 \dar["\pi_1"] \\
C \rar["\phi_C"] & C \rar["\id_C"] & C
\end{tikzcd}
\]
obtained by putting \parref{cone-situation-P} and \parref{nodal-nu-and-F}
together. The rational map \(\psi\) restricts to a morphism
\(S^\circ \to \mathbf{A}_1\), and sends a line \(\ell \subset X \setminus x_-\)
to its point of intersection \(\ell \cap \PP L_-^{[1],\perp}\) with the tangent
hyperplane to \(X\) at \(x_-\). Thus the pullback of \(\sO_{\pi_1}(-1)\) via
\(\psi\), at least on \(S^\circ\), is identified with the subsheaf
\[
\mathcal{L} \coloneqq \mathcal{S} \cap L_{-,S}^{[1],\perp} \subset V_S.
\]
Moreover, \(\mathcal{L}\) lies inside
\(\varphi_-^*\mathcal{V}_1 = \varphi_-^*\sO_C(-1) \oplus L_{-,S}\), the
subbundle of \(V_S\) parameterizing the lines spanned by points of
\(C\) and \(x_-\), and projection away from the \(L_{-,S}\) factor yields the
identification between the fibre coordinate of \(\mathcal{A}_1\) with
\(\varphi_-^*\sO_C(-1)\) on \(S^\circ\), as explained in
\parref{D-affine-bundles}.

Pulling back to \(S^\nu\) produces a subsheaf \(\nu^*\mathcal{L}
\hookrightarrow \mathcal{K}\), where \(\mathcal{K}\) is as in the proof of of
\parref{nodal-normalization}. The map \(\nu_1^\#\) in question is obtained from
the composition
\[
\nu^*\mathcal{L} \hookrightarrow
\mathcal{K} \twoheadrightarrow
\sO_{\tilde\varphi_+}(-1)
\]
by projecting out \(L_-\) and restricting to \(S^{\nu,\circ}\).
Projection away from \(L_-\) maps \(\mathcal{K}\) onto the subbundle of
\((W \oplus L_+)_{S^\nu}\) underlying planes spanned by the tangent lines to
\(C\) and \(x_+\); since \(\nu^*\mathcal{L}\) and \(\sO_{\tilde\varphi_+}(-1)\)
project to line bundles on \(C\), the map in question is pulled back via
\(\tilde\varphi_+^*\) of the map
\[
\phi_C^*(\sO_C(-1)) \hookrightarrow
\mathcal{E}_C \twoheadrightarrow
\mathcal{T}_C(-1)
\]
which, as explained in \parref{generalities-theta}, is given by \(\theta\).
Pushing along \(\phi_C \colon C \to C\) now gives the result.
\end{proof}

\subsectiondash{}\label{cohomology-phi}
Since \(\nu^\#\) is a map of algebras, the composition
\(\nu^\#_i \circ \mu_i \colon \mathcal{D}_1^{\otimes i} \hookrightarrow \mathcal{D}_i \to \mathcal{D}^\nu_i\)
of the \(i\)-fold multiplication map \(\mu_i\) followed by \(\nu^\#_i\) is
given by the \(i\)-th power \(\phi_{C,*}(\theta^i) \circ \phi_C^\#\) of the map
appearing in \parref{cohomology-nu}. The action of this map on
\(\mathrm{H}^1(C,\sO_C(-i))\) can be determined as follows: Let
\[
f \in \mathrm{H}^0(\PP W, \sO_{\PP W}(p+1))
\quad\text{and}\quad
\tilde\theta \in \mathrm{H}^0(\PP W,\sO_{\PP^2}(p^2-p+1))
\]
be an equation for \(C\) and any lift of
\(\theta \in \mathrm{H}^0(C,\sO_C(p^2-p+1))\), respectively. Then there is a
commutative diagram of sheaves on \(\PP W\) with exact rows given by
\[
\begin{tikzcd}
\sO_{\PP W}(-i-p-1) \rar[hook,"f"'] \dar["\phi_C^\#"]
& \sO_{\PP W}(-i) \rar[two heads] \ar[dd,"\phi_C^\#"]
& \sO_C(-i) \ar[dd,"\phi_C^\#"] \\
\phi_{C,*}\sO_{\PP W}(-(i+p+1)p^2) \dar["f^{p^2-1}"] \\
\phi_{C,*}\sO_{\PP W}(-ip^2 - p - 1) \rar[hook, "\phi_{C,*}(f)"] \dar["\phi_{C,*}(\tilde\theta^i)"]
& \phi_{C,*}\sO_{\PP W}(-ip^2) \rar[two heads] \dar["\phi_{C,*}(\tilde\theta^i)"]
& \phi_{C,*}\sO_C(-ip^2) \dar["\phi_{C,*}(\theta^i)"] \\
\phi_{C,*}\sO_{\PP W}(-i(p-1)-p-1) \rar[hook]
& \phi_{C,*}\sO_{\PP W}(-i(p-1)) \rar[two heads]
& \phi_{C,*}(\mathcal{T}_C(-1)^{\otimes i})\punct{.}
\end{tikzcd}
\]
Taking cohomology and explicitly computing with the cohomology of \(\PP W\)
gives:

\begin{Lemma}\label{cohomology-nonzero-map}
\(
\phi_{C,*}(\theta^i) \circ \phi_C^\# \colon
\mathrm{H}^1(C,\sO_C(-i)) \to
\mathrm{H}^1(C,\phi_{C,*}(\mathcal{T}_C(-1)^{\otimes i}))
\)
is nonzero for \(2 \leq i \leq p\).
\end{Lemma}

\begin{proof}
Choose coordinates \((x:y:z)\) on \(\PP W = \PP^2\) so that
\(f = x^p y + x y^p - z^{p+1}\). The following is a lift of \(\theta\), as
can be verified by showing it vanishes on the
\(\mathbf{F}_{p^2}\) points of \(C\), as done in \citeThesis{3.5.4}:
\[
\tilde\theta \coloneqq
\frac{x^{p^2} y - x y^{p^2}}{x^p y + x y^p}z =
\big(x^{p(p-1)} - x^{(p-1)(p-1)} y^{p-1} + \cdots - y^{p(p-1)}\big) z.
\]
View the cohomology groups of \(\PP^2\) as a module over its homogeneous
coordinate ring as explained in \citeSP{01XT}, and consider a class
\[
\xi \coloneqq \frac{1}{xyz} \frac{1}{x^{i + p - 2}} \in
\mathrm{H}^2(\PP^2,\sO_{\PP^2}(-i-p-1)).
\]
This acts on homogeneous polynomials by contraction, see \citeSP{01XV}. Observe
that \(\xi \cdot f = 0\) since \(f\) does not contain a pure power of \(x\),
and so \(\xi\) represents a class in \(\mathrm{H}^1(C,\sO_C(-i))\). I claim
that \(\phi_{C,*}(\theta^i)(\xi) \neq 0\). Indeed, since \(f\) is a Hermitian
\(q\)-bic equation, \(\phi_C\) is the \(p^2\)-power Frobenius by
\parref{generalities-hermitian}, and the diagram of \parref{cohomology-phi}
shows that \(\phi_{C,*}(\theta^i)(\xi)\) is represented by the product
\[
\xi^{p^2} \cdot (f^{p^2-1} \tilde\theta^i) =
\Big(\frac{1}{xyz} \frac{1}{x^{d+p-2}}\Big)^{p^2} \cdot
\Big((x^p y + x y^p - z^{p+1})^{p^2 - 1}
\Big(\frac{x^{p^2} y - x y^{p^2}}{x^p y + x y^p}\Big)^i z^i\Big).
\]
Consider the coefficient of \(z^{(p+1)(p-2)+i}\) in \(f^{p^2-1}\tilde\theta^i\):
Since \(0 < i < p+1\), this is the coefficient of \(z^i\) in \(\tilde\theta^i\)
multiplied by the coefficient of \(z^{(p+1)(p-2)}\) in \(f^{p^2-1}\). Writing
\[
f^{p^2-1} =
\big((x^p y + x y^p)^p - z^{p(p+1)}\big)^{p-1}
\big((x^p y + x y^p) - z^{p+1}\big)^{p-1}
\]
shows that the latter is \(-(x^p y + x y^p)^{p^2-p+1}\).
Therefore \(\xi^{p^2} \cdot (f^{p^2-1} \tilde\theta^i)\) has as a summand
\begin{multline*}
\frac{1}{x^{(i+p-1)p^2} y^{p^2} z^{p+2-i}} \cdot
  \Big(-(x^p y + x y^p)^{p^2-p+1}
  \Big(\frac{x^{p^2} y - x y^{p^2}}{x^p y + x y^p}\Big)^i\Big) \\
= \frac{-1}{x^{(i+p-2)p^2 + p - 1} y^{p-1} z^{p+2-i}} \cdot
\Big((x^{p-1} + y^{p-1})^{p^2 - p + 1 - i} (x^{p^2 - 1} - y^{p^2-1})^i\Big).
\end{multline*}
Since all monomials in \(y\) involve at least \(y^{p-1}\), the only potentially
nonzero contribution is the pure power of \(x\), so this is equal to
\[
\frac{-1}{x^{(i+p-1)p^2 + p - 1} y^{p-1} z^{p+2-i}} \cdot x^{(p-1)(p^2-p+1-i) + (p^2-1)i}
= \frac{-1}{x^{(i-1)p} y^{p-1} z^{p+2-i}}.
\]
This is nonzero in \(\mathrm{H}^2(\PP^2,\sO_{\PP^2}(-(i+1)(p+1)))\)
if \(2 \leq i \leq p\), so \(\phi_{C,*}(\theta^i)(\xi) \neq 0\).
\end{proof}

The following three statements now determine the crucial components of
\(\mathrm{H}^0(C,\mathcal{F})\):

\begin{Proposition}\label{cohomology-upper-bound}
\(\displaystyle
\mathrm{H}^0(C,\mathcal{F}_i) \cong
\begin{dcases*}
\Div^{p+i-2}_{\mathrm{red}}(W) & if \(0 \leq i \leq p - 1\), and \\
\Div^{p-3}(W) & if \(i = p\).
\end{dcases*}
\)
\end{Proposition}

\begin{proof}
Applying \parref{cohomology-F-sequences} with \(a = b = p-i-1\) when
\(i \neq p\), and \(a+1 = b = p\) when \(i = p\) shows that
\(\mathrm{H}^0(C,\mathcal{F}_i)\) contains the representation \(L\) appearing
on the right side of the purported isomorphism. The exact sequence
\[
0 \to
\mathrm{H}^0(C,\mathcal{F}_i) \to
\mathrm{H}^1(C,\mathcal{D}_i) \xrightarrow{\nu^\#_i}
\mathrm{H}^1(C,\mathcal{D}_i^\nu) \to
\mathrm{H}^1(C,\mathcal{F}_i) \to
0
\]
then implies the statement in the case \(0 \leq i \leq 1\) since
\(\mathrm{H}^1(C,\mathcal{D}_i) = L\) by \parref{cohomology-D} and
\parref{cohomology-twist}\ref{cohomology-twist.low}. When \(2 \leq i \leq p\),
\parref{cohomology-D} and
\parref{cohomology-twist}\ref{cohomology-twist.sequence}--\ref{cohomology-twist.not-simple}
together show that \(\mathrm{H}^1(C,\mathcal{D}_i)/L\) is a simple
\(\mathrm{SU}_3(p)\)-representation, so to conclude, it suffices to show
that \(\nu^\#_i\) is nonzero. As explained in \parref{cohomology-phi},
there is a factorization
\[
\phi_{C,*}(\theta^i) \circ \phi_C^\# \colon
\mathrm{H}^0(C,\mathcal{D}_1^{\otimes i}) \xrightarrow{\mu_i}
\mathrm{H}^0(C,\mathcal{D}_i) \xrightarrow{\nu_i^\#}
\mathrm{H}^0(C,\mathcal{D}_i^\nu).
\]
Observe that \(\mu_i\) is surjective on global sections: When \(i < p\), this
is because \(\mu_i \colon \mathcal{D}^{\otimes i} \to \mathcal{D}_i\) is
already an isomorphism; when \(i = p\), this follows from
\(\coker(\mu_p \colon \sO_C(-p) \to \Omega^1_{\PP W}(1)\rvert_C) = \sO_C(p-1)\).
Since \(\phi_{C,*}(\theta^i) \circ \phi_C^\#\) is nonzero
by \parref{cohomology-nonzero-map}, \(\phi\) is also nonzero.
\end{proof}

\begin{Corollary}\label{cohomology-right}
\(
\mathrm{H}^0(C,\mathcal{F}_{bp+p-1}) \cong
\Div^{2p-3-b}_{\mathrm{red}}(W)
\)
for \(0 \leq b \leq 2p-3\).
\end{Corollary}

\begin{proof}
When \(p-2 \leq b \leq 2p-3\),
\(\mathcal{F}_{bp+p-1} \cong \Div^{2p-3-b}_{\mathrm{red}}(W) \otimes \sO_C\)
by the sequence in \parref{D-F-duality} yielding the conclusion in this case.
When \(0 \leq b \leq p-3\), \parref{cohomology-del-sections},
\parref{cohomology-F-sequences}, and \parref{cohomology-upper-bound} together
give a sequence of inclusions
\[
\Div^{2p-3-b}_{\mathrm{red}}(W) \subseteq
\mathrm{H}^0(C,\mathcal{F}_{bp+p-1}) \stackrel{\partial^b}{\hookrightarrow}
\mathrm{H}^0(C,\mathcal{F}_{p-1-b}) =
\Div^{2p-3-b}_{\mathrm{red}}(W).
\]
Therefore equality holds throughout.
\end{proof}

\begin{Proposition}\label{cohomology-left}
\(\displaystyle
\mathrm{H}^0(C,\mathcal{F}_{bp}) \cong
\begin{dcases*}
\Div^{p-2-b}(W) & if \(0 \leq b \leq p-2\), and \\
0 & if \(p-1 \leq b \leq 2p-2\).
\end{dcases*}
\)
\end{Proposition}

\begin{proof}
When \(0 \leq b \leq 1\), this is \parref{cohomology-upper-bound}. Assume that
\(2 \leq b \leq 2p-2\). The final statement of \parref{D-F-duality} implies
that there is an exact sequence
\[
0 \to
\Fil_0\mathcal{F}_{bp} \to
\mathcal{F}_{bp} \xrightarrow{\partial}
\mathcal{F}_{(b-2)p+p-1} \to
\gr_{p-1}\mathcal{F}_{(b-2)p+p-1} \to
0.
\]
Since \(\mathrm{H}^0(C,\Fil_0\mathcal{F}_{bp}) = \Div^{p-2-b}(W)\) by
\parref{cohomology-fil-F}, where negative divided powers are taken to be zero,
it suffices to show that \(\partial\) vanishes on global sections. Exactness of
the sequence means this is equivalent to injectivity of
\(\mathcal{F}_{(b-2)p+p-1} \to \gr_{p-1}\mathcal{F}_{(b-2)p+p-1}\)
on global sections. So consider the composite
\[
\Div^{2p-1-b}_{\mathrm{red}}(W) \otimes \sO_C \subset
\mathcal{F}_{(b-2)p+p-1} \twoheadrightarrow
\gr_{p-1}\mathcal{F}_{(b-2)p+p-1}
\]
where the first map is the inclusion of the subbundle from the sequence of
\parref{D-F-duality}. The first map
is an isomorphism on global sections by \parref{cohomology-right}. Comparing
with \parref{representations-divs} shows that
\[
\Div^{2p-1-b}_{\mathrm{red}}(W) =
\begin{dcases*}
L(p-b, b-1) & if \(2 \leq b \leq p-1\) and \\
L(0,2p-1-b) & if \(p \leq b \leq 2p-1\),
\end{dcases*}
\]
so it is a simple \(\mathrm{SU}_3(p)\) representation. Thus it
suffices to see that the composite is a nonzero map of sheaves. Since the maps
respect filtrations, it suffices to observe that
\[
\gr_{p-1}\big(\Div^{2p-1-b}_{\mathrm{red}}(W) \otimes \sO_C\big) =
\Fil_0\big(\Sym^{2p-1-b}_{\mathrm{red}}(W^\vee) \otimes \sO_C\big)^\vee =
\Div^{2p-1-b}_{\mathrm{red}}(\mathcal{T})
\]
is nonzero. The result now follows.
\end{proof}

\begin{Theorem}\label{cohomology-theorem}
\(\mathrm{H}^0(C,\mathcal{F}) \cong \Lambda_1 \oplus \Lambda_2\)
as a representation of \(\mathrm{G}\), where
\begin{align*}
\Lambda_1 & \coloneqq
\bigoplus\nolimits_{b = 0}^{p-2}
\Div^{p-2-b}(W) \otimes
\Sym^{p-1}(U) \otimes
L_+^{\otimes b} \otimes
L_-^{\vee, \otimes p-1}, \;\;\text{and} \\
\Lambda_2 & \coloneqq
\bigoplus\nolimits_{a = 0}^{p-2}
\Div^{p+a-1}_{\mathrm{red}}(W) \otimes
\Sym^{p-2-a}(U) \otimes
L_-^{\vee,\otimes p-1}.
\end{align*}
\end{Theorem}

\begin{proof}
Begin by identifying each \(\mathrm{H}^0(C,\mathcal{F}_i)\) as a representation
for \(\mathrm{SU}_3(p)\). The claim is that the inclusions from
\parref{cohomology-F-sequences} are equalities: that
\[
\mathrm{H}^0(C,\mathcal{F}_{bp+p-1-a(p+1)}) =
\Div^{2p-3-b}_{\mathrm{red}}(W)
\]
if \(0 \leq b \leq 2p-3\) and \(0 \leq a \leq \min(b,p-1)\), and that the
group vanishes otherwise.
Choose \(0 \leq b \leq 3p-3\). Starting from
\(\mathrm{H}^0(C,\mathcal{F}_{bp+p-1})\) and successively applying \(\partial\)
a total of \(\min(b,p-1)\) times produces, thanks to
\parref{cohomology-del-sections}, a chain of inclusions
\[
\mathrm{H}^0(C,\mathcal{F}_{bp+p-1}) \subseteq
\cdots \subseteq
\begin{dcases*}
\mathrm{H}^0(C,\mathcal{F}_{p-1-b}) & if \(0 \leq b \leq p-2\), and \\
\mathrm{H}^0(C,\mathcal{F}_{(b-p+1)p}) & if \(p-1 \leq b \leq 3p-3\).
\end{dcases*}
\]
By convention, set \(\mathrm{H}^0(C,\mathcal{F}_i) = 0\) whenever
\(i > \delta\). The spaces on the left are given by \parref{cohomology-right}
whereas the spaces on the right are given by \parref{cohomology-upper-bound}
and \parref{cohomology-left}, and for each fixed \(a\), the lower
and upper bounds match. Therefore equality holds throughout. The
\(\AutSch(L_- \subset U,\beta_U)\) factor of \(\mathrm{G}\) of the
representation is obtained by matching weights and using
\parref{cohomology-del-sections} to identify the action of the unipotent
radical.
\end{proof}

\begin{Corollary}\label{cohomology-dimension-result}
\(\displaystyle
\dim_{\mathbf{k}}\mathrm{H}^0(C,\mathcal{F}) =
(p^2 + 1)\binom{p}{2} + \binom{p}{3}
\).
\end{Corollary}

\begin{proof}
Sum the cohomology groups in \parref{cohomology-theorem}
column-wise, summing over residue classes modulo \(p\):
\begin{align*}
\dim_\kk\mathrm{H}^0(C,\mathcal{F})
& =
\sum\nolimits_{a = 0}^{p-1}
\sum\nolimits_{b = 0}^{a+p-2}
\dim_\kk\mathrm{H}^0(C,\mathcal{F}_{(a+p-2-b)p + a}) \\
& =
\sum\nolimits_{a = 0}^{p-1}
\sum\nolimits_{b = 0}^{a+p-2} \dim_\kk \Div^b_{\mathrm{red}}(W) \\
& =
\sum\nolimits_{a = 0}^{p-1}
\sum\nolimits_{b = 0}^{a+p-2} \big(\dim_\kk \Div^b(W) - \dim_\kk (W^{[1]} \otimes \Div^{b-p}(W))\big) .
\end{align*}
Since \(W\) is a \(3\)-dimensional vector space, \(\Div^b(W)\) is
\(\binom{b+2}{2}\) dimensional for all \(b \geq 0\), so using standard binomial
coefficient identities gives
\begin{align*}
\dim_\kk\mathrm{H}^0(C,\mathcal{F})
& =
\sum\nolimits_{a = 0}^{p-1}
\sum\nolimits_{b = 0}^{a+p-2} \binom{b+2}{2} -
3\sum\nolimits_{a = 2}^{p-1}
\sum\nolimits_{b = 0}^{b-2} \binom{b+2}{2} \\
& =
\sum\nolimits_{a = 0}^{p-1} \binom{a+p+1}{3}
-3
\sum\nolimits_{a = 2}^{p-1} \binom{a+1}{3}
= \binom{2p+1}{4} - 4\binom{p+1}{4}.
\end{align*}
It can now be directly verified that
\(\binom{2p+1}{4} - 4\binom{p+1}{4} = (p^2+1)\binom{p}{2} + \binom{p}{3}\).
\end{proof}

\begin{figure}[t]
\[
\begin{smallmatrix}
28&36&42&46&51&48&55&42 \\
24&36&39&42&56&72&60&46 \\
18&30&36&36&51&73&75&48 \\
10&18&24&28&36&51&56&48 \\
 6&10&15&21&28&36&42&46 \\
 3& 6&10&15&24&36&39&42 \\
 1& 3& 6&10&18&30&36&36 \\
 0& 1& 3& 6&10&18&24&28 \\
 0& 0& 1& 3& 6&10&15&21 \\
 0& 0& 0& 1& 3& 6&10&15 \\
 0& 0& 0& 0& 1& 3& 6&10 \\
 0& 0& 0& 0& 0& 1& 3& 6 \\
 0& 0& 0& 0& 0& 0& 1& 3 \\
 0& 0& 0& 0& 0& 0& 0& 1 \\
 0& 0& 0& 0& 0& 0& 0
\end{smallmatrix}
\quad\quad\quad\quad
\begin{smallmatrix}
36&45&52&58&60&61&71&60&52 \\
31&45&45&58&75&60&77&90&57 \\
21&31&36&45&58&57&63&71&60 \\
15&21&28&36&45&52&58&60&61 \\
11&18&21&31&45&45&58&75&60 \\
 6&11&15&21&31&36&45&58&57 \\
 3& 6&10&15&21&28&36&45&52 \\
 1& 3& 6&11&18&21&31&45&45 \\
 0& 1& 3& 6&11&15&21&31&36 \\
 0& 0& 1& 3& 6&10&15&21&28 \\
 0& 0& 0& 1& 3& 6&11&18&21 \\
 0& 0& 0& 0& 1& 3& 6&11&15 \\
 0& 0& 0& 0& 0& 1& 3& 6&10 \\
 0& 0& 0& 0& 0& 0& 1& 3& 6 \\
 0& 0& 0& 0& 0& 0& 0& 1& 3 \\
 0& 0& 0& 0& 0& 0& 0& 0& 1 \\
 0& 0& 0& 0& 0& 0& 0& 0
\end{smallmatrix}
\]
\caption{The dimensions of the \(\mathrm{H}^0(C,\mathcal{F}_i)\) are displayed
with \(q = 8\) on the left, and \(q = 9\) on the right. The numbers are
arranged so that the first row displays the dimensions of
\(\mathrm{H}^0(C,\mathcal{F}_i)\) for \(0 \leq i \leq q-1\). These were
obtained from computer calculations done with \textit{Macaulay2} \cite{M2}.}
\label{cohomology-F-restriction-remarks.figure}
\end{figure}

\subsectiondash{Remarks toward general \(q\)}\label{cohomology-F-restriction-remarks}
The assumption that \(q = p\) was used in at least three ways: First, to
to apply the Borel--Weil--Bott Theorem in \parref{cohomology-BWB-vanishing} and
to identify divided powers with symmetric powers in \parref{cohomology-fil-F};
Second, to reduce the action of \(\boldsymbol{\alpha}_q\) to the action of the
single operator \(\partial\) which enjoys the injectivity property
\parref{cohomology-F-sequences}; Third, to show in \parref{cohomology-twist}
that the \(\mathrm{SU}_3(q)\) representations appearing are either simple or
have very short composition series.

In any case, part of the difficulty to extending the computation past the
prime case is that the formula \parref{cohomology-dimension-result}
does not hold for general \(q\), and there may be some dependence on the
exponent of \(q\). A computer computation shows that
\[
\dim_\kk \mathrm{H}^0(C,\mathcal{F})
=
\begin{dcases*}
106  \\
2096  \\
3231  \\
\end{dcases*}
\quad\text{whereas}\quad
(q^2+1)\binom{q}{2} + \binom{q}{3} =
\begin{dcases*}
106 & if \(q = 4\), \\
1876 & if \(q = 8\), and \\
3036 & if \(q = 9\).
\end{dcases*}
\]
The dimensions of \(\mathrm{H}^0(C,\mathcal{F}_i)\) in the cases
\(q = 8\) and \(q = 9\) are given in Figure
\parref{cohomology-F-restriction-remarks.figure}. Certain
general features from the prime case hold true---for instance, the
action of \(\boldsymbol{\alpha}_q\) still relates graded components
which differ in weight by \(q-1\)---but are jumps in certain entries,
related to jumps in cohomology of homogeneous bundles on \(\PP W\).

\section{Smooth \texorpdfstring{\(q\)}{q}-bic threefolds}\label{section-smooth}
Finally, return to a smooth \(q\)-bic threefold \(X\) in \(\PP V\), and \(S\)
its Fano surface of lines. The aim of this section is to compute the cohomology
of \(\sO_S\) when \(q = p\), thereby proving Theorem
\parref{intro-smooth-cohomology}. The proof is contained in
\parref{smooth-proof} at the end of the section, and is achieved through a
careful degeneration argument. Specifically, the cohomology of \(S\) is related
to the cohomology of the Fano scheme \(S_0\) of a \(q\)-bic threefold of type
\(\mathbf{1}^{\oplus 3} \oplus \mathbf{N}_2\) via a special \(1\)-parameter
degeneration of the \(q\)-bic threefold \(X\); this family carries additional
symmetries that allows one to bootstrap the results of
\S\S\parref{section-nodal}--\parref{section-cohomology-F} to complete the
computation in the smooth case. The construction is as follows:

\begin{Proposition}\label{smooth-family}
Let \(x_-, x_+ \in X\) be Hermitian points such that
\(\langle x_-,x_+ \rangle \not\subset X\). Then there exists a
\(q\)-bic threefold \(\mathfrak{X} \subset \PP V \times \mathbf{A}^1\)
over \(\mathbf{A}^1\) such that
\begin{enumerate}
\item\label{smooth-family.sections}
the constant sections \(x_\pm \colon \mathbf{A}^1 \to \PP V \times \mathbf{A}^1\)
factor through \(\mathfrak{X}\);
\item\label{smooth-family.cone}
\((X_t, x_-, \mathbf{T}_{X_t,x_-})\) is a smooth cone situation for all \(t \in \mathbf{A}^1\);
\item\label{smooth-family.smooth}
the projection \(\mathfrak{X} \to \mathbf{A}^1\) is smooth
away from \(0 \in \mathbf{A}^1\) and \(X = \mathfrak{X}_1\); and
\item\label{smooth-family.central}
\(X_0\) is of type \(\mathbf{1}^{\oplus 3} \oplus \mathbf{N}_2\) with singular
point \(x_+\).
\end{enumerate}
Moreover, there exists a choice of coordinates \((x_0:x_1:x_2:x_3:x_4)\) on
\(\PP V \cong \PP^4\) such that
\[
\mathfrak{X} =
\mathrm{V}(x_0^q x_1 + t x_0 x_1^q + x_2^{q+1} + x_3^{q+1} + x_4^{q+1}) \subset
\PP^4 \times \mathbf{A}^1,
\]
\(x_- = (1:0:0:0:0)\), and \(x_+ = (0:1:0:0:0)\).
\end{Proposition}

Property \parref{smooth-family}\ref{smooth-family.smooth} means that
\(\mathfrak{X}\) together with the section \(x_-\) defines a family of smooth
cone situations, degenerating the situation
\parref{cone-situation-examples}\ref{cone-situation-examples.smooth} to
\parref{cone-situation-examples}\ref{cone-situation-examples.N2-}. It is
straightforward to see that the explicit \(q\)-bic threefold \(\mathfrak{X}\)
over \(\mathbf{A}^1\) has the advertised properties. However, to pin down the
dependencies and automorphisms of the situation, it is useful to give an
invariant construction of \(\mathfrak{X}\).

\begin{proof}
Let \((V,\beta)\) be a \(q\)-bic form defining \(X\) and write
\(x_\pm = \PP L_\pm\). The assumption that
\(\langle x_-, x_+ \rangle = \PP U\) is not contained in \(X\) is means
that the restriction \(\beta_U\) of \(\beta\) to \(U\) is nondegenerate. Since
\(U\) is Hermitian, it has a unique orthogonal complement \(W\).

Set \(V[t] \coloneqq V \otimes_\kk \kk[t]\), and similarly for \(U[t]\) and
\(W[t]\). Let \(\beta_{W} \otimes \id_{\kk[t]}\) be the constant extension
of \(\beta_{W}\) to a \(q\)-bic form over \(\kk[t]\) on
\(W[t]\), and let
\[
\beta_U^{L_\pm} \colon U[t]^{[1]} \otimes_{\kk[t]} U[t] \to \kk[t]
\]
be the unique \(q\)-bic form over \(\kk[t]\) with
\(\operatorname{Gram}(\beta_U^{L_\pm}; e_-\otimes 1, e_+\otimes 1) =
\left(\begin{smallmatrix} 0 & 1 \\ t & 0 \end{smallmatrix}\right)\)
where \(e_-\) and \(e_+\) are a basis of \(U\) satisfying
\(\beta_U(e_-^{[1]}, e_+) = \beta_U(e_+^{[1]}, e_-) = 1\). Thus
its restriction to \(t = 1\) is \(\beta_U\), and its restriction to \(t = 0\)
is of type \(\mathbf{N}_2\). Let \(\beta^{L_\pm}\) be \(q\)-bic form on
\(V[t]\) given by the orthogonal sum of \(\beta_{W}[t]\) and
\(\beta_U^{L_\pm}\). Then the \(q\)-bic
\(\mathfrak{X} \subset \PP V \times \mathbf{A}^1\) over
\(\mathbf{A}^1 \coloneqq \Spec\kk[t]\) defined by \((V[t],\beta^{L_\pm})\) is
the desired threefold.
\end{proof}

\subsectiondash{Group scheme actions}\label{smooth-family-groups}
Two group schemes act on the situation \(\mathfrak{X} \to \mathbf{A}^1\) from
\parref{smooth-family}: First, \(\mathfrak{X}\) admits an action over
\(\mathbf{A}^1\) by the automorphism group scheme
\(\AutSch(V[t],\beta^{L_\pm})\) of the \(q\)-bic form over \(\kk[t]\), as
defined in \citeForms{5.1}. The finite flat subgroup scheme \(\mathfrak{G}\)
that respects the orthogonal decomposition
\(\beta^{L_\pm} = \beta_W[t] \oplus \beta_U^{L_\pm}\) and leaves the section
\(x_- \colon \mathbf{A}^1 \to \mathfrak{X}\) invariant furthermore acts on
the family of smooth cone situations, and may be presented as
\[
\mathfrak{G}
\cong
\mathrm{U}_3(q)
\times
\Set{
\begin{pmatrix} \lambda & \epsilon \\ 0 & \lambda^{-q} \end{pmatrix}
\in \mathbf{GL}_2(\kk[t]) :
\lambda \in \boldsymbol{\mu}_{q^2-1},
\epsilon^q + t\lambda^{q-1} \epsilon = 0
}.
\]
Second, consider the \(\mathbf{G}_m\)-action on \(\PP V \times \mathbf{A}^1\)
with weights
\[
\mathrm{wt}(W) = 0,
\quad
\mathrm{wt}(L_-) = -1,
\quad
\mathrm{wt}(L_+) = q,
\quad
\mathrm{wt}(t) = q^2-1.
\]
This leaves the \(q\)-bic form \(\beta^{L_\pm}\) defining \(\mathfrak{X}\)
invariant, and so it induces an action on both \(\mathfrak{X}\) and
\(\mathfrak{G}\) over \(\mathbf{A}^1\) such that the action map
\(\mathfrak{G} \times_{\mathbf{A}^1} \mathfrak{X} \to \mathfrak{X}\) is
\(\mathbf{G}_m\)-equivariant over \(\mathbf{A}^1\).

\subsectiondash{Family of Fano schemes}\label{smooth-family-fano}
Let \(\mathfrak{S} \to \mathbf{A}^1\) be the relative Fano scheme of lines
associated with the family of \(q\)-bic threefolds
\(\mathfrak{X} \to \mathbf{A}^1\). The projective geometry constructions of
\S\S\parref{section-cone-situation}--\parref{section-smooth-cone-situation}
work in families and, applied to the family of smooth cone situations in
\parref{smooth-family}\ref{smooth-family.cone}, yields a commutative diagram
\[
\begin{tikzcd}[column sep=.5em, row sep=.75em]
& \widetilde{\mathfrak{S}} \ar[dr, "\mathrm{P}"] \ar[dl] \\
\mathfrak{S} \ar[dr, "\Phi"'] && \mathfrak{T} \ar[dl, "\Pi"] \\
& C \times {\mathbf{A}^1}
\end{tikzcd}
\]
of morphisms of schemes over \(\mathbf{A}^1\) where \(C\) is the smooth \(q\)-bic
curve in \(\PP W\); \(\mathfrak{T}\) is the degeneracy locus in
\(\PP \times \mathbf{A}^1\) as in \parref{cone-situation-P} and
\parref{cone-situation-equations-of-T}; and \(\widetilde{\mathfrak{S}}\) is
\(q\)-fold covering of \(\mathfrak{T}\) as in
\parref{cone-situation-section-over-T-subbundle}.
The key properties are: \(\widetilde{\mathfrak{S}} \to \mathfrak{S}\) is a
blowup along \(q^3 + 1\) smooth sections of
\(\Phi \colon \mathfrak{S} \to C \times \mathbf{A}^1\) by
\parref{smooth-cone-situation-blowup}; and
\(\mathrm{P} \colon \widetilde{\mathfrak{S}} \to \mathfrak{T}\) is a quotient
by a unipotent group scheme of order \(q\) as in
\parref{smooth-cone-situation-quotient}. In particular, this implies that
there is an isomorphism
\[
\mathbf{R}^1\Phi_*\sO_{\mathfrak{S}} \cong
\mathbf{R}^1\Pi_*
\mathrm{P}_*\sO_{\widetilde{\mathfrak{S}}}
\]
of locally free \(\sO_{C \times \mathbf{A}^1}\)-modules. The unipotent group
quotient induces a \(q\)-step filtration on the right, and so this isomorphism
puts a \(q\)-step filtration on \(\mathbf{R}^1\Phi_*\sO_{\mathfrak{S}}\),
globalizing that from \parref{smooth-cone-situation-pushforward}.

The family \(\Phi \colon \mathfrak{S} \to C \times \mathbf{A}^1\)
relates the cohomology of the singular \(\varphi_0 \colon S_0 \to C\) and
smooth \(\varphi \colon S \to C\) fibres above \(0 \in \mathbf{A}^1\) and
\(1 \in \mathbf{A}^1\), respectively, in a rather subtle way: On the one hand,
recall that \(\mathbf{R}^1\varphi_{0,*}\sO_{S_0}\) is a graded \(\sO_C\)-module
which, as shown in \parref{normalize-splitting}, coincides with the positively
graded parts of the \(\sO_C\)-module \(\mathcal{F}\) introduced in
\parref{nodal-F}. Comparing \parref{nodal-automorphisms} with
\parref{smooth-family-groups} shows that this grading coincides with the
grading induced by the \(\mathbf{G}_m\) action on the family \(\mathfrak{X}\).
On the other hand, taking the fibre at \(1\) of the group scheme \(\mathfrak{G}\)
shows that \(\varphi \colon S \to C\) is equivariant only for the finite
\'etale group scheme \(\boldsymbol{\mu}_{q^2-1}\), meaning that
\(\mathbf{R}^1\varphi_*\sO_S\) admits only a weight decomposition by
\(\mathbf{Z}/(q^2-1)\mathbf{Z}\). That these fit into one family gives the
following relation between the decompositions:

\begin{Proposition}\label{smooth-family-rees}
The choice of \(x_-, x_+ \in X\) as in \parref{smooth-family} induces a
canonical weight decomposition
\[
\mathbf{R}^1\varphi_*\sO_S
= \bigoplus\nolimits_{\alpha \in \mathbf{Z}/(q^2-1)\mathbf{Z}}
(\mathbf{R}^1\varphi_*\sO_S)_\alpha
\]
into subbundles, each of which fits into a short exact sequence
\[
0 \to
\mathcal{F}_\alpha \to
(\mathbf{R}^1\varphi_*\sO_S)_\alpha \to
\mathcal{F}_{\alpha+q^2-1} \to
0.
\]
\end{Proposition}

\begin{proof}
It remains to produce the short exact sequences. Since
\(\Phi \colon \mathfrak{S} \to C \times \mathbf{A}^1\) is equivariant for the
action of \(\mathbf{G}_m\) as described in \parref{smooth-family-groups},
\(\mathbf{R}^1\Phi_*\sO_{\mathfrak{S}}\) is \(\mathbf{G}_m\)-equivariant.
The Rees construction, as in \cite[Lemma 19]{Simpson}, endows the unit fibre
\(\mathbf{R}^1\varphi_*\sO_S\) with a filtration whose graded pieces are weight
components of the central fibre \(\mathbf{R}^1\varphi_{0,*}\sO_{S_0}\) which
successively increase by the weight \(q^2-1\) of the \(\mathbf{G}_m\) action on
the base \(\mathbf{A}^1\). Identifying \(\mathbf{R}^1\varphi_{0,*}\sO_{S_0}\)
with the positively graded parts of \(\mathcal{F}\) as in the proof of
\parref{normalize-splitting}, and noting that the weights appearing in
\(\mathcal{F}_{>0}\) lie in \([1,2q^2-q-2]\) by
\parref{nodal-compute-conductor} and \parref{conductor-dual}, it follows that
this filtration has only two steps, and so reduces to a short exact sequence
\[
0 \to
\bigoplus\nolimits_{\alpha = 1}^{q^2-1}
\mathcal{F}_\alpha \to
\mathbf{R}^1\varphi_*\sO_S \to
\bigoplus\nolimits_{\alpha = 1}^{q^2-q-1}
\mathcal{F}_{\alpha + q^2-1} \to
0.
\]
Finally, since the action of the group scheme \(\mathfrak{G}\) is equivariant
for the \(\mathbf{G}_m\) action, this sequence furthermore respects the
\(\mathbf{Z}/(q^2-1)\mathbf{Z}\) weight decomposition described above, yielding
the short exact sequences in the statement.
\end{proof}

For indices \(\alpha = bq\) with \(1 \leq b \leq q-2\), \parref{D-F-duality}
identifies the quotient in this sequence as
\[
\mathcal{F}_{q^2 + bq - 1} \cong
\Div^{q-2-b}(W) \otimes \sO_C.
\]
When \(q = p\), this makes it easy to show its corresponding exact sequence
is not split:

\begin{Lemma}\label{smooth-family-nonsplit}
If \(q = p\), then for each \(1 \leq b \leq p-2\), the sequence
\[
0 \to
\mathcal{F}_{bp} \to
(\mathbf{R}^1\varphi_*\sO_S)_{bp} \to
\Div^{p-2-b}(W) \otimes \sO_C \to 0
\]
is not split and
\(\mathrm{H}^0(C,\mathcal{F}_{bp}) \cong \mathrm{H}^0(C,(\mathbf{R}^1\varphi_*\sO_S)_{bp})\).
\end{Lemma}

\begin{proof}
That the diagram in \parref{smooth-family-fano} is equivariant for the
\(\mathbf{G}_m\) action together with the isomorphism
\(\mathbf{R}^1\Phi_*\sO_{\mathfrak{S}} \cong \mathbf{R}^1\Pi_*\mathrm{P}_*\sO_{\widetilde{\mathfrak{S}}}\)
means that the sequences in \parref{smooth-family-rees} are compatible with
the \(p\)-step filtrations from \parref{smooth-cone-situation-pushforward}.
Therefore projection to the top graded piece gives a commutative square
\[
\begin{tikzcd}
(\mathbf{R}^1\varphi_*\sO_S)_{bp} \rar \dar &
\Div^{p-2-b}(W) \otimes \sO_C \dar \\
\mathrm{gr}_{p-1}(\mathbf{R}^1\varphi_*\sO_S)_{bp} \rar &
\Div^{p-2-b}(\mathcal{T})
\end{tikzcd}
\]
where the bottom right term is computed as in \parref{cohomology-left}.

Suppose now that the sequence in the statement were split. Global sections
would then lift on the \((p-1)\)-st graded pieces. However, combined with the
Borel--Weil--Bott computation in \parref{cohomology-BWB-vanishing},
\parref{smooth-final-graded} below shows that the bottom left term has no
sections:
\[
\mathrm{H}^0(C,\gr_{p-1}(\mathbf{R}^1\varphi_*\sO_S)_{bp}) =
\mathrm{H}^0(C,\Div^{2q-2-b}(\mathcal{T}) \otimes \sO_C(-1)) = 0.
\]
This gives a contradiction since, either by Borel--Weil--Bott
or simplicity of \(\Div^{p-2-b}(W)\) as a \(\mathrm{U}_3(p)\)-representation,
the right hand map is an isomorphism on global sections. Therefore the sequence
is not split, and the map
\(\mathcal{F}_{bp} \to (\mathbf{R}^1\varphi_*\sO_S)_{bp}\) is an isomorphism on
global sections.
\end{proof}

It remains to determine the top graded piece of the \(bq\) weight component of
\(\varphi_*\sO_S\) with respect to the filtration from
\parref{smooth-cone-situation-pushforward}:

\begin{Lemma}\label{smooth-final-graded}
\(\gr_{q-1}(\mathbf{R}^1\varphi_*\sO_S)_{bq} \cong
\Div^{2q-2-b}(\mathcal{T}) \otimes \sO_C(-1)\) for each \(1 \leq b \leq q-2\).
\end{Lemma}

\begin{proof}
Taking \(i = q-1\) in \parref{smooth-cone-situation-pushforward} shows
\[
\gr_{q-1}(\mathbf{R}^1\varphi_*\sO_S) =
\mathbf{R}^1\pi_*(\gr_{q-1}(\rho_*\sO_{\tilde{S}})) =
\mathbf{R}^1\pi_*(\sO_T(1,-q) \otimes \pi^*\sO_C(-1) \otimes L_-).
\]
By \parref{cone-situation-equations-of-T}, \(\sO_T(1,-q)\) is resolved by a
complex \([\mathcal{E}'_2 \to \mathcal{E}'_1]\) of \(\sO_\PP\)-modules with
\begin{align*}
\mathcal{E}_2'
& = \sO_\PP(-q+1,-2q-1) \otimes \pi^*\sO_C(-1) \oplus \sO_\PP(-q,-2q) \otimes L_+, \;\text{and}\\
\mathcal{E}_1'
& = \sO_\PP(0,-2q) \oplus \sO_\PP(-q+1,-q-1) \oplus \sO_\PP(-q+1,-2q) \otimes \pi^*\sO_C(-1) \otimes L_+.
\end{align*}
The resolution provides a spectral sequence computing
\(\mathbf{R}^1\pi_*\sO_T(1,-q)\) with \(E_1\) page given by
\[
\begin{tikzcd}[row sep=1em, column sep=1em]
E_1^{-2,3} \rar["d_1"] & E_1^{-1,3} \rar["d_1"] & E_1^{0,3} \\
E_1^{-2,2} \rar["d_1"] & E_1^{-1,2} \rar["d_1"] & E_1^{0,2}
\end{tikzcd}
\quad
=
\quad
\begin{tikzcd}[row sep=1em, column sep=1em]
\mathbf{R}^3\pi_*\mathcal{E}_2' \rar["\phi"] &
\mathbf{R}^3\pi_*\mathcal{E}_1' &
0 \\
0 &
\mathbf{R}^2\pi_*\mathcal{E}_1' \rar["\wedge^2\phi^\vee"] &
\mathbf{R}^2\pi_*\sO_\PP(1,-q)
\end{tikzcd}
\]
and with all other terms vanishing. Since \(\boldsymbol{\mu}_{q^2-1}\) acts
through linear automorphisms of \(\PP\) over \(C\), the differentials of the
spectral sequence are compatible with the
\(\mathbf{Z}/(q^2-1)\mathbf{Z}\)-gradings on each term.

Let \(1 \leq b \leq q-2\) and consider the weight \(bq\) components of the
spectral sequence: Recalling from \parref{cone-situation-P} that
\(\PP = \PP\mathcal{V}_1 \times_C \PP\mathcal{V}_2\) with
\(\mathcal{V}_1 \cong \sO_C(-1) \oplus L_{-,C}\) and
\(\mathcal{V}_2 \cong \mathcal{T} \oplus L_{+,C}\), it
follows that the relative dualizing sheaf of the projective bundle factors are
\[
\omega_{\PP\mathcal{V}_1/C} \cong
\sO_{\pi_1}(-2) \otimes \pi_1^*\sO_C(1) \otimes L_-^\vee
\;\;\text{and}\;\;
\omega_{\PP\mathcal{V}_2/C} \cong
\sO_{\pi_2}(-3) \otimes \pi_2^*\sO_C(-1) \otimes L_+^\vee,
\]
and that
\(\omega_{\PP/C} \cong \sO_\PP(-2,-3) \otimes L_-^\vee \otimes L_+^\vee\).
Using this, a direct computation gives
\begin{align*}
(\mathbf{R}^3\pi_*(\mathcal{E}_2' \otimes \pi^*\sO_C(-1) \otimes L_-))_{bq}
& \cong \Div^{2q-2-b}(\mathcal{T}) \otimes \sO_C(-1) \otimes L_-^{\otimes q} \otimes L_+^{\otimes b+1}, \\
(\mathbf{R}^3\pi_*(\mathcal{E}_1' \otimes \pi^*\sO_C(-1) \otimes L_-))_{bq}
& \cong 0, \\
(\mathbf{R}^2\pi_*(\mathcal{E}_1' \otimes \pi^*\sO_C(-1) \otimes L_-))_{bq}
& \cong \Div^{q-2-b}(\mathcal{T}) \otimes L_- \otimes L_+^{\otimes q+b}, \;\text{and} \\
(\mathbf{R}^2\pi_*(\sO_\PP(1,-q) \otimes \pi^*\sO_C(-1) \otimes L_-))_{bq}
& \cong \Div^{q-2-b}(\mathcal{T}) \otimes L_+^{\otimes b}.
\end{align*}
The differential \(\wedge^2\phi^\vee\) between the latter two sheaves is given
by \(u_1 v_{21}' + u_2 v_{22}'\), where \(u_1\) and \(u_2\) are as in
\parref{cone-situation-P}, and \(v_{21}'\) and \(v_{22}'\) are the bottom
components of \(v'\) from \parref{cone-situation-v'}. The section
\(v_{21}'\) contains \(u_2'^q\), where \(u_2'\) is constructed in
\parref{cone-situation-closure} as the coordinate function to the subbundle
\(\mathcal{T} \subset \mathcal{V}_2\). Since the divided powers of
\(\mathcal{T}\) appearing have exponent strictly less than \(q\), multiplication
by \(u_2'^q\) is the zero map. A similar analysis then shows that the
remaining component \(u_2 v_{22}' = u_2^q \cdot \beta_2 \cdot u_1'\) acts via
the isomorphism \(\beta_2 \colon L_- \to L_+^{\vee,\otimes q}\). This implies
that, at least on weight \(bq\) components, the spectral sequence degenerates
on this page and that
\[
\gr_{q-1}(\mathbf{R}^1\varphi_*\sO_S)_{bq}
\cong (\mathbf{R}^3\pi_*(\mathcal{E}_2' \otimes \pi^*\sO_C(-1) \otimes L_-))_{bq}
\cong \Div^{2q-2-b}(\mathcal{T}) \otimes \sO_C(-1).
\qedhere
\]
\end{proof}

\subsectiondash{Proof of Theorem \parref{intro-smooth-cohomology}}\label{smooth-proof}
Since \(\mathrm{H}^1(S,\sO_S)\) is canonically the Lie algebra of
\(\mathbf{Pic}_S\), and, for any prime \(\ell \neq p\),
\(\mathrm{H}^1_{\mathrm{\acute{e}t}}(S,\mathbf{Z}_\ell)\) is
the \(\ell\)-adic Tate module of \(\mathbf{Pic}_S\), there is always an inequality
\[
\dim_\kk\mathrm{H}^1(S,\sO_S) \geq
\dim \mathbf{Pic}_S =
\frac{1}{2} \rank_{\mathbf{Z}_\ell} \mathrm{T}_\ell \mathbf{Pic}_S =
\frac{1}{2} \dim_{\mathbf{Q}_\ell}\mathrm{H}^1_{\mathrm{\acute{e}t}}(S,\mathbf{Q}_\ell).
\]
By the \'etale cohomology computation for \(S\) in
\cite[\href{https://arxiv.org/pdf/2307.06160.pdf\#IntroTheorem.2}{\textbf{Theorem B}}]{fano-schemes}---which
may also be deduced from the Theorem recalled in the Introduction---the
dimension of \(\mathrm{H}^1(S,\sO_S)\) is always at least \(q(q-1)(q^2+1)/2\)
with no assumption on \(q\).

Assume \(q = p\) is prime. The corresponding upper bound follows by
semicontinuity of cohomology, see \citeSP{0BDN}, for the flat family
\(\Phi \colon \mathfrak{S} \to \mathbf{A}^1\) from \parref{smooth-family-fano},
the cohomology computation for the singular surface \(S_0\) from
Theorem \parref{intro-nodal}, and the non-splitting result of
\parref{smooth-family-nonsplit}. In more detail, and slightly more
directly, the Leray spectral sequence for \(\varphi \colon S \to C\) yields
a short exact sequence
\[
0 \to
\mathrm{H}^1(C,\varphi_*\sO_S) \to
\mathrm{H}^1(S,\sO_S) \to
\mathrm{H}^0(C,\mathbf{R}^1\varphi_*\sO_S) \to
0.
\]
Since \(\varphi_*\sO_S = \sO_C\) by \parref{smooth-cone-situation-pushforward},
the first term has dimension \(p(p-1)/2\). For the second term, consider the
\(\mathbf{Z}/(p^2-1)\mathbf{Z}\) weight decomposition from the action of
\(\boldsymbol{\mu}_{p^2-1}\). The short exact sequences in
\parref{smooth-family-rees} yield, for each \(\alpha = 1, 2, \ldots, p^2-1\),
inequalities
\[
\dim_\kk \mathrm{H}^0(C,(\mathbf{R}^1\varphi_*\sO_S)_\alpha) \leq
\dim_\kk\mathrm{H}^0(C,\mathcal{F}_\alpha) +
\dim_\kk\mathrm{H}^0(C,\mathcal{F}_{\alpha + p^2-1}).
\]
When \(\alpha = bp\) with \(1 \leq b \leq p-2\),
\parref{smooth-family-nonsplit} refines this to an equality
\[
\dim_\kk \mathrm{H}^0(C,(\mathbf{R}^1\varphi_*\sO_S)_\alpha)=
\dim_\kk\mathrm{H}^0(C,\mathcal{F}_{bp}).
\]
Summing these over \(\alpha\) gives the inequality
\[
\dim_\kk\mathrm{H}^0(C,\mathbf{R}^1\varphi_*\sO_S) \leq
\dim_\kk\mathrm{H}^0(C,\mathcal{F})
- \dim_\kk\mathrm{H}^0(C,\mathcal{F}_0)
- \sum\nolimits_{b = 1}^{p-2} \dim_\kk\mathrm{H}^0(C,\mathcal{F}_{p^2+bp-1}).
\]

Consider the negative terms on the right. First, \parref{D-F-duality} shows
that \(\mathcal{F}_{p^2+bp-1} \cong \Div^{p-2-b}(W) \otimes \sO_C\). Second,
\parref{normalize-splitting} implies that \(\mathcal{F}_0\) is the cokernel
of the map \(\sO_C \to \phi_{C,*}\sO_C\) which, up to an automorphism, is
the \(p^2\)-power Frobenius morphism. Since the \(p\)-power Frobenius already
acts by zero on \(\mathrm{H}^1(C,\sO_C)\) by
\parref{generalities-frobenius-vanishes}, the long exact sequence in cohomology
shows
\[
\mathrm{H}^0(C,\mathcal{F}_0)
\cong \mathrm{H}^1(C,\sO_C)
\cong \Div^{p-2}(W).
\]
Therefore the negative terms in the inequality sum up to
\begin{align*}
\dim_\kk\mathrm{H}^0(C,\mathcal{F}_0) +
\sum\nolimits_{b = 1}^{p-2} \dim_\kk\mathrm{H}^0(C,\mathcal{F}_{p^2+bp-1})
& = \sum\nolimits_{b = 0}^{p-2} \dim_\kk \Div^b(W) \\
& = \sum\nolimits_{b = 0}^{p-2} \binom{b+2}{2}
= \binom{p+1}{3}
= \binom{p}{2} + \binom{p}{3}.
\end{align*}
Combining this with \parref{cohomology-theorem} then shows that
\[
\dim_\kk\mathrm{H}^0(C,\mathbf{R}^1\varphi_*\sO_S) \leq
(p^2+1)\binom{p}{2} + \binom{p}{3} - \binom{p}{2} - \binom{p}{3} =
p^2\binom{p}{2}.
\]
The short exact sequence for \(\mathrm{H}^1(S,\sO_S)\) then gives
\[
\dim_\kk\mathrm{H}^1(S,\sO_S) \leq
\binom{p}{2} + p^2\binom{p}{2} = \frac{1}{2}p(p-1)(p^2+1),
\]
completing the proof.
\qed

\appendix

\section{Representation theory computations}\label{section-representations}
This appendix collects some facts and computations pertaining to the modular
representation theory of the algebraic group \(\mathbf{SL}_3\) and the finite
unitary group \(\mathrm{SU}_3(q)\), acting linearly via automorphisms on a
\(3\)-dimensional \(\kk\)-vector space \(W\).

\subsectiondash{Root data}\label{representations-root-data}
Choose a maximal torus and Borel subgroup
\(\mathbf{T} \subset \mathbf{B} \subset \mathbf{SL}_3\), and let
\begin{align*}
\mathrm{X}(\mathbf{T}) & \coloneqq
\Hom(\mathbf{T},\mathbf{G}_m) \cong
\mathbf{Z}\{\epsilon_1, \epsilon_2, \epsilon_3\}/(\epsilon_1 + \epsilon_2 + \epsilon_3) \\
\mathrm{X}^\vee(\mathbf{T}) & \coloneqq \Hom(\mathbf{G}_m,\mathbf{T})
\cong \Set{a_1 \epsilon_1^\vee + a_2 \epsilon_2^\vee + a_3 \epsilon_3^\vee \in \mathbf{Z}\{\epsilon_1^\vee, \epsilon_2^\vee, \epsilon_3^\vee\} : a_1 + a_2 + a_3 = 0}
\end{align*}
be the lattices of characters and cocharacters of \(\mathbf{T}\); here, upon
conjugating \(\mathbf{T}\) to the diagonal matrices in \(\mathbf{SL}_3\), the
characters \(\epsilon_i\) extract the \(i\)-th diagonal entry, whereas the
cocharacters \(\epsilon_i^\vee\) include into the \(i\)-th diagonal entry. Let
\(
\langle -,- \rangle \colon
\mathrm{X}(\mathbf{T}) \times \mathrm{X}^\vee(\mathbf{T}) \to
\Hom(\mathbf{G}_m, \mathbf{G}_m) \cong \mathbf{Z}
\)
be the natural root pairing, so that
\(\langle \epsilon_i, \epsilon_j^\vee \rangle = \delta_{ij}\). The
simple roots, simple coroots, and positive roots corresponding to \(\mathbf{B}\)
are
\[
\alpha_1 \coloneqq \epsilon_1 - \epsilon_2,\;\;
\alpha_2 \coloneqq \epsilon_2 - \epsilon_3,
\quad
\alpha_1^\vee \coloneqq \epsilon_1^\vee - \epsilon_2^\vee,\;\;
\alpha_2^\vee \coloneqq \epsilon_2^\vee - \epsilon_3^\vee,
\quad
\Phi^+ \coloneqq \set{\alpha_1, \alpha_2, \alpha_1 + \alpha_2}.
\]
The fundamental weights are
then \(\varpi_1 \coloneqq \epsilon_1\) and
\(\varpi_2 \coloneqq \epsilon_1 + \epsilon_2\); the half sum of all the
positive roots is then \(\rho = \varpi_1 + \varpi_2\); and the set of
dominant weights is
\[
\mathrm{X}_+(\mathbf{T}) =
\Set{a \varpi_1 + b \varpi_2 \in \mathrm{X}(\mathbf{T}) :
a,b \in \mathbf{Z}_{\geq 0}}.
\]
Highest weight theory puts the set of simple representations of
\(\mathbf{SL}_3\) in bijection with the set of dominant weights
\(\mathrm{X}_+(\mathbf{T})\); write \(L(a,b)\) for the simple module with
highest weight \(a\varpi_1 + b\varpi_2\).

\subsectiondash{Flag varieties}\label{representations-flag-varieties}
Let \(\Flag W \cong \mathbf{SL}_3/\mathbf{B}\) be the full flag
variety of \(W\). This is the \((1,1)\)-divisor in \(\PP W \times \PP W^\vee\)
cut out by the trace section, and so
\[
\Pic(\Flag W) =
\big\{
  \sO_{\Flag W}(a,b) \coloneqq
  \sO_{\PP W}(a) \boxtimes \sO_{\PP W^\vee}(b)\rvert_{\Flag W} :
  a,b \in \mathbf{Z}
\big\}.
\]
The map
\(a \varpi_1 +  b \varpi_2 \mapsto \sO_{\Flag W}(a,b)\) gives an
isomorphism \(\mathrm{X}(\mathbf{T}) \to \Pic(\Flag W)\) of abelian
groups.
Following the conventions of \cite[II.2.13(1)]{Jantzen:RAGS}, the
\emph{Weyl module} corresponding to a dominant weight
\(a \varpi_1 + b \varpi_2 \in \mathrm{X}_+(\mathbf{T})\) is
\[
\Delta(a,b)
\coloneqq \mathrm{H}^0(\Flag W, \sO_{\Flag W}(b,a))^\vee
\cong \mathrm{H}^0(\PP W, \Sym^a(\mathcal{T}_{\PP W}(-1)) \otimes \sO_{\PP W}(b))^\vee
\]
where the isomorphism follows from the simple computation that, for \(a,b \in
\mathbf{Z}\),
\[
\pr_{\PP W,*}(\sO_{\Flag W}(a,b)) =
\begin{dcases*}
\Sym^b(\mathcal{T}_{\PP W}(-1)) \otimes \sO_{\PP W}(a) & if \(b \geq 0\), \\
0 & if \(b < 0\).
\end{dcases*}
\]
For example, \(\Delta(a,0) = \Sym^a(W)^\vee\) and
\(\Delta(0,b) = \Div^b(W) \cong \Sym^b(W^\vee)^\vee\)
are the spaces of \(a\)-th symmetric and \(b\)-th divided powers, respectively.

Cohomology of line bundles on \(\Flag W\) is classically determined
via the Borel--Weil--Bott Theorem of \cite{Bott, Demazure}. This, however, is
rather subtle in positive characteristic, see \cite[II.5.5]{Jantzen:RAGS}. In
general, Kempf's Theorem \cite[Theorem 1 on p.586]{Kempf} shows that higher
cohomology always vanishes when the corresponding weight \(\lambda\) is
dominant. In the present case of \(\mathbf{SL}_3\), Griffith
gave a complete answer in \cite[Theorem 1.3]{Griffith:BWB}. Using this,
a straightforward computation gives:

\begin{Lemma}\label{representations-BWB-PP2}
Let \(0 \leq b \leq p - 1\). Then
\begin{enumerate}
\item\label{representations-BWB-PP2.H0}
\(\mathrm{H}^0(\PP^2, \Sym^b(\mathcal{T}_{\PP^2}(-1))(a)) = 0\)
whenever \(a < 0\), and
\item\label{representations-BWB-PP2.H1}
\(\mathrm{H}^1(\PP^2, \Sym^b(\mathcal{T}_{\PP^2}(-1))(a)) = 0\)
whenever \(a < p\). \qed
\end{enumerate}
\end{Lemma}

This leads to the following computation for smooth plane curves \(C \subset \PP^2\)
of degree \(p+1\):

\begin{Corollary}\label{cohomology-BWB-vanishing}
For integers \(0 \leq b \leq p - 1\) and \(a \leq 0\),
\[
\mathrm{H}^0(C,\Sym^b(\mathcal{T}_{\PP W}(-1))(a)\rvert_C) =
\begin{dcases*}
\Sym^b(W) & if \(a = 0\), and \\
0 & if \(a < 0\).
\end{dcases*}
\]
\end{Corollary}

\begin{proof}
The restriction sequence
\[
0 \to
\Sym^b(\mathcal{T}_{\PP W}(-1))(a - p - 1) \to
\Sym^b(\mathcal{T}_{\PP W}(-1))(a) \to
\Sym^b(\mathcal{T}_{\PP W}(-1))(a)\rvert_C \to
0
\]
implies it suffices to show that
\(\mathrm{H}^0(\PP W,\Sym^b(\mathcal{T}_{\PP W}(-1))) = \Sym^b(W)\)
and
\[
\mathrm{H}^0(\PP W,\Sym^b(\mathcal{T}_{\PP W}(-1))(a)) =
\mathrm{H}^1(\PP W,\Sym^b(\mathcal{T}_{\PP W}(-1))(a-p)) = 0\;
\;\;\text{when}\; a < 0.
\]
The identification of global sections follows the Euler sequence;
since \(0 \leq b \leq p - 1\), the vanishing follows from the Borel--Weil--Bott
Theorem \emph{\'a la} Griffith, see \parref{representations-BWB-PP2}.
\end{proof}

\subsectiondash{Jantzen filtration and sum formula}\label{representations-filtration}
The Weyl modules \(\Delta(\lambda)\) from \parref{representations-flag-varieties}
are generally not irreducible in positive characteristic. Their simple composition
factors can sometimes be described using Jantzen's filtration and sum formula,
as described in \cite[II.8.19]{Jantzen:RAGS}. In the situation at hand, this
means the following: given a dominant weight
\(\lambda \in \mathrm{X}_+(\mathbf{T})\), there is a decreasing filtration
\[
\Delta(\lambda) =
\Delta(\lambda)^0 \supseteq
\Delta(\lambda)^1 \supseteq
\Delta(\lambda)^2 \supseteq \cdots
\]
such that \(L(\lambda) = \Delta(\lambda)/\Delta(\lambda)^1\). Furthermore,
there is the \emph{sum formula}:
\[
\sum\nolimits_{i > 0} \ch(\Delta(\lambda)^i) =
\sum\nolimits_{\alpha \in \Phi^+} \sum\nolimits_{m : 0 < mp < \langle \lambda + \rho, \alpha^\vee\rangle}
\nu_p(mp) \chi(s_{\alpha,mp} \cdot \lambda)
\]
where \(\ch\) extracts the \(\mathbf{T}\)-character of a module,
\(\nu_p \colon \mathbf{Z} \to \mathbf{Z}\) is the \(p\)-adic valuation,
\(s_{\alpha,mp}\) is the affine reflection
\(\lambda \mapsto \lambda + (mp - \langle \lambda, \alpha^\vee \rangle) \alpha\)
on \(\mathrm{X}(\mathbf{T})\),
\(s_{\alpha,mp} \cdot \lambda \coloneqq s_{\alpha,mp}(\lambda + \rho) - \rho\)
is the dot action, and
\[
\chi(\lambda) \coloneqq
\sum\nolimits_{i \geq 0}
(-1)^i \big[\mathrm{H}^i(\Flag W, \sO_{\Flag W}(\lambda))\big]
\]
is the Euler characteristic of the line bundle corresponding to \(\lambda\)
with values in the representation ring of \(\mathbf{T}\). As a simple
application, consider a weight \(\lambda = a\varpi_1 + b\varpi_2\) in which all
the root pairings
\[
\langle \lambda + \rho, \alpha_1^\vee \rangle = a + 1, \quad
\langle \lambda + \rho, \alpha_2^\vee \rangle = b + 1, \quad
\langle \lambda + \rho, \alpha_1^\vee + \alpha_2^\vee \rangle = a + b + 2
\]
are at most \(p\). Then the right hand side of the sum formula
is empty, implying the following:

\begin{Lemma}\label{representations-easy-simples}
If \(a,b \in \mathbf{Z}_{\geq 0}\) satisfy \(a + b \leq p - 2\), then
\(\Delta(a,b)\) is simple. \qed
\end{Lemma}

The next two statements describe the structure of the Weyl modules of
weights \(b\varpi_2\) and \(\varpi_1 + b\varpi_2\):

\begin{Lemma}\label{representations-divs}
The Weyl module \(\Delta(0,b) = \Div^b(V)\) of highest weight
\(b\varpi_2\) satisfies:
\begin{enumerate}
\item\label{representations-divs.simple}
If \(0 \leq b \leq p-1\), then \(L(0,b) = \Delta(0,b)\) is simple.
\item\label{representations-divs.not-simple}
If \(p \leq b \leq 2p-3\), then \(L(0,b) = W^{[1]} \otimes \Div^{b-p}(W)\) and
there is a short exact sequence
\[ 0 \to L(b-p+1,2p-2-b) \to \Delta(0,b) \to L(0,b) \to 0. \]
\end{enumerate}
\end{Lemma}

\begin{proof}
That \(\Delta(0,b)\) is simple when \(0 \leq b \leq p-2\) follows from
\parref{representations-easy-simples}. When \(b \geq p-1\), the following
term appears in the sum formula:
\begin{align*}
\chi(s_{\alpha_1+\alpha_2,p} \cdot b\varpi_2)
& = \chi((p-b-2)\varpi_1 + (p-2)\varpi_2)  \\
& = \chi(s_{\alpha_1} \cdot ((b-p)\varpi_1 + (2p-3-b)\varpi_2))
= -\chi((b-p)\varpi_1 + (2p-3-b)\varpi_2)
\end{align*}
with the final inequality is due to \cite[II.5.9]{Jantzen:RAGS}. This, in
particular, vanishes when \(b = p-1\) because the line bundle
\(\sO_{\Flag W}(-1,a)\) never has cohomology, giving
\ref{representations-divs.simple}.

If \(p \leq b \leq 2p-3\), the identification of \(L(0,b)\) follows from
\ref{representations-divs.simple} together with the Steinberg Tensor Product
Theorem, \cite[II.3.17]{Jantzen:RAGS}. For the short exact sequence, note
that the sum formula has two terms, indexed by
\((\alpha_2,p)\) and \((\alpha_1+\alpha_2,p)\). The latter is as above, and
the former is given by
\begin{align*}
\chi(s_{\alpha_2,p} \cdot b\varpi_2)
& = \chi((b-p+1)\varpi_1 + (2p-2-b)\varpi_2)
= \ch(\Delta(b-p+1,2p-2-b)) \\
& = \ch(L(b-p+1,2p-2-b)) + \ch(L(b-p,2p-3-b)).
\end{align*}
where the final equality arises from the sum formula applied to this Weyl
module. Putting this together with the calculation above shows that
\[
\sum\nolimits_{i > 0} \ch(\Delta(0,b)^i)
= \chi(s_{\alpha_2,p} \cdot b\varpi_2) + \chi(s_{\alpha_1+\alpha_2,p} \cdot b\varpi_2)
= \ch(L(b-p+1,2p-2-b))
\]
which means \(\Delta(0,b)^1 = L(b-p+1,2p-2-b)\), whence the
exact sequence in \ref{representations-divs.not-simple}.
\end{proof}

\begin{Lemma}\label{representations-1-b}
The Weyl module
\(\Delta(1,b) \cong
\ker(\operatorname{ev} \colon W^\vee \otimes \Div^b(W) \to \Div^{b-1}(W))\)
is simple if \(0 \leq b \leq p-3\), and, if \(b = p-2\), it fits into a short
exact sequence
\[ 0 \to L(0,p-3) \to \Delta(1,p-2) \to L(1,p-2) \to 0. \]
\end{Lemma}

\begin{proof}
Simplicity when \(0 \leq b \leq p-3\)  follows from
\parref{representations-easy-simples}. When \(b = p-2\), the sum formula reads
\[
\sum\nolimits_{i > 0} \ch(\Delta(1,p-2)^i) =
\chi(s_{\alpha_1+\alpha_2,p} \cdot (\varpi_1 + (p-2)\varpi_2)) =
\ch(L(0,p-3)).
\]
Thus \(L(0,p-3)\) is the only composition factor in \(\Delta(1,p-2)^1\),
giving the result.
\end{proof}

Let \(\mathrm{SU}_3(p)\) be the \'etale subgroup scheme of \(\mathbf{SL}_3\)
which preserves a nondegenerate \(q\)-bic form \((W,\beta)\). Steinberg's
Restriction Theorem \cite{Steinberg}, see also \cite[Theorem 2.11]{Humphreys},
implies that the irreducible representations of \(\mathrm{SU}_3(p)\) arise via
restriction from \(\mathbf{SL}_3\). Namely:

\begin{Theorem}\label{representations-steinberg-restriction}
Let \(0 \leq a,b \leq p-1\). The restriction of the
\(\mathbf{SL}_3\)-modules \(L(a,b)\) to \(\mathrm{SU}_3(p)\) remain
simple, are pairwise nonisomorphic, and give all isomorphism classes of
simple \(\mathrm{SU}_3(p)\)-modules. \qed
\end{Theorem}

Abusing notation, write \(L(a,b)\) and \(\Delta(a,b)\) for the
\(\mathrm{SU}_3(p)\)-modules obtained via restriction of the corresponding
\(\mathbf{SL}_3\)-modules. The next statement gives an alternate construction
of the Weyl module \(\Delta(1,b)\) from \parref{representations-1-b} as a
\(\mathrm{SU}_3(p)\)-module:

\begin{Lemma}\label{representations-F-weyl}
For each \(0 \leq b \leq p - 1\), there is an isomorphism of
\(\mathrm{SU}_3(p)\)-modules:
\[
\Delta(1,b) \cong
\ker(f \colon W^{[1]} \otimes \Div^b(W) \xrightarrow{\beta} W^\vee \otimes \Div^b(W) \xrightarrow{\mathrm{ev}} \Div^{b-1}(W)).
\]
\end{Lemma}

\begin{proof}
By construction, there is a \(\mathrm{SU}_3(p)\)-equivariant commutative
diagram
\[
\begin{tikzcd}
W^{[1]} \otimes \Div^b(W) \dar["\beta^\vee"'] \rar["f"']
& \Div^{b-1}(W) \dar[equal] \\
V^\vee \otimes \Div^b(W) \rar["\mathrm{ev}"]
& \Div^{b-1}(W)\punct{.}
\end{tikzcd}
\]
Thus the kernels of the two rows are isomorphic as representations of
\(\mathrm{SU}_3(p)\). Since the bottom kernel is \(\Delta(1,b)\) from
\parref{representations-1-b}, the result follows.
\end{proof}

Let \(C \subset \PP^2\) be the smooth \(q\)-bic curve associated with the
nonsingular \(q\)-bic form \((W,\beta)\). Then \(\mathrm{SU}_3(p)\) acts
through linear automorphisms on \(C\), and so acts on various cohomology groups
of \(C\). The following identifies a few of these representations that are
particularly useful in \S\parref{section-cohomology-F}:

\begin{Lemma}\label{cohomology-twist}
The \(\mathrm{SU}_3(p)\) representations
\(\mathrm{H}^1(C,\sO_C(-i))\) for \(0 \leq i \leq p\) are:
\begin{enumerate}
\item\label{cohomology-twist.low}
If \(0 \leq i \leq 1\), then
\(\mathrm{H}^1(C,\sO_C(-i)) \cong \Div^{p+i-2}(W)\) is simple.
\item\label{cohomology-twist.sequence}
If \(2 \leq i \leq p\), then there is a short exact sequence
\[
0 \to
\Div^{p+i-2}_{\mathrm{red}}(W) \to
\mathrm{H}^1(C,\sO_C(-i)) \to
\Delta(1,i-2) \to
0.
\]
\item\label{cohomology-twist.simple}
If \(i \neq p\), then the quotient \(\Delta(1,i-2) = L(1,i-2)\) is simple.
\item\label{cohomology-twist.not-simple}
If \(i = p\), then
\(\Delta(1,p-2) \cong \mathrm{H}^0(C,\Omega^1_{\PP W}(1)\rvert_C)\)
and a short exact sequence
\[
0 \to
\Div^{p-3}(W) \to
\mathrm{H}^0(C,\Omega^1_{\PP W}(1)\rvert_C) \to
L(1,p-2) \to
0.
\]
\end{enumerate}
\end{Lemma}

\begin{proof}
Cohomology of the ideal sheaf sequence of \(C\) in \(\PP W\) twisted by
\(\sO_{\PP W}(-i)\) shows that
\[
\mathrm{H}^1(C,\sO_C(-i)) \cong
\ker(\cdot f \colon \Div^{p+i-2}(W) \to \Div^{i-3}(W)),
\]
with which \ref{cohomology-twist.low} then follows from
\parref{representations-divs}\ref{representations-divs.simple}. For
\(2 \leq i \leq p\), consider the commutative diagram
\[
\begin{tikzcd}[row sep=1.5em]
0 \rar
& \Div^{i+p-2}_{\mathrm{red}}(W) \rar
& \Div^{i+p-2}(W) \rar \dar["f"']
& W^{[1]} \otimes \Div^{i-2}(W) \rar \dar["f"]
& 0  \\
&& \Div^{i-3}(W) \rar[equal]
& \Div^{i-3}(W)
\end{tikzcd}
\]
in which the top row is exact and vertical maps surjective.
Since, by \parref{representations-F-weyl},
\[
\Delta(1,i-2) \cong
\ker(\cdot f \colon W^{[1]} \otimes \Div^{i-2}(W) \to \Div^{i-3}(W)),
\]
taking kernels of the vertical maps thus yields the exact sequence of
\ref{cohomology-twist.sequence}. When \(2 \leq i \leq p-1\),
the first statement of \parref{representations-1-b} shows that \(\Delta(1,i-2)
= L(1,i-2)\) is simple, proving \ref{cohomology-twist.simple}. When
\(i = p\), the Euler sequence yields an isomorphism
\[
\mathrm{H}^1(C,\Omega^1_{\PP W}(1)\rvert_C) \cong
\ker(W^\vee \otimes \mathrm{H}^1(C,\sO_C) \to \mathrm{H}^1(C,\sO_C(-1)))
\cong \Delta(1,p-2),
\]
upon which the second statement of \parref{representations-1-b} yields
the exact sequence of \ref{cohomology-twist.not-simple}.
\end{proof}

\bibliographystyle{amsalpha}
\bibliography{main}
\end{document}